\documentclass [10pt] {article}
\usepackage{color}
\providecommand{\keywords}[1]{\textbf{Keywords:} #1}
\usepackage{amsfonts}
\usepackage{amsmath}
\usepackage{amssymb}
\usepackage{hyperref}
\usepackage{bm}
\usepackage[alphabetic]{amsrefs}
\usepackage{fullpage}
\usepackage{theorem}
\usepackage{tensor}
\usepackage{etex}
\usepackage[all]{xy}

\newtheorem{thm}{Theorem}[section]
\newtheorem{cor}[thm]{Corollary}
\newtheorem{lem}[thm]{Lemma}
\newtheorem{prop}[thm]{Proposition}
\newtheorem{prop-defn}[thm]{Proposition/Definition}

\newtheorem{conjec}[thm]{Conjecture}

\theorembodyfont{\rmfamily}
\newtheorem{defn}[thm]{Definition}
\newtheorem{exa}[thm]{Example}

\newtheorem{nota}[thm]{Notation}
\newtheorem{ass}[thm]{Assumptions}
\newtheorem{rem}[thm]{Remark}

\numberwithin{equation}{section}
\newenvironment{proof}{\noindent \emph{Proof.}}{\hspace{\stretch{1}}$\Box$}

\newcommand{\mcC} {\mathcal{C}}

\newcommand{\mcM} {\mathcal{M}}
\newcommand{\mcN} {\mathcal{N}}

\newcommand{\mcS} {\mathcal{S}}

\newcommand{\ii} {\mathrm{i}}
\newcommand{\ee} {\mathrm{e}}

\newcommand{\ind} {\indices}

\newcommand{\lp} [1] {{\left( #1 \right. }}
\newcommand{\rp} [1] {{\left. #1 \right) }}
\newcommand{\lb} [1] {{\left[ #1 \right. }}
\newcommand{\rb} [1] {{\left. #1 \right] }}

\newcommand{\Rho} {\mathrm{P}}

\newcommand{\im} {\mathop{\mathrm{im}}}
\newcommand{\End} {\mathrm{End}}

\newcommand{\gr} {\mathrm{gr}}

\newcommand{\Spin} {\mathrm{Spin}}
\newcommand{\SL} {\mathrm{SL}}
\newcommand{\GL} {\mathrm{GL}}

\newcommand{\SO} {\mathrm{SO}}

\newcommand{\OO} {\mathrm{O}}

\newcommand{\Gr} {\mathrm{Gr}}

\newcommand{\G} {\mathrm{G}}
\newcommand{\Tgt} {\mathrm{T}}

\newcommand{\Id} {\mathrm{Id}}

\newcommand{\so} {\mathfrak{so}}

\newcommand{\slie} {\mathfrak{sl}}
\newcommand{\glie} {\mathfrak{gl}}
\newcommand{\g} {\mathfrak{g}}

\newcommand{\prb} {\mathfrak{p}}

\newcommand{\mfz} {\mathfrak{z}}

\newcommand{\mfC} {\mathfrak{C}}

\newcommand{\mfN} {\mathfrak{N}}
\newcommand{\mfV} {\mathfrak{V}}

\newcommand{\mfM} {\mathfrak{M}}

\newcommand{\mfS} {\mathfrak{S}}

\newcommand{\mfA} {\mathfrak{A}}

\newcommand{\mfF} {\mathfrak{F}}
\newcommand{\mfI} {\mathfrak{I}}

\newcommand{\mfW} {\mathfrak{W}}

\newcommand{\Pp} {\mathbb{P}}

\newcommand{\CP} {\mathbb{CP}}

\newcommand{\C} {\mathbb{C}}
\newcommand{\Z} {\mathbb{Z}}

\newcommand{\Cl} {\mathcal{C}\ell}

\newcounter{mnotecount}[section]
\renewcommand{\themnotecount}{\thesection.\arabic{mnotecount}}

\newcommand{\mnote}[1]
{\protect{\stepcounter{mnotecount}}$^{\mbox{\footnotesize
$
\bullet$\themnotecount}}$ \marginpar{\color{red}
\raggedright\tiny\em
$\!\!\!\!\!\!\,\bullet$\themnotecount: #1} }

\begin{document}

\title{Pure spinors, intrinsic torsion and curvature in even dimensions}
\author{Arman Taghavi-Chabert\\
{\small Masaryk University, Faculty of Science, Department of Mathematics and Statistics,}\\
 {\small Kotl\'{a}\v{r}sk\'{a} 2, 611 37 Brno, Czech Republic } }
\date{}

\maketitle

\begin{abstract}
We study the geometric properties of a $2m$-dimensional complex manifold $\mathcal{M}$ admitting a holomorphic reduction of the frame bundle to the structure group $P \subset \mathrm{Spin}(2m,\mathbb{C})$, the stabiliser of the line spanned by a pure spinor at a point. Geometrically, $\mathcal{M}$ is endowed with a holomorphic metric $g$, a holomorphic volume form, a spin structure compatible with $g$, and a holomorphic pure spinor field $\xi$ up to scale. The defining property of $\xi$ is that it determines an almost null structure, ie an $m$-plane distribution $\mathcal{N}_\xi$ along which $g$ is totally degenerate.

We develop a spinor calculus, by means of which we encode the geometric properties of $\mcN_\xi$ corresponding to the algebraic properties of the intrinsic torsion of the $P$-structure. This is the failure of the Levi-Civita connection $\nabla$ of $g$ to be compatible with the $P$-structure. In a similar way, we examine the algebraic properties of the curvature of $\nabla$.

Applications to spinorial differential equations are given. In particular, we give necessary and sufficient conditions for the almost null structure associated to a pure conformal Killing spinor to be integrable. We also conjecture a Goldberg-Sachs-type theorem on the existence of a certain class of almost null structures when $(\mathcal{M},g)$ has prescribed curvature.

We discuss applications of this work to the study of real pseudo-Riemannian manifolds.
\end{abstract}

\keywords{complex Riemannian geometry; pure spinors; distributions; intrinsic torsion; curvature prescription; spinorial equations}

\section{Introduction}
Let $\mcM$ be a complex manifold of dimension $n$, and denote by $\Tgt \mcM$ and $\Tgt^* \mcM$ its holomorphic tangent and cotangent bundles respectively, and by $\mathrm{F} \mcM$ its holomorphic  frame bundle. Following \cite{LeBrun1983}, we define a \emph{holomorphic metric} on $\mcM$ to be a non-degenerate holomorphic section $g$ of the bundle $\odot^2 \Tgt^* \mcM$ --- here $\odot$ denotes the symmetric tensor product. We identify $\Tgt \mcM$ and $\Tgt^* \mcM$ by means of $g$. The pair $(\mcM,g)$ will be referred to as a \emph{complex Riemannian manifold}, and is characterised equivalently by a holomorphic reduction of the structure group of $\mathrm{F} \mcM$ to the complex orthogonal group $\OO(n,\C)$. Analogously to real pseudo-Riemannian geometry, there is a unique torsion-free holomorphic affine connection $\nabla$ preserving $g$, also referred to as the \emph{Levi-Civita connection} of $g$, with associated curvature tensors, which depend holomorphically on $\mcM$. We shall also assume the existence of a global \emph{holomorphic volume form} $\varepsilon \in \Gamma( \wedge^{n} \Tgt^* \mcM )$ normalised to $g(\varepsilon,\varepsilon) = n!$ --- here, we have extended $g$ to a non-degenerate bilinear form on the bundle $\wedge^\bullet \Tgt \mcM$ of holomorphic differential forms, and its dual. This  induces a further holomorphic reduction of the structure group of $\mathrm{F} \mcM$ to the complex special orthogonal group $\SO(n,\C)$. The pair $(g,\varepsilon)$ can be used to define a holomorphic Hodge duality operator $\star$ on $\wedge^{\bullet} \Tgt^* \mcM$. We shall henceforth assume $n=2m$. Then $\star$ squares to plus or minus the identity on $\wedge^m \Tgt^* \mcM$, and thus splits $\wedge^m \Tgt^* \mcM$ as a direct sum of the two eigensubbundles $\wedge^m_\pm \Tgt^* \mcM$ of $\star$. Elements of $\wedge^m_\pm \Tgt^* \mcM$ are referred to as holomorphic \emph{self-dual} and \emph{anti-self-dual} $m$-forms.

This article is concerned with the local geometric properties of an \emph{almost null structure} on $(\mcM,g)$, i.e.\ a holomorphic rank-$m$ distribution $\mcN \subset \Tgt \mcM$ totally null with respect to $g$, i.e $g(v,w)=0$ for all $v$ and $w$ in $\mcN_p$, and $\dim \mcN_p = m$ at any point $p$ of $\mcM$. Being determined (ie annihilated) by a holomorphic $m$-form, an almost null structure may be either self-dual or anti-self-dual, and is also referred to as an \emph{$\alpha$-plane} or \emph{$\beta$-plane distribution} accordingly.

There is a slick way to describe an almost null structure if we assume in addition $(\mcM,g)$ to be \emph{spin}, i.e.\ it admits a holomorphic reduction to $\Spin(2m,\C)$, the two-fold covering of $\SO(2m,\C)$. In this case, $(\mcM,g)$ is endowed with two irreducible spinor bundles $\mcS^+$ and $\mcS^-$. Sections of $\Tgt \mcM$ acts on sections of $\mcS^\pm$ via Clifford multiplication $\cdot : \Tgt \mcM \times \mcS^\pm \rightarrow \mcS^\mp$. In particular, a holomorphic section $\xi$ of $\mcS^+$ or $\mcS^-$ determines a distribution $\mcN_\xi$ on $\mcM$ in the sense that
\begin{align*}
(\mcN_\xi )_p& := \left\{ v \in \Tgt_p \mcM : v \cdot \xi_p \right\} \, , &  \mbox{at any point $p$ in $\mcM$.}
\end{align*}
The defining property of the Clifford multiplication tells us that $\mcN_\xi$ is totally null. When $\mcN_\xi$ has dimension $m$ at every point, $\xi$ is said to be \emph{pure}. If we refer to a pure spinor $\xi$ defined \emph{up to scale} as a \emph{projective pure spinor} $[\xi]$, it is clear that a projective pure spinor field $[\xi]$ determines a unique almost null structure $\mcN_\xi$. Conversely, any almost null structure arises in this way. Whether $\xi$ lies in $\mcS^+$ or $\mcS^-$ corresponds to whether $\mcN_\xi$ is self-dual or anti-self-dual. All spinors in $\mcS^\pm$ are pure in dimensions two, four and six, but when $m>3$, the property of being pure imposes non-trivial algebraic conditions on the components of a spinor. 

The geometric properties of an almost null structure $\mcN_\xi$ associated to a projective pure spinor $[\xi]$ can be expressed in terms of the covariant derivative of $[\xi]$. For instance, if $\mcN_\xi$ is integrable, i.e.\ $[\Gamma(\mcN_\xi), \Gamma(\mcN_\xi)] \subset \Gamma(\mcN_\xi)$, then one can show that the leaves of its foliation are totally geodetic, i.e.\ $\nabla_X Y \in \Gamma (\mcN_\xi)$ for any holomorphic sections $X$, $Y$ of $\mcN_\xi$. This condition can also be expressed as \cite{Hughston1988}
\begin{align}\label{eq-totgeod}
\nabla_X \xi & = \lambda_X \xi \, , & \mbox{for any $X \in \Gamma (\mcN_\xi)$, and some holomorphic function $\lambda_X$ dependent on $X$,}
\end{align}
where, with a slight abuse of notation, $\nabla$ denotes the spin connection induced from the Levi-Civita connection. Note that \eqref{eq-totgeod} is independent of the scale of $\xi$. Further, if $\xi$ satisfies \eqref{eq-totgeod}, then
\begin{align}\label{eq-inttotgeod}
C (X,Y,Z,W) & = 0 \, , & \mbox{for all $X, Y, Z, W \in \Gamma (\mcN_\xi)$.}
\end{align}
where $C$ denotes the Weyl tensor of $\nabla$, i.e.\ the conformally invariant part of the Riemann tensor of $\nabla$. 

The investigation of conditions such as \eqref{eq-totgeod} and \eqref{eq-inttotgeod} will be the subject of this article. For this purpose, we note that an almost null structure $\mcN_\xi$ on $(\mcM,g)$ associated to a projective pure spinor field $[\xi]$ is equivalent to a holomorphic reduction of the structure group of $\mathrm{F} \mcM$  to the stabiliser $P \subset G := \Spin(2m,\C)$ of $[\xi]$ at a point. This $P$ is an instance of a \emph{parabolic} subgroup, and is isomorphic to the semi-direct product $G_0 \ltimes P_+$ where part $G_0$ is reductive, and $P_+$ is nilpotent. The Lie algebras $\prb \subset \g \cong \so(2m,\C)$ of $P$ is  isomorphic to $\g_0 \oplus \prb_+$, where $\g_0 \cong \glie(m,\C)$ and $\prb_+ \cong \wedge^2 \C^m$ are the Lie algebras of $G_0$ and $P_+$ respectively. Here, we have identified $(\mcN_\xi)_p \cong \C^m$ at any point $p$.

Condition \eqref{eq-totgeod} is intimately connected to the notion of \emph{intrinsic torsion} or \emph{structure function} of a first-order $\mathrm{G}$-structure \cites{Chern1953,Bernard1960,Salamon1989}. In the present context, where the structure group is $P$, this is an invariant of $\mcN_\xi$, which, at any point, lies in the $P$-module $\mfW := \mfV \otimes \g / \prb$, where $\mfV \cong \C^{2m}$ is the standard representation of $\g$. Geometrically, it is the obstruction to finding a unique torsion-free connection compatible with $\mcN_\xi$. In other words, it measures the failure of the Levi-Civita connection to preserve $\mcN_\xi$. A number of geometric properties of $\mcN_\xi$ can be encoded as $P$-invariant algebraic conditions on its intrinsic torsion. For instance, condition \eqref{eq-totgeod} can be shown to be equivalent to the intrinsic torsion belonging to a certain proper $P$-submodule of $\mfW$. Identifying all the possible $P$-submodules of $\mfW$ provides a systematic way of `classifying' $\mathrm{G}$-structures with structure group $P$. Such an approach was adopted to provide a classification of almost Hermitian manifolds by Gray and Hervella in \cite{Gray1980}. 

Dealing with condition \eqref{eq-inttotgeod} is similar. In general, if $\mfM$ is a finite $G$-module, $P$ induces a filtration
\begin{align}\label{eq-filtration-M}
\{ 0 \} = \mfM^{\ell+1} \subset \mfM^{\ell} \subset \mfM^{\ell-1} \subset \ldots \subset \mfM^{-k+1} \subset \mfM^{-k} := \mfM
\end{align}
of indecomposable $P$-modules $\mfM^i$ for some $k$ and $\ell$. The nilpotent part $P_+$ acts trivially on each of the associated quotients $\mfM^i/\mfM^{i+1}$, while the reductive part $G_0$, and hence $P$, acts reducibly on these. This applies in particular to the case where $\mfM$ is  the space $\mfC$ of Weyl tensors at a point. We shall see, in this case, $k=\ell=2$ , and condition \eqref{eq-inttotgeod} tells us that the Weyl tensor belongs to the $P$-submodule $\mfC^{-1}:= \mfM^{-1}$ at a point. A precedent for this approach in almost Hermitian geometry can be found in \cites{Tricerri1981,Falcitelli1994}.

The aims of the paper are to
\begin{itemize}
\item give a $P$-invariant decomposition of the space $\mfW$ of intrinsic torsions;
\item give $P$-invariant decompositions of the spaces of curvature tensors, in particular, the tracefree Ricci tensors, Cotton-York tensors and Weyl tensors;
\item apply these decompositions to the study of almost null structures and pure spinor fields on complex Riemannian manifolds.
\end{itemize}

An integral part of this article will be the construction of a \emph{spinor calculus} in relation to the $P$-structure above. This essentially impinges on the remark \cites{Hughston1988,Budinich1989} that if $\xi \in \Gamma(\mcS^+)$ is pure, then any $Z \in \Gamma( \mcN_\xi )$ satisfies
\begin{align}\label{eq-spin2vec}
g(Z,X) & = \langle \zeta , X \cdot \xi \rangle \, , & \mbox{for some $\zeta \in \Gamma( (\mcS^-)^* )$ and for any $X \in \Gamma(\Tgt \mcM)$.}
\end{align}
Here $\langle \cdot , \cdot \rangle$ is the natural pairing between $\mcS^-$ and $(\mcS^-)^*$. This fact will allow us to construct maps whose kernels can be used to define certain $P$-submodules of a given $P$- or $G$-module. This is a standard procedure in representation theory where (irreducible) representations are described in terms of kernels of suitable multilinear maps. For instance, the kernel of the symmetrisation map $\otimes^2 \C^m \rightarrow \odot^2 \C^m$ is the irreducible $\SL(m,\C)$-module $\wedge^2 \C^m$. The only difference here is that the maps will now depend on $[\xi]$.

Before we proceed, it is important to note that there will be obstructions to the global existence of a holomorphic metric or of a holomorphic volume form, not to say of a holomorphic spin structure on a complex manifold. While these issues are interesting in their own right, we shall not be concerned with them in this article, some of which are dealt with in \cite{LeBrun1983}. This being said, all our considerations will essentially be \emph{local}. In particular, we must emphasise that a spin structure can always be introduced locally, and our use of spinors in this context arises essentially from practical considerations.

What is more, a complex manifold $\mcM$ can always be manufactured by \emph{complexifying} a real-analytic oriented manifold $\mcM'$ --- see \cites{Whitney1959,Woodhouse1977,Eastwood1984}. In this case, $\mcM$ is endowed with a reality structure that singles out  $\mcM'$ as a real slice in $\mcM$. Any real analytic structure on $\mcM'$ can be extended to a holomorphic one in a neighbourhood of $\mcM'$ in $\mcM$. This will apply more particularly to a metric $g'$ and spin structure on $\mcM'$. We then obtain a spin complex Riemannian manifold $(\mcM,g)$ from  $(\mcM',g')$. This approach is typically exemplified by the study of real-analytic four-dimensional Lorentzian manifolds, which was central to the development of twistor theory --- see \cites{Penrose1984,Penrose1986} and references therein.

In fact, it is instructive to recall how \eqref{eq-totgeod} and \eqref{eq-inttotgeod} look like when $(\mcM,g)$ is a four-dimensional complex Riemannian manifold. First, $\Spin(4,\C)$ is no longer simple, but isomorphic to $\SL(2,\C)^+ \times \SL(2,\C)^-$ where $\SL(2,\C)^\pm$ are two copies of $\SL(2,\C)$ acting on $\mcS^\pm$.  Following \cite{Penrose1984}, we adorn elements of $\mcS^+$ and $\mcS^-$ with abstract indices, eg $\xi^{A'}$ and $\zeta^A$ respectively. Let us fix a projective spinor $[\xi^{A'}]$ in $\Pp \mcS^+$. Using the fact that $\Tgt \mcM \cong \mcS^- \otimes \mcS^+$ in dimension four, the relation \eqref{eq-spin2vec} simply tells us that any vector $Z^{AB'}$ tangent to the distribution defined by $\xi^{A'}$ must be of the form $Z^{AB'} = \zeta^A \xi^{B'}$ for some $\zeta^A$. Then, equation \eqref{eq-totgeod} can be re-expressed as
\begin{align}\label{eq-foliating_spinor4}
 \xi^{B'} \xi^{A'} \nabla_{AA'} \xi_{B'} & = 0 \, ,
\end{align}
where $\nabla_{AB'}$ is the Levi-Civita connection.  Similarly, condition \eqref{eq-inttotgeod} can be shown to reduce to one on the self-dual part of the Weyl tensor, which we identify with a totally symmetric spinor\footnote{This is often referred to as the \emph{Weyl spinor} in the extant literature, but we shall avoid the term in this article.} $\Psi_{A' B' C' D'}$:
\begin{align}\label{eq-poly-Weyl-spinor}
\Psi_{A' B' C' D'} \xi^{A'} \xi^{B'} \xi^{C'} \xi^{D'} & = 0 \, .
\end{align}
When $(\mcM,g)$ is the complexication of a real-analytic four-dimensional Lorentzian manifold, equation \eqref{eq-foliating_spinor4} describes a real-analytic \emph{shearfree congruence of null geodesics}, and any spinor $\xi^{A'}$ satisfying \eqref{eq-poly-Weyl-spinor} is referred to as a \emph{(gravitational) principal spinor} of $\Psi_{A' B' C' D'}$. Both concepts play an important r\^{o}le in the study of exact solutions of Einstein's field equations.

Finally, while complexifying a \emph{smooth} pseudo-Riemannian manifold will present difficulties in general, the present work can be easily adapted to the setting of an oriented and time-oriented smooth real manifold $\mcM$ equipped with a metric $g$ of signature $(m,m)$ and a spin structure, without the need of complexification. One can then define smooth \emph{real} almost null structures on $(\mcM,g)$ associated to smooth \emph{real} pure spinor fields.

An odd-dimensional analogue of the present paper is given in \cite{Taghavi-Chabert2013}.

\paragraph{Structure of the paper:} Section \ref{sec-algebra} contains a construction of a spinor calculus based on a choice of pure spinor up to scale. Proposition \ref{prop-pure-spinors} is a new algebraic characterisation of intersections of $\alpha$- and $\beta$-planes. Algebraic applications are then given in sections \ref{sec-alg_intrinsic_torsion} and \ref{sec-class-curvature}: Proposition \ref{prop-intrinsic_torsion} gives an invariant decomposition of the space of intrinsic torsions, while Propositions \ref{prop-main_Ricci}, \ref{prop-main_CY} and \ref{prop-main_Weyl} give invariant decompositions of the spaces of Ricci tensors, Cotton-York tensors and Weyl tensors respectively. 

Geometric applications can be found in section \ref{sec-geometry}: Proposition \ref{prop-intrinsic_torsion_connection} is a direct consequence of Proposition \ref{prop-intrinsic_torsion}, and characterises the intrinsic torsion of an almost null structure $\mcN_\xi$ in terms of the covariant derivative of its associated projective pure spinor $[\xi]$. Proposition \ref{prop-conformal-invariance-spinor} examines the conformal invariance of the intrinsic torsion of $\mcN_\xi$. Integrability conditions for the existence of geodetic and recurrent pure spinors are derived in Propositions \ref{prop-int_cond_foliating_spinor} and \ref{prop-int_cond_recurrent_spinor} respectively.
In section \ref{sec-spin-diff-eq}, we study the relation between solutions to differential equations on pure spinor fields: Propositions \ref{prop-foliating_twistor_spinor} and \ref{prop-foliating_twistor_spinor6} give necessary and sufficient conditions on a pure conformal Killing spinor for its associated almost null structure to be integrable. Next, we put forward Conjecture \ref{conjec-GS} generalising the complex Goldberg-Sachs theorem of \cite{Taghavi-Chabert2012}. Finally, in section \ref{sec-real-geom}, we briefly discuss the extent to which the findings of the present article can be applied to real pseudo-Riemannian manifolds.

We round up the paper with three appendices. We have collected in appendix \ref{sec-spinor-descript} material describing the $\g_0$- and $\prb$-submodules of the spaces of curvature tensors. Appendix \ref{sec-spin-calculus4-6} contains a brief discussion of spinor calculus in dimensions four and six. In appendix \ref{sec-conformal} we give some concise background on conformal spin geometry.

\section{Spinor calculus}\label{sec-algebra}
The aim of this section is to construct a spinor calculus given a preferred pure spinor, emphasising its relation with representation theory. While we recall standard facts on the theory of spinors, which can be found in one form or another in the literature \cites{Cartan1981,Budinich1988,Budinich1989,Harnad1992}, in particular the appendix of \cite{Penrose1986}, our approach, which extends the calculus of \cite{Hughston1988}, is relatively novel. Details on the representation theory aspect are given in \cites{Baston1989,Fulton1991,vCap2009}.

\subsection{Clifford algebras and spinor representations}
Let $\mfV$ be an $n$-dimensional complex vector space. We shall adopt the abstract index notation of \cites{Penrose1984} for most of this paper. Standard index-free notation will be used on occasion.
Elements of $\mfV$ and its dual $\mfV^*$ will carry upstairs and downstairs lower-case Roman indices respectively, eg $V^a \in \mfV$ and $\alpha_a \in \mfV^*$. This notation extends to tensor products of $\mfV$ and $\mfV^*$, i.e.\ we write $T \ind{_{ab}^c_d}$ for an element of $\otimes^2 \mfV^* \otimes \mfV \otimes \mfV^*$.
We equip $\mfV$ with a non-degenerate symmetric bilinear form $g_{ab} = g_{(ab)} \in \odot^2 \mfV^*$. Here, as elsewhere, symmetrisation is denoted by round brackets, while skew-symmetrisation by square brackets, eg $\alpha_{abc} = \alpha_{[abc]} \in \wedge^3 \mfV^*$. The metric tensor $g_{ab}$ together with its inverse $g^{ab}$ establishes an isomorphism between $\mfV$ and $\mfV^*$, so that one will lower or raise the indices of tensorial quantities as needed. We shall also make a choice of orientation, i.e.\ an element of $\wedge^n \mfV$, and denote the associated Hodge star operator on $\wedge^\bullet \mfV$ by $\star$. Elements of the two eigenspaces $\wedge^m_+ \mfV$ and $\wedge^m_- \mfV$ of $\star$ on $\wedge^m \mfV$ are referred to as self-dual and anti-self-dual $m$-forms respectively.

We shall be dealing with spinor representations, and for this reason, we shall essentially view any finite representation of the complex special orthogonal group $\SO(2m,\C)$ as finite representation of the spin group $G := \Spin(2m,\C)$, the two-fold covering of $\SO(2m,\C)$.

The \emph{Clifford algebra $\Cl ( \mfV , g )$ of $(\mfV , g )$} is defined as the quotient algebra $\bigotimes^\bullet \mfV / \mfI$ where $\mfI$ is the ideal generated by elements of the form $v \otimes v  + g ( v , v )$,
where $v \in \mfV$.  This implies that $\Cl(\mfV,g)$ is isomorphic to the exterior algebra $\wedge^\bullet \mfV$ as vector spaces, the wedge product of the latter being now replaced by the \emph{Clifford product} $\cdot : \Cl(\mfV,g) \times \Cl(\mfV,g) \rightarrow \Cl(\mfV,g)$ defined by $v \cdot w := v \wedge w - g(v) \lrcorner w$ for any $v$ and $w$ in $\mfV$ viewed as elements of $\Cl(\mfV,g)$.

From now on, we assume $n=2m$. Let $\mfN \subset \mfV$ be a totally null $m$-dimensional subspace, i.e.\ $g|_\mfN = 0$, and fix a dual $\mfN^*$ of $\mfN$ so that $\mfV \cong \mfN \oplus \mfN^*$. Then the vector space $\mfS := \wedge^\bullet \mfN$ can be turned into a $\Cl(\mfV,g)$-module by restricting the Clifford product to it: for any $\xi \in \wedge^\bullet \mfS$, $(v,w) \in \mfN \oplus \mfN^* \cong \mfV$, the action of $\mfV \subset \Cl(\mfV,g)$ on $\mfS$ is given by $(v,w) \cdot \xi = v \wedge  \xi - w \lrcorner  \xi$. The $2^m$-dimensional $\Cl(\mfV,g)$-module $\mfS$ is known as the \emph{spinor space} of $(\mfV,g)$. Further, $\mfS$ splits as $\mfS = \mfS^+ \oplus \mfS^- $, where $\mfS^\pm$ are the $\pm$-eigenspaces of the orientation on $\mfV$, viewed as an element of $\Cl(\mfV,g)$, with
\begin{align*}
 \mfS^+ & \cong \wedge^m \mfN \oplus \wedge^{m-2} \mfN \oplus \ldots \, , &  \mfS^- & \cong \wedge^{m-1} \mfN \oplus \wedge^{m-3} \mfN \oplus \ldots \, .
\end{align*}
The $2^{m-1}$-dimensional complex vector spaces $\mfS^+$ and $\mfS^-$ are called the \emph{positive and negative (chiral) spinor spaces} respectively, and can be shown to be irreducible representations of $G=\Spin(2m,\C)$.

It turns out that the Clifford algebra can also be realised as the algebra of complex $2^m \times 2^m$-matrices acting on $\mfS = \mfS^+ \oplus \mfS^-$. Elements of $\mfS^+$, respectively $\mfS^-$, will carry upstairs primed, respectively unprimed, upper-case Roman indices, eg $\xi^{A'}$, respectively $\alpha^A$, and similarly for their duals $(\mfS^+)^*$ and $(\mfS^-)^*$ with downstairs indices, eg $\eta_{A'}$ and $\beta_A$ respectively. As we shall be working with $\mfS^\pm$ rather than $\mfS$, it will be convenient to think of the generators of the Clifford algebra $\Cl(\mfV,g)$ in terms of the (Van der Waerden) $\gamma$-matrices $\gamma \ind{_a _A ^{B'}}$ and $\gamma \ind{_a _{A'} ^B}$, which satisfy the (reduced) Clifford property
\begin{align}\label{eq-red_Clifford_property}
 \gamma \ind{_{\lp{a}} _{A'} ^C} \gamma \ind{_{\rp{b}} _C ^{B'}} & = - g_{ab} \delta \ind*{_{A'}^{B'}} \, ,& \gamma \ind{_{\lp{a}} _A ^{C'}} \gamma \ind{_{\rp{b}} _{C'} ^B} & = - g_{ab} \delta \ind*{_A^B} \, , 
\end{align}
where $\delta_{A'}^{B'}$ and $\delta_A^B$ are the identity elements on $\mfS^+$ and $\mfS^-$ respectively.
Thus, only skew-symmetrised products of $\gamma$-matrices count, and we shall make use of the notational short hand
\begin{align}
\begin{aligned}\label{eq-multi-gamma}
 \gamma \ind{_{a_1 a_2 \ldots a_q} _A ^{B'}} & := \gamma \ind{_{\lb{a_1}} _A ^{C_1'}} \gamma \ind{_{a_2} _{C_1'} ^{C_2}} \ldots \gamma \ind{_{\rb{a_q}} _{C_{q-1}} ^{B'}} \, , & \gamma \ind{_{a_1 a_2 \ldots a_q} _{A'} ^B} & := \gamma \ind{_{\lb{a_1}} _{A'} ^{C_1}} \gamma \ind{_{a_2} _{C_1} ^{C_2'}} \ldots \gamma \ind{_{\rb{a_q}} _{C_{q-1}'} ^{B}} \, , \\
\gamma \ind{_{a_1 a_2 \ldots a_p} _A ^B} & := \gamma \ind{_{\lb{a_1}} _A ^{C_1'}} \gamma \ind{_{a_2} _{C_1'} ^{C_2}} \ldots \gamma \ind{_{\rb{a_p}} _{C_p'} ^B} \, , & \gamma \ind{_{a_1 a_2 \ldots a_p} _{A'} ^{B'}} & := \gamma \ind{_{\lb{a_1}} _{A'} ^{C_1}} \gamma \ind{_{a_2} _{C_1} ^{C_2'}} \ldots \gamma \ind{_{\rb{a_p}} _{C_{p-1}} ^{B'}} \, ,
\end{aligned}
\end{align}
where $p$ is even and $q$ is odd.
These matrices give us an explicit realisation of the isomorphism $\Cl( \mfV, g_{ab} ) \cong \bigwedge^\bullet \mfV$ as vector spaces. Since $\star: \wedge^k \mfV \overset{\cong}{\rightarrow} \wedge^{2m-k} \mfV$, it is enough to consider forms of degree from $0$ to $m$.

The spinor space $\mfS$ and its dual $\mfS^*$ are equipped with non-degenerate bilinear forms, which realise the isomorphisms
\begin{align}\label{eq-iso-spin}
\begin{aligned}
\gamma \ind{_{A'B'}} \, , \gamma \ind{_{AB}} : \mfS^\pm & \overset{\cong}{\longrightarrow}(\mfS^\pm)^* \, , & & \qquad \mbox{when $m$ even,} \\
\gamma \ind{_{A'B}} \, , \gamma \ind{_{AB'}} : \mfS^\pm & \overset{\cong}{\longrightarrow} (\mfS^\mp)^* \, , & & \qquad \mbox{when $m$ odd,}
 \end{aligned}
\end{align}
by means of which we can raise or lower spinor indices. Thus, the $\gamma$-matrices \eqref{eq-multi-gamma} give rise to bilinear maps
\begin{align}
\begin{aligned}\label{eq-gam-spin}
 \gamma \ind{_{a_1 a_2 \ldots a_p} _{A'B'}} \, , \quad  & \gamma \ind{_{a_1 a_2 \ldots a_p} _{AB}} \, , & & \mbox{for $p \equiv m \pmod{2}$,} \\
 \gamma \ind{_{a_1 a_2 \ldots a_p} _{A'B}} \, ,  \quad & \gamma \ind{_{a_1 a_2 \ldots a_p} _{AB'}} \, , & & \mbox{for $p \equiv m-1 \pmod{2}$,}
\end{aligned}
\end{align}
from $\mfS^\pm \times \mfS^\pm$ or $\mfS^\pm \times \mfS^\mp$ to $\wedge^\bullet \mfV$. The spinor indices of the maps \eqref{eq-iso-spin} and \eqref{eq-gam-spin} are subject to symmetries as explained in \cite{Penrose1986}, and this allows us to prove the following technical lemma needed subsequently.
\begin{lem}\label{lem-technical}
When $m-p$ is even,
\begin{multline*}
 \gamma \ind{_a_{A'}^B} \gamma \ind{_{b_1 \ldots b_p}_{BD}} \gamma \ind{_c_{C'}^D} = (-1)^m \left( \gamma \ind{_{c a b_1 \ldots b_p}_{A'C'}} + g \ind{_{ca}} \gamma \ind{_{b_1 \ldots b_{p-1} b_p}_{A'C'}} \right. \\
 \left. - 2 \, p \, g \ind{_{\lb{b_1}|\lp{a}}} \gamma \ind{_{\rp{c}|b_2 \ldots \rb{b_p}}_{A'C'}} + p (p+1) g \ind{_{a \lb{b_1}}} g \ind{_{|c|b_2}} \gamma \ind{_{b_3 \ldots b_{p-1}\rb{b_p}}_{A'C'}} \right)
\end{multline*}
In particular,
\begin{align*}
 \gamma \ind{^a_{A'}^B} \gamma \ind{_{b_1 \ldots b_p}_{BD}} \gamma \ind{_a_{C'}^D} & = 
 (-1)^m 2 ( m - p)\gamma \ind{_{b_1 \ldots b_p}_{A'C'}} \, .
\end{align*}
 When $m-p$ is odd,
\begin{multline*}
 \gamma \ind{_a_{A'}^B} \gamma \ind{_{b_1 \ldots b_p}_{BD'}} \gamma \ind{_c_C^{D'}} = (-1)^{m-1} \left( \gamma \ind{_{c a b_1 \ldots b_p}_{A'C}} + g \ind{_{ca}} \gamma \ind{_{b_1 \ldots b_{p-1} b_p}_{A'C}} \right. \\
 \left. - 2 p g \ind{_{\lb{b_1}|\lp{a}}} \gamma \ind{_{\rp{c}|b_2 \ldots \rb{b_p}}_{A'C}} + p (p+1) g \ind{_{a \lb{b_1}}} g \ind{_{|c|b_2}} \gamma \ind{_{b_3 \ldots b_{p-1}\rb{b_p}}_{A'C}} \right)
\end{multline*}
In particular,
\begin{align*}
 \gamma \ind{^a_{A'}^B} \gamma \ind{_{b_1 \ldots b_p}_{BD'}} \gamma \ind{_a_C^{D'}} & = 
 (-1)^{m-1} 2 ( m - p)\gamma \ind{_{b_1 \ldots b_p}_{A'C}} \, .
\end{align*}
\end{lem}

Our treatment will be overwhelmingly dimension independent, and for this reason, we shall avoid making use of the bilinear forms \eqref{eq-iso-spin} and \eqref{eq-gam-spin}. It suffices to say that when $p=m$, the bilinear forms \eqref{eq-gam-spin} are always symmetric, and yield injections from $\wedge^m_\pm \mfV$ to $\odot^2 \mfS^\pm$, and surjections from $\odot^2 \mfS^\pm$ to $\wedge^m_\pm \mfV^*$.

\subsection{Null structures and pure spinors}
\begin{defn}
 A \emph{null structure} on $\mfV$ is an $m$-dimensional vector subspace $\mfN \subset \mfV$ that is totally null, i.e. $g_{ab} X^a Y^b = 0$ for all $X^a,Y^a \in \mfN$. A self-dual, respectively anti-self-dual, null structure is called an \emph{$\alpha$-plane}, respectively, a \emph{$\beta$-plane}.
\end{defn}

Let $\xi \ind*{^{A'}}$ be a non-zero spinor in $\mfS^+$, and consider the map
\begin{align*}
 \xi \ind*{_a^A} := \xi \ind*{^{B'}} \gamma \ind{_a_{B'}^A} : \mfV \rightarrow \mfS^- \, .
\end{align*}
By \eqref{eq-red_Clifford_property}, the kernel of $\xi \ind*{_a^A} : \mfV \rightarrow \mfS^-$ is totally null.
\begin{defn}
 A non-zero (positive) spinor $\xi \ind*{^{A'}}$ is said to be \emph{pure} if the kernel of $\xi \ind*{_a^A}: \mfV \rightarrow \mfS^-$ is $m$-dimensional, and thus defines a null structure. 
 
 The projectivisation of the line $\langle \xi^{A'} \rangle$ spanned by a pure spinor $\xi^{A'}$ will be referred to as a \emph{projective (positive) pure spinor} $[\xi^{A'}] \in \Pp \mfS^+$.
 
The same definitions apply to a negative spinor.
\end{defn}
Leaving the details aside, one can show
\begin{prop}[\cite{Cartan1981}]\label{prop-fundamental}
There is a one-to-one correspondence between projective pure spinors and null structures on $(\mfV,g)$. Positive, respectively negative, pure spinors correspond to self-dual, respectively anti-self-dual, null structures.
\end{prop}

Henceforth, we shall assume $m>2$ leaving the special case $m=2$ to appendix \ref{sec-spin-calculus4}. For the remaining of this section and sections \ref{sec-alg_intrinsic_torsion} and \ref{sec-class-curvature}, $\xi \ind*{^{A'}}$ will denote a positive pure spinor. It goes without saying that our statements apply analogously to negative pure spinors. We set
\begin{align}\label{eq-basic-defn}
 \mfS^{\frac{m}{4}} & := \langle \xi \ind*{^{A'}} \rangle \, , & \mfS^{\frac{m-2}{4}} & := \im \xi \ind*{_a^A} : \mfV \rightarrow \mfS^- \, , &
 \mfV^{-\frac{1}{2}} & := \mfV \, , & \mfV^{\frac{1}{2}} & := \ker \xi \ind*{_a^A} : \mfV \rightarrow \mfS^-\, ,
\end{align}
so that one can express the $\alpha$-plane associated to $\xi^{A'}$ as the filtration
\begin{align}\label{eq-null-structure}
\{ 0 \} =: \mfV^{\frac{3}{2}} \subset \mfV^{\frac{1}{2}} \subset \mfV^{-\frac{1}{2}} \, .
\end{align}
The full meaning of this notation, borrowed from \cite{vCap2009}, will be explained in the course of this section. For the moment, the reader should think of these numerical indices as homogeneity degrees. Thus, the map $\xi \ind*{_a^A}$ yields an isomorphism between  $\mfV^{-\frac{1}{2}} / \mfV^{\frac{1}{2}}$ and $\mfS^{\frac{m-2}{4}}$, which we can write as
\begin{align}\label{eq-iso}
 \left( \mfV^{-\frac{1}{2}} / \mfV^{\frac{1}{2}} \right) \otimes \mfS^{\frac{m}{4}} & \cong \mfS^{\frac{m-2}{4}} \, .
\end{align}
While the factor $\mfS^{\frac{m}{4}}$ on the LHS of \eqref{eq-iso} may appear notationally redundant, it nonetheless balances the degrees on each side of \eqref{eq-iso}, i.e.\ $-\frac{1}{2} + \frac{m}{4} = \frac{m-2}{4}$. From \eqref{eq-iso}, it is also clear that $\mfS^{\frac{m-2}{4}}$ is an $m$-dimensional subspace of $\mfS^-$. 

With a slight abuse of notation, we can also think of the map $\xi \ind*{_a^A}$ dually as $\xi \ind*{_a^A} : \mfV^* \leftarrow (\mfS^-)^*$ so that the dual counterpart of \eqref{eq-iso} is given by
\begin{align}\label{eq-dual_iso}
 \mfV^{\frac{1}{2}} & \cong \mfS^{\frac{m}{4}} \otimes \left( \mfS^{-\frac{m-2}{4}} / \mfS^{-\frac{m-6}{4}}  \right) \, ,
\end{align}
where we have defined
\begin{align*}
 \mfS^{-\frac{m-2}{4}} & := (\mfS^-)^* \, , & \mfS^{-\frac{m-6}{4}} & := \ker \xi \ind*{_a^A} : \mfV^* \leftarrow (\mfS^-)^*\, , 
\end{align*}
and made use of $\mfV^{\frac{1}{2}} \cong \left( \mfV^{-\frac{1}{2}} / \mfV^{\frac{1}{2}} \right)^*$. Isomorphism \eqref{eq-dual_iso} can be expressed concretely as follows.
\begin{lem}[\cites{Hughston1988,Budinich1989}]\label{lem-vector_decomposition}
A non-zero vector $V^a$ is an element of $\mfV^{\frac{1}{2}}$ if and only if $V^a = \xi \ind*{^a ^B} v_B$ for some non-zero spinor $v_A$ in $\mfS^{-\frac{m-2}{4}}/\mfS^{-\frac{m-6}{4}}$.
\end{lem}

Since $\mfV^{\frac{1}{2}}$ is a totally null $m$-dimensional vector subspace, we can now conclude
\begin{prop}[\cite{Hughston1988}]\label{prop-purity_cond}
A non-zero spinor $\xi^{A'}$ is pure if and only if it satisfies
\begin{align}\label{eq-purity_cond}
 \xi \ind*{^a ^A} \xi \ind*{_a^B} & = 0 \, .
\end{align} 
\end{prop}
Applying Lemma \ref{lem-technical} to Proposition \ref{prop-purity_cond}, one recovers the following well-known characterisation of pure spinors due to Cartan.
\begin{prop}[\cite{Cartan1981}]\label{prop-Cartan_char}
A non-zero spinor $\xi^{A'}$ is pure if and only if it satisfies
\begin{align}
\left. \begin{aligned}\label{eq-purity_cond_Cartan}
 \gamma \ind{_{a_1 \ldots a_p} _{A' B'}} \xi^{A'} \xi^{B'} & = 0 \, , & \mbox{for all $p < m$, $p \equiv m \pmod 4$,} \\
  \gamma \ind{_{A' B'}} \xi^{A'} \xi^{B'} & = 0 \, , & \mbox{when $m = 0 \pmod 2$,} \\
  \gamma \ind{_{a_1 \ldots a_m} _{A' B'}} \xi^{A'} \xi^{B'} & \neq 0 \, .
  \end{aligned} \right\}
\end{align}
In particular, all non-zero spinors are pure when $m \leq 3$
\end{prop}
We shall refer to both equations \eqref{eq-purity_cond} and \eqref{eq-purity_cond_Cartan} as the \emph{purity conditions} of a spinor $\xi^{A'}$.

Proposition \ref{prop-Cartan_char} tells us that the only non-trivial irreducible component of the tensor product $\xi^{A'} \xi^{B'}$ of a pure spinor $\xi^{A'}$ lies in $\wedge^m_+ \mfV$. In fact, the self-dual $m$-form
$\phi \ind{_{a_1 \ldots a_m}} := \gamma \ind{_{a_1 \ldots a_m} _{A' B'}} \xi^{A'} \xi^{B'}$ annihilates $\mfV^{\frac{1}{2}}$, i.e.\ $\xi \ind*{^{a_1}^A} \phi \ind{_{a_1 a_2 \ldots a_m}} = 0$. In particular, it must be \emph{null} (or \emph{simple} or \emph{decomposable}), ie
\begin{align*}
\phi \ind{_{a_1 \ldots a_m}} & = \xi \ind*{_{a_1}^{A_1}} \ldots \xi \ind*{_{a_m}^{A_m}} \varepsilon \ind*{_{A_1 \ldots A_m}} \quad \in \quad \wedge^m \mfV^{\frac{1}{2}} \, , &
\mbox{for some $\varepsilon \ind{_{A_1 \ldots A_m}} \in \wedge^m \left(\mfS^{-\frac{m-2}{4}}/\mfS^{-\frac{m-6}{4}} \right)$.}
\end{align*}

The next proposition generalises Proposition \ref{prop-Cartan_char} in a certain sense.
\begin{prop}[\cite{Cartan1981}]\label{prop-pure-spinors-Cartan}
 Let $\alpha^{A'}$ and $\beta^{A}$ be two pure spinors of opposite chirality. Then the $\alpha$-plane associated to $\alpha^{A'}$ intersects the $\beta$-plane associated to $\beta^A$ in a totally null $k$-plane where $k \equiv  m-1 \pmod 2$ and $k \leq m-1$ if and only if
\begin{align}
\left. \begin{aligned}\label{eq-pure_pair2}
 \gamma \ind{_{a_1 a_2 \ldots a_p} _{A' B}} \alpha^{A'} \beta^B & = 0 \, , & \qquad \qquad \mbox{for all $p < k$, $p \equiv k \pmod 2$,} \\
  \gamma_{A' B} \alpha^{A'} \beta^B & = 0 \, , & \qquad \qquad \mbox{when $m \equiv 1 \pmod 2$,} \\
 \gamma \ind{_{a_1 a_2 \ldots a_k} _{A' B}} \alpha^{A'} \beta^B & \neq 0 \, .
 \end{aligned}
 \right\}
\end{align}

Let $\beta^B$ and $\rho^A$ be any two negative pure spinors not proportional to each other. Then the $\beta$-planes associated to $\beta^B$ and $\rho^A$ intersect in a totally null $k$-plane where $k \equiv m \pmod 2$ and $k \leq m-2$ if and only if
\begin{align}
\left. \begin{aligned}\label{eq-intersect_pure_pair2}
 \gamma \ind{_{a_1 a_2 \ldots a_p} _{A B}} \beta^A \rho^B & = 0 \, , & \qquad \qquad \mbox{for all $p < k$, $p \equiv k \pmod 2$,} \\
  \gamma_{A B} \beta^A \rho^B  & = 0 \, , & \qquad \qquad \mbox{when $m \equiv 0 \pmod 2$,} \\
 \gamma \ind{_{a_1 a_2 \ldots a_k} _{A B}} \beta^A \rho^B & \neq 0 \, .
 \end{aligned}
 \right\}
\end{align}
The same result holds for any two positive pure spinors not proportional to each other.
\end{prop}

An application of Lemma \ref{lem-technical} leads to the following reformulation of Proposition \ref{prop-pure-spinors-Cartan} when $k=m-1$ and $k=m-2$.
\begin{prop}\label{prop-pure-spinors}
 Let $\alpha^{A'}$ and $\beta^{A}$ be two pure spinors of opposite chirality. Then the $\alpha$-plane associated to $\xi^{A'}$ intersects the $\beta$-plane associated to $\beta^A$ in a totally null $(m-1)$-plane if and only if
\begin{align} \label{eq-pure_pair1}
 \alpha \ind{^a^A} \beta \ind*{_a^{B'}} & = - 2 \, \alpha^{B'} \beta^A \, ,
\end{align}
where $ \alpha \ind{^a^A} := \gamma \ind{^a_{B'}^A} \alpha^{B'} $ and $\beta \ind{^a^{A'}} := \gamma \ind{^a_B^{A'}} \beta^B$.

Let $\beta^B$ and $\rho^A$ be any two negative pure spinors not proportional to each other. Then the $\beta$-planes associated to $\beta^B$ and $\rho^A$ intersect in a totally null $(m-2)$-plane if and only if
\begin{align}\label{eq-intersect_pure_pair1}
 \beta \ind{^a^{(A'}} \rho \ind*{_a^{B')}} & = 0 \, ,
\end{align}
where $\beta \ind{^a^{A'}} := \gamma \ind{^a_B^{A'}} \beta^B$ and $\rho \ind{^a^{A'}} := \gamma \ind{^a_B^{A'}} \rho^B$.

The same result holds for any two positive pure spinors not proportional to each other.
\end{prop}

Finally, as a direct consequence of the previous propositions, we obtain
\begin{cor}\label{prop-pure-spinors3}
Let $\xi^{A'}$ be a pure spinor, and let $\mfV^{\frac{1}{2}}$ and $\mfS^{\frac{m-2}{4}}$ be defined as in \eqref{eq-basic-defn}. Then
\begin{itemize}
\item Any non-zero spinor in $\mfS^{\frac{m-2}{4}}$ is a pure spinor. 
\item The $\beta$-plane associated to any non-zero spinor in $\mfS^{\frac{m-2}{4}}$ intersects the $\alpha$-plane $\mfV^{\frac{1}{2}}$ in a totally null $(m-1)$-plane.
\item The $\beta$-planes associated to any two non-proportional non-zero spinors in $\mfS^{\frac{m-2}{4}}$ intersect in a totally null $(m-2)$-plane.
\end{itemize}
\end{cor}

\begin{proof}
Let $\beta^A$ and $\rho^A$ be a two non-zero spinors in $\mfS^{\frac{m-2}{4}}$ so that $\beta^A = b^a \xi \ind*{_a^A}$ and $\rho^A = \rho^a \xi \ind*{_a^A}$ for some $b^a$ and $\rho^a$ not $\mfV^{\frac{1}{2}}$. In particular, we can assume $b^a$, $\rho^a$ to lie in a complementary subspace of $\mfV^{\frac{1}{2}}$ in $\mfV$, so that they are null, and thus annihilate $\beta^A$ and $\rho^A$ respectively. Assume that $\beta^{[A} \rho^{B]} \neq 0$.  We simply check:
\begin{itemize}
\item $\beta \ind{^a^{A'}} \beta \ind*{_a^{B'}} = b^a b^b \gamma \ind{_a_A^{A'}} \gamma \ind{_b_B^{B'}} (\xi \ind*{^c^A} \xi \ind*{_c^B} ) + 4 b^a \beta \ind*{_a^{(A'}} \xi ^{B')} + 4 b_a b^a \xi^{A'} \xi^{B'} = 0$;

\item $\xi \ind*{^a^A} \beta \ind*{_a^{B'}} = \xi \ind*{^a^A} \left( - b^c \xi \ind*{_a^C} \gamma \ind{_c_C^{B'}} - 2 b_a \xi^{B'} \right) = - 2 \, \beta^A \xi^{B'}$;

\item Finally, since $\mfS^{\frac{m-2}{4}}$ is a vector space of pure spinors, the sum of $\beta^A$ and $\rho^B$ is also a pure spinor, and the result follows by polarisation, i.e.\ $0 = \left( \beta \ind{^a^{A'}} + \rho \ind{^a^{A'}} \right) \left( \beta \ind*{_a^{B'}} + \rho \ind*{_a^{B'}} \right) = 2 \, \beta \ind{^a^{(A'}} \rho \ind*{_a^{B')}}$ i.e.\ the algebraic condition \eqref{eq-intersect_pure_pair1} is satisfied.
\end{itemize}
The result follows by Proposition \ref{prop-pure-spinors}.
\end{proof}

\begin{rem}
The last part of Corollary \ref{prop-pure-spinors3} is an articulation of a standard theorem  \cites{Chevalley1954,Budinich1989} which states that a sufficient and necessary condition for the sum of two pure spinors to be pure is that their respective totally null $m$-planes intersect in a totally null $(m-2)$-plane. 
\end{rem}

\paragraph{Splitting}
It is often more convenient to eliminate the quotient vector spaces in the isomorphisms \eqref{eq-iso} and \eqref{eq-dual_iso} in favour of splittings adapted to them. We split the filtration \eqref{eq-null-structure} as
\begin{align}\label{eq-V-splitting}
 \mfV & = \mfV_{-\frac{1}{2}} \oplus \mfV_{\frac{1}{2}} \, ,
\end{align}
where $\mfV_{-\frac{1}{2}} \subset \mfV^{-\frac{1}{2}}$ is complementary to $\mfV_{\frac{1}{2}} := \mfV^{\frac{1}{2}}$ and is linearly isomorphic to $\mfV^{-\frac{1}{2}} / \mfV^{\frac{1}{2}}$. We note that $\mfV_{-\frac{1}{2}}$ is a totally null $m$-plane dual to $\mfV_{\frac{1}{2}} := \mfV^{\frac{1}{2}}$ by virtue of $\mfV \cong \mfV^*$. In particular, there exists a pure spinor $\eta \ind*{_{A'}}$ dual to $\xi^{A'}$ such that $\mfV_{-\frac{1}{2}}$ annihilates $\eta_{A'}$, ie
\begin{align}\label{eq-ker-eta}
 \mfV_{-\frac{1}{2}} & = \ker \eta \ind*{^a_A} : \mfV \rightarrow (\mfS^-)^* \, ,
\end{align}
Conversely, any choice of spinor dual to $\xi^{A'}$ induces a splitting \eqref{eq-V-splitting} of $\mfV$.

With no loss, we normalise $\xi^{A'}$ and $\eta_{A'}$ as $\xi \ind*{^{A'}} \eta \ind*{_{A'}} = - \frac{1}{2}$. We set
\begin{align*}
 \mfS_{-\frac{m-2}{4}} & := \im \eta \ind*{_a_A} : \mfV \rightarrow (\mfS^-)^* \, ,
\end{align*}
so that by Lemma \ref{lem-vector_decomposition}, any vector $V^a$ in $\mfV_{-\frac{1}{2}}$ takes the form $V^a = \eta \ind*{^a_A} v^A$ for some spinor $v^A$ in $\mfS_{\frac{m-2}{4}} := \mfS^{\frac{m-2}{4}}$, dual to $\mfS_{-\frac{m-2}{4}}$.

Finally, to make the pairing between $\mfS_{-\frac{m-2}{4}}$ and $\mfS_{\frac{m-2}{4}}$ more explicit, we introduce the map
\begin{align}\label{eq-id_m}
 I \ind*{_B^A} & := \eta \ind*{_a_B} \xi \ind*{^a^A} :\mfS^- \rightarrow \mfS^- \, .
\end{align}
By the Clifford property \eqref{eq-red_Clifford_property}, $I \ind*{_B^A}$ is idempotent with trace $I \ind*{_A^A}=m$. Thus, $I \ind*{_B^A}$ must be the identity on $\mfS_{\frac{m-2}{4}}$, or dually, on $\mfS_{-\frac{m-2}{4}}$.

\subsection{The stabiliser of a projective pure spinor in $\so(2m,\C)$ for $m>2$}
We now turn to the decomposition of the Lie algebra $\g := \so(2m,\C)$, which we shall identify with the space $\wedge^2 \mfV^*$ of $2$-forms by means of $g_{ab}$. We remind the reader that we assume $m>2$.

\paragraph{Filtration}
The filtration \eqref{eq-null-structure} on $\mfV$ induces a filtration of vector subspaces
\begin{align}\label{eq-filtration_g}
 \{ 0 \} =: \g^2 \subset \g^1 \subset \g^0 \subset \g^{-1} := \g 
\end{align}
of $\g$, where
\begin{align}\label{eq-def_g_filt}
 \g^0 & := \left\{ \phi \ind{_{ab}} \in \g : \xi \ind*{^a ^A} \xi \ind*{^b ^B} \phi_{ab} = 0 \right\} \, , &
 \g^1 & := \left\{ \phi \ind{_{ab}} \in \g : \xi \ind*{^b ^{A}} \phi \ind{_{ba}} = 0 \right\} \, .
\end{align}
In fact, as can easily be checked from the definitions \eqref{eq-def_g_filt}, $\g$ is a \emph{filtered Lie algebra}, in the sense that the Lie bracket $[ \cdot , \cdot] : \g \times \g \rightarrow \g$ is compatible with the filtration on $\g$, i.e.\ $[ \g^i , \g^j ] \subset \g^{i+j}$, with the convention $\g^i = \{ 0 \}$ for $i \geq 2$, and $\g^i = \g$ for all $i \leq -1$.

\begin{prop}
 The Lie subalgebra $\prb := \g^0$ is the stabiliser of $\xi^{A'}$, ie
\begin{align*}
\phi \ind{_{ab}} \gamma \ind{^{ab}_{B'}^{A'}} \xi \ind*{^{B'}} & \propto \xi \ind*{^{A'}} \, .
\end{align*}
\end{prop}

\begin{proof}
 From the identity $\xi \ind*{^a^A} \xi \ind*{^b^B} \phi \ind{_{ab}} = - \frac{1}{4} \phi \ind{_{ab}} \xi \ind*{^{D'}} \gamma \ind{^{ab}_{D'}^{C'}} \gamma \ind{^c _{C'} ^A} \xi \ind*{_c ^B}$, it follows that the stabiliser of $\xi^{A'}$ is contained in $\g^0$, and, by rewriting $\phi \ind{_{ab}} \gamma \ind{^{ab}_{B'}^{A'}} \xi^{B'} = \phi \, \xi^{A'}$ for some $\phi$, and using \eqref{eq-purity_cond}, in fact contains $\g^0$.
\end{proof}

The Lie subalgebra $\prb$ is a \emph{Lie parabolic subalgebra} of $\so(2m,\C)$. From the Lie bracket commutation relation of $\g^i$, each vector subspace $\g^i$ is a $\prb$-module.

\paragraph{Splitting}
The splitting \eqref{eq-V-splitting} of $\mfV$ adapted to the null structure associated to $\xi^{A'}$ endows $\g$ with the structure of a $|1|$-graded Lie algebra, ie
\begin{align}\label{eq-split-g}
 \g & = \g_{-1} \oplus \g_0 \oplus \g_1 \, , & [ \g_i , \g_j ] \subset \g_{i+j} \, .
\end{align}
where $\g_i \subset \g^i$ are complementary to $\g^{i+1}$, for each $i=-1,0,1$, and we set $\g_i := \{ 0 \}$ when $|i|>2$ for convenience. Explicitly, we have
\begin{align*}
 \g_{-1} & \cong \wedge^2 \mfV_{-\frac{1}{2}} \, , & \g_0 & \cong \mfV_{-\frac{1}{2}} \otimes \mfV_{\frac{1}{2}} \, , & \g_1 & \cong \wedge^2 \mfV_{\frac{1}{2}} \, .
\end{align*}
In particular, $\g_1$ and $\g_{-1}$ are dual to one another, and $\g_0$ is isomorphic to $\glie(m,\C)$, the Lie algebra of the complex general linear group $\GL(m,\C)$ with standard representation $\mfV_{\frac{1}{2}}$. If $\eta_{A'}$ is a pure spinor with $\eta_{A'} \xi^{A'} = -\frac{1}{2}$ so that \eqref{eq-ker-eta} holds, then we can write
\begin{align*}
 \phi \ind{_{ab}} & = \eta \ind*{_{a}_{A}} \eta \ind*{_{b}_{B}} \phi \ind{^{AB}} + 2 \, \xi \ind*{_{\lb{a}}^{A}} \eta \ind*{_{\rb{b}}_{B}} \phi \ind{_A^B} + \xi \ind*{_{a}^{A}} \xi \ind*{_{b}^{B}} \phi \ind{_{AB}} \quad  \in \quad \g_{-1} \oplus \g_0 \oplus \g_1 \, ,
\end{align*}
where $\phi \ind{^{AB}} = \phi \ind{^{[AB]}} \in \wedge^2 \mfS_{\frac{m-2}{4}}$, $\phi \ind{_A^B} \in \mfS_{-\frac{m-2}{4}}\otimes \mfS_{\frac{m-2}{4}}$, $\phi \ind{_{AB}} = \phi \ind{_{[AB]}} \in \wedge^2 \mfS_{-\frac{m-2}{4}}$. Here, we emphasise that spinor indices should not be raised nor lowered, i.e.\ $\phi \ind{^{AB}}$, $\phi \ind{_{AB}}$ and $\phi \ind{_A^B}$ are independent of each other.

By the commutation relation, $\g_1$ is nilpotent. Further, since $\g_0$ is reductive, there is a direct sum decomposition $\g_0 = \mfz_0 \oplus \slie_0$ where $\mfz_0$ is the one-dimensional center of $\g_0$, and $\slie_0$ is the simple part of $\g_0$, which is isomorphic to $\slie(m,\C)$, the Lie algebra of the complex special linear group $\SL(m,\C)$. The center $\mfz_0$ can be seen to be spanned by the  element  
\begin{align}\label{eq-grading-element}
 E \ind{_{ab}} & := - \xi \ind*{_{\lb{a}}^A} \eta \ind*{_{\rb{b}}_A} = - \xi \ind*{_a^A} \eta \ind*{_b_A} + \frac{1}{2} g \ind{_{ab}} \, ,
\end{align}
with respect to which any $\phi_{ab} \in \slie_0$ is tracefree, i.e.\ $E \ind{^{ab}} \phi \ind{_{ab}} = 0$. 
More generally, any $\phi_{ab} \in \g_0$ admits the decomposition
\begin{align*}
 \phi \ind{_{ab}} & = \phi \, \omega \ind{_{ab}}  + 2 \, \xi \ind*{_{\lb{a}}^{A}} \eta \ind*{_{\rb{b}}_{B}} \phi \ind{_A^B} \quad  \in \quad \g_0 = \mfz_0 \oplus \slie_0 \, ,
\end{align*}
where $\phi \in \C$ and $\phi \ind{_A^B} \in \mfS_{-\frac{m-2}{4}}\otimes \mfS_{\frac{m-2}{4}}$ is tracefree in the sense that $\phi \ind{_A^B} I \ind*{_B^A} = 0$ where we recall $I \ind*{_A^B} = \xi \ind*{^a^A} \eta \ind*{_a_B}$ is the identity on $\mfS_{\frac{m-2}{4}}$ (see \eqref{eq-id_m}). Here, we have defined, for convenience, $\omega \ind{_{ab}} := - 2 \, E \ind{_{ab}}$ so that $\omega \ind{_a^c} \omega \ind{_c^b} = g \ind*{_a^b}$.

The element $E_{ab}$ has the property $\xi \ind*{^b^A} E \ind{_b^a} = \frac{1}{2} \xi \ind*{^a^A}$ and $\eta \ind*{^b_A} E \ind{_b^a} = - \frac{1}{2} \eta \ind*{^a_A}$, i.e.\ $E_{ab}$ has eigenvalues $\pm \frac{1}{2}$ on $\mfV_{\pm\frac{1}{2}}$. The action of $E_{ab}$ extends by derivation to any tensor product of $\mfV$ and $\mfV^*$, and in particular $E_{ab}$ has eigenvalues $i$ on $\g_i$ for $i=-1,0,1$. Now, the image of $E \ind{_{ab}}$ in the Clifford algebra $\Cl(\mfV,g)$ restricted to $\End(\mfS^+)$ is $E \ind{_{B'}^{A'}} := - \frac{1}{4} E \ind{_{ab}} \gamma \ind{^{ab}_{B'}^{A'}}$, and has eigenvalues $\frac{m}{4}$ on $\mfS_{\frac{m}{4}}$, and similarly for the action of $E_{ab}$ on $\mfS^-$ and their duals. For these reasons, $E_{ab}$ is referred to as \emph{the grading element of $\g$}.

\paragraph{Parabolic Lie subgroups}
Moving to the level of Lie groups, we denote by $P$ the stabiliser of the projective pure spinor $[\xi]$ in $G=\Spin(2m,\C)$ so that $P$ has Lie algebras $\prb$. More precisely, $P$ admits a Levy decomposition $P = G_0 \ltimes P_+$, where the image of $G_0$ in $\SO(2m,\C)$ under the covering map is the complex general Lie group $\GL(m,\C)$ while $P_+$ has nilpotent Lie algebra $\g_1$. Any of the $\prb$-invariant structures, including filtrations and associated graded vector spaces, considered in this article are also $P$-invariant and can be regarded as finite representations of $P$. Similarly, we can view $\g_0$-modules as $G_0$-modules.

\paragraph{Associated graded vector space}
Associated to the filtration \eqref{eq-filtration_g} is the graded $\prb$-module $\gr(\g)=\bigoplus_{i=-1}^1 \gr_i (\g)$ where $\gr_i (\g) := \g^i/\g^{i+1}$. In fact, each $\gr_i(\g)$ is a $\prb$-module since $[\g^1,\g^i] \subset \g^{i+1}$. Each $\gr_i(\g)$ is lineraly isomorphic to the $\g_0$-module $\g_i$ of the splitting \eqref{eq-split-g}. There is a further direct sum decomposition
\begin{align*}
 \gr_0 (\g) & \cong \g_0^0 \oplus \g_0^1  \, , & & \mbox{where} & \g_0^0 & := \left( \g_1 + \mfz_0 \right) / \g_1 \, , & \g_0^1 & := \left( \g_1 + \slie_0 \right) / \g_1 \, ,
\end{align*}
of $\prb$-modules, where $\g_{0}^0 \cong \mfz_0$ and $\g_{0}^1 \cong \slie_0$ as vector spaces. Let us set $\g_{\pm1}^0 := \gr_{\pm1} (\g)$ for convenience. Then we can represent $\gr(\g)$ in the form of a directed graph
\begin{align}\label{eq-Penrose-diagram-g}
\begin{aligned}
 \xymatrix@R=1em{ & \g_0^0 \ar[dr] & \\
\g_1^0 \ar[ur] \ar[dr] &  & \g_{-1}^0 \\
& \g_0^1 \ar[ur] & }
\end{aligned}
\end{align}
where the arrows are defined by the property
\begin{align}\label{eq-def-arrow1}
 \g_i^j & \longrightarrow \g_{i-1}^k & \Longleftrightarrow & &  \breve{\g}_i^j \subset \g_1 \cdot \breve{\g}_{i-1}^k \, ,
\end{align}
for any irreducible $\g_0$-module $\breve{\g}_i^j$ linearly isomorphic to $\g_i^j$. Here the $\cdot$ denotes the algebraic action of $\g$ on any $\g$-module.

Such a description can be made explicit by defining the maps, for any $\phi \ind{_{ab}} \in \g$,
\begin{align}\label{eq-map_grading_g}
 {}^\g _\xi \Pi_{-1}^0 (\phi) & := \xi \ind*{^a^A} \xi \ind*{^b^B} \phi_{ab} \, , &
 {}^\g _\xi \Pi_0^0 (\phi) & := \xi \ind*{^{ab} ^{A'}} \phi_{ab} \, , &
 {}^\g _\xi \Pi_0^1 (\phi) & := \xi \ind*{^c ^{A}} \phi \ind{_{cb}} + \frac{1}{n} \gamma \ind{_b _{C'} ^A} \xi \ind*{^{cd} ^{C'}} \phi \ind{_{cd}} \, .
\end{align}
where $\xi \ind*{_{ab}^{A'}} := \xi \ind*{^{B'}} \gamma \ind{_{ab}_{B'}^{A'}} : \wedge^2 \mfV \rightarrow \mfS^+$ and $\mfS^{\frac{m-4}{4}} := \im \xi \ind*{_{ab}^{A'}} : \wedge^2 \mfV \rightarrow \mfS^+$. Then the kernels of the maps ${}^\g _\xi \Pi_i^j$ correspond to the $\prb$-submodules of $\g$, ie
\begin{align*}
\g^1 + \mfz_0 & =  \{ \phi_{ab} \in \g : {}^\g _\xi \Pi_0^1 (\phi) = 0 \} \, , & \g^1 + \slie_0 & =  \{ \phi_{ab} \in \g : {}^\g _\xi \Pi_0^0 (\phi) = 0 \} \, , \\
\g^0 & =  \{ \phi_{ab} \in \g : {}^\g _\xi \Pi_{-1}^0 (\phi) = 0 \} \, , & \g^1 & =  \{ \phi_{ab} \in \g : {}^\g _\xi \Pi_0^0 (\phi) = {}^\g _\xi \Pi_0^1 (\phi) = 0 \} \, .
\end{align*}
The inclusions $\g^1 \subset \g^1 + \mfz_0 \subset \g^0$ and $\g^1 \subset \g^1 + \slie_0 \subset \g^0$ now follow from $\ker {}^\g _\xi \Pi_0^i  \subset \ker {}^\g _\xi \Pi_{-1}^0$.
Passing now to the associated graded module $\gr(\g)$, we can express the irreducible $\prb$-modules $\g_0^j$ in terms of ${}^\g _\xi \Pi_i^j$, eg.
\begin{align*}
\g_0^0 & = \left\{ \phi \ind{_{ab}} \in \g^0 : {}^\g _\xi \Pi_0^1 (\phi) = 0 \right\} / \g^1 \, , &
\g_0^1 & = \left\{ \phi \ind{_{ab}} \in \g^0 : {}^\g _\xi \Pi_0^0 (\phi) = 0 \right\} / \g^1 \, ,
\end{align*}                                      
and so on. The irreducibility of $\g_i^j$ from the fact that the maps ${}^\g _\xi \Pi_i^j$ are saturated with symmetries.

\subsection{Generalisations}\label{sec-gen-fil}
In more generality, for any arbitrary finite $\g$-module $\mfM$, the parabolic subalgebra $\prb$ induces a filtration
\begin{align}\label{eq-M-filtration}
\{ 0 \} =: \mfM^{\ell+1} \subset \mfM^\ell \subset \mfM^{\ell-1} \subset \ldots \subset \mfM^{-k+1} \subset \mfM^{-k} := \mfM \, ,
\end{align}
for some $k$ and $\ell$, of $\prb$-modules, with associated graded $\prb$-module $\gr (\mfM) = \bigoplus \gr_i(\mfM)$, where $\gr_i (\mfM) := \mfM^i / \mfM^{i+1}$,
on which the grading element $E$ acts diagonalisably, with eigenvalues $i$. Each $\gr_i (\mfM)$ splits as a direct sum of irreducible $\prb$-modules $\gr_i (\mfM) = \mfM_i^0 \oplus \mfM_i^1 \oplus \ldots \oplus \mfM_i^\ell$ for some $\ell$ depending on $i$, and each $\mfM_i^j$ is isomorphic to an irreducible $\g_0$-module $\breve{\mfM}_i^j$. It is in fact easier to obtain the irreducible $\g_0$-modules of $\gr(\mfM)$ by  `branching' from $\g$ to $\g_0$. Using ad hoc methods, one can construct suitable bases for the irreducible $\g_0$-modules, and check that they add up to a basis for $\mfM$. In particular, one must have $\dim \mfM = \sum_{i,j} \dim \mfM_i^j = \sum_{i,j} \dim \breve{\mfM}_i^j$.

We can then construct a graph on $\gr (\mfM)$  as follows: we let the nilpotent part $\g_1$ of $\prb$ act on each $\breve{\mfM}_i^j$, and draw an arrow $\mfM_i^j \rightarrow \mfM_{i-1}^k$ for some $i,j,k$, whenever $\breve{\mfM}_i^j \subset \g_1 \cdot \breve{\mfM}_{i-1}^k$.
This graph allows us to identify the $\prb$-submodules of $\mfM$ obtained by letting $\prb$ act on each $\breve{\mfM}_i^j$. Such $\prb$-submodules can be expressed in terms of kernels of maps ${}^\mfM _\xi \Pi_i^j$ analogous to \eqref{eq-map_grading_g}. From the irreducibilty of $\mfM_i^j$, each ${}^\mfM _\xi \Pi_i^j$ must be `saturated' with symmetries in the sense of \cite{Penrose1984}. The main application of this procedure will be found in section \ref{sec-class-curvature}, where we shall take $\mfM$ to be the space of some irreducible (algebraic) curvature tensors.

If $\mfM$  is a tensor, as opposed to spinor, representation, then we can view it as a $\g$-submodule of  $\bigotimes^p \mfV$. The filtration \eqref{eq-null-structure} induces a filtration of $\prb$-modules on $\bigotimes^p \mfV$, and thus on $\mfM$ in the obvious way: each $\mfM^i$ in \eqref{eq-M-filtration} is a $\prb$-submodule of
\begin{align*}
\sum_{i_1 + \ldots + i_p = j} \mfV^{i_1} \otimes \ldots \otimes \mfV^{i_p} \, .
\end{align*}
When $\mfM$ is one of the spinor modules $\mfS^{\pm}$, the description of the filtration of $\prb$-modules can be carried out as follows. Define the maps
\begin{align*}
 \xi \ind*{_{a_1 \ldots a_{2k-1}}^A} & := \xi \ind*{^{B'}} \gamma \ind{_{a_1 \ldots a_{2k-1}}_{B'} ^A} : \wedge^{2k-1} \mfV \rightarrow \mfS^- \, , &
 \xi \ind*{_{a_1 \ldots a_{2k}}^{A'}} & := \xi \ind*{^{B'}} \gamma \ind{_{a_1 \ldots a_{2k}}_{B'}^{A'}} : \wedge^{2k} \mfV \rightarrow \mfS^+ \, ,
\end{align*}
for $k=1, \ldots , m$. Then, using the Clifford property \eqref{eq-red_Clifford_property}, one can see that $\mfS^\pm$ admit $\prb$-invariant filtrations
\begin{align*}
\left.
\begin{aligned}
  \mfS^{\frac{m}{4}} \subset \mfS^{\frac{m-4}{4}} \subset \ldots \subset \mfS^{-\frac{m-4}{4}} \subset \mfS^{-\frac{m}{4}} & = \mfS^+ \, \\
  \mfS^{\frac{m-2}{4}} \subset \mfS^{\frac{m-6}{4}} \subset \ldots \subset \mfS^{-\frac{m-6}{4}} \subset \mfS^{-\frac{m-2}{4}} & = \mfS^- \,
 \end{aligned}
\right\} & & \mbox{when $m$ is even,} \\
\left.
\begin{aligned}
   \mfS^{\frac{m}{4}} \subset \mfS^{\frac{m-4}{4}} \subset \ldots \subset \mfS^{-\frac{m-6}{4}} \subset \mfS^{-\frac{m-2}{4}} & = \mfS^+ \,  \\
  \mfS^{\frac{m-2}{4}} \subset \mfS^{\frac{m-6}{4}} \subset \ldots \subset \mfS^{-\frac{m-4}{4}} \subset \mfS^{-\frac{m}{4}} & = \mfS^- \,
 \end{aligned}
\right\} & & \mbox{when $m$ is odd,}
\end{align*}
where we have defined $\mfS^{\frac{m}{4}} := \langle \xi^{A'} \rangle$ and
\begin{align*}
 \mfS^{\frac{m-4k+2}{4}} & := \im \xi \ind*{_{a_1 \ldots a_{2k-1}}^A} : \wedge^{2k-1} \mfV \rightarrow \mfS^- \, , &
 \mfS^{\frac{m-4k}{4}} & := \im \xi \ind*{_{a_1 \ldots a_{2k}}^{A'}} : \wedge^{2k} \mfV \rightarrow \mfS^+ \, ,
\end{align*}
for $k=1, \ldots ,m$. Using the isomorphisms \eqref{eq-iso-spin}, the above filtrations are also filtrations on the dual spinor spaces $(\mfS^\pm)^*$, where each of the $\prb$-modules can be identified with the kernels
\begin{align*}
 \mfS^{-\frac{m-4k-2}{4}} & = \ker \xi \ind*{_{a_1 \ldots a_{2k-1}} ^B} : \wedge^{2k-1} \mfV^* \leftarrow (\mfS^-)^* \, , &
 \mfS^{-\frac{m-4k-4}{4}} & = \ker \xi \ind*{_{a_1 \ldots a_{2k}} ^{B'}} : \wedge^{2k} \mfV^* \leftarrow (\mfS^+)^* \, ,
\end{align*}
for $k=1, \ldots ,m$, and $\mfS^{-\frac{m-4}{4}} = \ker \xi^{A'} : \C \leftarrow (\mfS^+)^*$.

A choice of splitting \eqref{eq-V-splitting} fixes $\g_0$-modules $\mfS_i \subset \mfS^i$ such that $\mfS^i = \mfS_i \oplus \mfS^{i+1}$ and thus induces gradings
\begin{align*}
\left.
\begin{aligned}
  \mfS_{\frac{m}{4}} \oplus \mfS_{\frac{m-4}{4}} \oplus \ldots \oplus \mfS_{-\frac{m-4}{4}} \oplus \mfS_{-\frac{m}{4}} & = \mfS^+ \, \\
  \mfS_{\frac{m-2}{4}} \oplus \mfS_{\frac{m-6}{4}} \oplus \ldots \oplus \mfS_{-\frac{m-6}{4}} \oplus \mfS_{-\frac{m-2}{4}} & = \mfS^- \,
 \end{aligned}
\right\} & & \mbox{when $m$ is even,} \\
\left.
\begin{aligned}
   \mfS_{\frac{m}{4}} \oplus \mfS_{\frac{m-4}{4}} \oplus \ldots \oplus \mfS_{-\frac{m-6}{4}} \oplus \mfS_{-\frac{m-2}{4}} & = \mfS^+ \,  \\
  \mfS_{\frac{m-2}{4}} \oplus \mfS_{\frac{m-6}{4}} \oplus \ldots \oplus \mfS_{-\frac{m-4}{4}} \oplus \mfS_{-\frac{m}{4}} & = \mfS^- \,
 \end{aligned}
\right\} & & \mbox{when $m$ is odd.}
\end{align*}
The grading element $E$ of $\g$ defined by \eqref{eq-grading-element} has eigenvalues $\frac{2i-m}{4}$ on $\mfS_{\frac{2i-m}{4}}$.

Finally generalising \eqref{eq-dual_iso}, one has isomorphisms
\begin{align*}
 \wedge^k \mfV^{\frac{1}{2}} & \cong \mfS^{\frac{m}{4}} \otimes \left( \mfS^{-\frac{m-2k}{4}}/\mfS^{-\frac{m-2k-4}{4}} \right) \, , & k & = 0 \, , \ldots \, , m-1\, , \\
 \wedge^m \mfV^{\frac{1}{2}} & \cong \mfS^{\frac{m}{4}} \otimes \mfS^{\frac{m}{4}} \, ,
\end{align*}
the latter being the purity condition of Proposition \ref{prop-Cartan_char}. In particular, when $k=2$, we have
\begin{align*}
\g^1 & \cong \mfS^{\frac{m}{4}} \otimes \left( \mfS^{-\frac{m-4}{4}} / \mfS^{-\frac{m-8}{4}} \right) \, , & \left( \g^{-1} / \g^0 \right) \otimes \mfS^{\frac{m}{4}} & \cong \mfS^{\frac{m-4}{4}} / \mfS^{\frac{m}{4}} \, .
\end{align*}
 We note that this description of $\mfS^\pm$ is consistent with the identification of $\mfS$ with $\wedge^\bullet \mfV^{\frac{1}{2}}$.

\subsection{Null Grassmannians}
The space of all null structures in $(\mfV,g)$ splits into two connected components $\Gr_m^+( \mfV,g)$ and $\Gr_m^-( \mfV,g)$, which we identify with the spaces of self-dual null structures ($\alpha$-planes) and anti-self-dual null structures ($\beta$-planes) in $(\mfV,g)$ respectively. These spaces are conventionally referred to as \emph{null} (or \emph{isotropic}) \emph{Grassmannians}. Proposition \ref{prop-fundamental} tells us that $\Gr_m^+( \mfV,g)$ can be identified with the space of all projective pure positive spinors, and must therefore be isomorphic to the homogeneous space $G/P$ where as before $G = \Spin(2m,\C)$ and $P$ is the stabiliser of a projective pure positive spinor. The description of $\Gr_m^- (\mfV,g)$ is similar. In particular, when $m=1,2,3$, the absence of purity conditions means that each of $\Gr_m^\pm (\mfV,g)$ is isomorphic to the complex projective space $\CP^{\frac{1}{2}m(m-1)}$, and when $m>3$, each can be realised as compact complex subvarieties of $\Pp \mfS^\pm$ of dimension $\frac{1}{2}m(m-1)$ being the dimension of $\g_{-1} \cong \g/\prb$.

\subsection{Real pure spinors}\label{sec-real-spin}
One can also consider a $2m$-dimensional real vector space $\mfV$ equipped with a definite or indefinite non-degenerate symmetric bilinear form $g$. In general, the spinor representations of $(\mfV,g)$ are complex vector spaces equipped with a real or quaternionic structure. We then have a notion of pure spinor of \emph{real index} $r$, where $r$ is the real dimension of the intersection of the associated totally null complex $m$-plane $\mfN_\xi$ of the complexification of $(\mfV,g)$ with its complex conjugate. The real index depends on the signature of $g$. For instance, if $g$ is positive definite, $r$ is always zero: $\mfN_\xi$ and its conjugate define a Hermitian signature on $(\mfV,g)$. In Lorentzian signature, $r$ is always $1$, and the analogous structure is known as a \emph{Robinson structure} \cites{Nurowski2002,Trautman2002,Taghavi-Chabertb}. We refer to \cite{Kopczy'nski1992} for details.

More relevant to the present article, however, is the case when $g$ has signature $(m,m)$. Then, $r \equiv m \pmod{2} \leq m$, the spinor representations are spanned by \emph{real} pure spinors (when $r=m$) associated to \emph{real} totally null $m$-planes in $\mfV$. The algebraic setup of the previous sections carries over to this real setting with no major change. The complex Lie algebra $\so(2m,\C)$ is replaced by the real form $\so(m,m)$. The parabolic Lie subalgebra stabilising a real pure spinor is a real form of the complex parabolic $\prb$, and is also described in terms of a $|1|$-grading on $\so(m,m)$. The story is similar at the Lie group level, where $\Spin(2m,\C)$ is replaced by the connected identity component of the real Lie group $\Spin(m,m)$. The next two sections \ref{sec-alg_intrinsic_torsion} and \ref{sec-class-curvature} can also be translated into this real case with no important issue.

\section{Decomposition of the intrinsic torsion}\label{sec-alg_intrinsic_torsion}
As before, we assume $m>2$. Let us consider the $\prb$-module
\begin{align}\label{eq-intrinsic-torsion}
\mfW := \mfV \otimes \left( \g/\prb \right) \, ,
\end{align}
where, as usual, $\g:=\so(2m,\C)$, with standard representation $\mfV$, and $\prb$ the parabolic Lie subalgebra of $\g$ stabilising a projective pure spinor $[  \xi^{A'} ]$.  In section \ref{sec-geometry}, we shall give the module $\mfW$ the geometrical interpretation of the space of intrinsic torsions of a $\mathrm{G}$-structure with structure group $P$.

\begin{nota}\label{nota-prop}
In the table of the following proposition, `$\prb$-module' and `$\g_0$-module' are abbreviated `$\prb$-mod' and `$\g_0$-mod' respecitvely. We also use the notation $\mfM \circledcirc \mfM'$ for the \emph{Cartan product} of two representations $\mfM$ and $\mfM'$ --- see \cite{Eastwood2004}. This is the unique irreducible representation of highest dimension in the tensor product $\mfM \otimes \mfM'$. For $\glie(m,\C)$-modules, this is either the symmetric product $\odot$, the tracefree tensor product $\otimes_\circ$, or a combination of both depending on $\mfM$ and $\mfM'$. Finally, the algebraic action of $\g$ on any $\g$-module will be denoted by a dot $\cdot$.
\end{nota}

\begin{prop}\label{prop-intrinsic_torsion}
The $\prb$-module $\mfW$ admits a filtration
\begin{align}\label{eq-intor-incl}
 \mfW^{-\frac{1}{2}} \subset \mfW^{-\frac{3}{2}} \, ,
\end{align}
of $\prb$-modules on $\mfW$, where $\mfW^{-\frac{3}{2}} := \mfV^{-\frac{1}{2}} \otimes \left( \g^{-1}/\g^0 \right)$ and $\mfW^{-\frac{1}{2}} := \mfV^{\frac{1}{2}} \otimes \left( \g^{-1} /\g^0 \right)$.

The associated graded $\prb$-module
\begin{align}\label{eq-assgradW}
\gr(\mfW) & = \gr_{-\frac{1}{2}} (\mfW) \oplus \gr_{-\frac{3}{2}} (\mfW) = \mfW^{-\frac{1}{2}} \oplus \left( \mfW^{-\frac{3}{2}} / \mfW^{-\frac{1}{2}} \right)
\end{align}
decomposes into a direct sum
\begin{align*}
 \gr_{-\tfrac{1}{2}} (\mfW) & = \mfW_{-\tfrac{1}{2}}^0 \oplus \mfW_{-\tfrac{1}{2}}^1 \, , & \gr_{-\tfrac{3}{2}} (\mfW) & = \mfW_{-\tfrac{3}{2}}^0 \oplus \mfW_{-\tfrac{3}{2}}^1 \, , &
\end{align*}
of irreducible $\prb$-modules as described below:
\vspace{-6.5mm}
\begin{center}
\begin{minipage}[b]{0.45\linewidth}\centering
\begin{displaymath}
{\renewcommand{\arraystretch}{1.8}
\begin{array}{||c|c|c||}
\hline
\text{$\prb$-mod} & \text{$\g_0$-mod} & \text{Dimension} \\
\hline
\mfW_{-\frac{3}{2}}^0 & \wedge^3 \mfV_{-\frac{1}{2}} & {\scriptstyle \frac{1}{3!}m(m-1)(m-2)} \\
\mfW_{-\frac{3}{2}}^1 & \mfV_{-\frac{1}{2}} \circledcirc \left( \wedge^2 \mfV_{-\frac{1}{2}} \right) & \frac{1}{3}m(m^2-1) \\
\hline
\end{array}
}
\end{displaymath}
\end{minipage}
\begin{minipage}[b]{0.45\linewidth}\centering
\begin{displaymath}
{\renewcommand{\arraystretch}{1.8}
\begin{array}{||c|c|c||}
\hline
\text{$\prb$-mod} & \text{$\g_0$-mod} & \text{Dimension} \\
\hline
\mfW_{-\frac{1}{2}}^0 & \mfV_{-\frac{1}{2}}
& m \\
\mfW_{-\frac{1}{2}}^1 & \mfV_{\frac{1}{2}} \circledcirc \left( \wedge^2 \mfV_{-\frac{1}{2}} \right) & {\scriptstyle \frac{1}{2}m(m+1)(m-2)} \\
\hline
\end{array}
}
\end{displaymath}
\end{minipage}
\end{center}
Further,
\begin{align}\label{eq-W-ker}
 \mfW_i^j & = \{ \Gamma _{abc} \xi^{bB} \xi^{cC} \in \mfW^i : {}^\mfW _\xi \Pi_i^k ( \Gamma ) = 0 \, , \, \, \mbox{for all $k \neq j$} \} / \mfW^{i+1} \, , & i & = - \frac{3}{2}, - \frac{1}{2} \, ,
\end{align}
where
\begin{align*}
 {}^\mfW _\xi \Pi_{-\frac{3}{2}}^0 (\Gamma ) & := \Gamma \ind{_{abc}} \xi \ind*{^a^{[A}} \xi \ind*{^b^{B}} \xi \ind*{^c^{C]}} \, , \\
 {}^\mfW _\xi \Pi_{-\frac{3}{2}}^1 (\Gamma ) & := \Gamma \ind{_{abc}} \xi \ind*{^a^{(A}} \xi \ind*{^b^{B)}} \xi \ind*{^c^C} \, , \\
 {}^\mfW _\xi \Pi_{-\frac{1}{2}}^0 (\Gamma ) & := \xi \ind*{^{A'}} \Gamma \ind{_{bcd}} \xi \ind*{^{cd}^{D'}} \gamma \ind{^b_{D'}^B} + \xi \ind*{^b^B} \Gamma \ind{_{bcd}} \xi \ind*{^{cd}^{A'}} \, , \\
{}^\mfW _\xi \Pi_{-\frac{1}{2}}^1 (\Gamma ) & :=
\begin{cases}
\Gamma \ind{_{abc}} \xi \ind*{^b^B} \xi \ind*{^c^C} - \frac{1}{2(m-1)} \left( \xi \ind*{_a^{\lb{B}}} \Gamma \ind{_{bcd}} \xi \ind*{^{cd}^{D'}} \gamma \ind{^b_{D'}^{\rb{C}}} + \xi \ind*{^b^{\lb{B}}} \Gamma \ind{_{bcd}} \xi \ind*{^{cd}^{D'}} \gamma \ind{_a_{D'}^{\rb{C}}}  \right) \, , & m > 3 \, , \\
\Gamma \ind{_{abc}} \xi \ind*{^b^B} \xi \ind*{^c^C} - \frac{1}{4} \left( \xi \ind*{_a^{\lb{B}|}} \Gamma \ind{_{bcd}} \xi \ind*{^{cd}_D} \gamma \ind{^b^{D|\rb{C}}} + \xi \ind*{^b^{\lb{B}|}} \Gamma \ind{_{bcd}} \xi \ind*{^{cd}_D} \gamma \ind{_a^{D|\rb{C}}}  \right)
+ \frac{1}{6} \xi \ind*{^b^A} \Gamma \ind{_{bcd}} \xi \ind*{^{cd}_A } \gamma \ind{_a^{BC}} \, , & m = 3 \, ,
\end{cases} 
\end{align*}
where $\Gamma \ind{_{abc}} \in \mfV \otimes \g$.  For $m=3$, we have made use of the isomorphism $\mfS^+ \cong (\mfS^-)^*$. Notationally, the primed indices are eliminated, and the $\gamma$-matrices take the form $\gamma \ind{^a^{AB}}$ and $\gamma \ind{^a_{AB}}$, and are skew-symmetric in their spinor indices. 

Finally, the $\prb$-module $\gr(\mfW)$ can be expressed by means of the directed graph
\begin{align*}
 \xymatrix@R=1em{
 \mfW_{-\frac{1}{2}}^1 \ar@{.>}[ddr] \ar[r] & \mfW_{-\frac{3}{2}}^1 \\
& \\
\mfW_{-\frac{1}{2}}^0 \ar[uur] \ar[r] & \mfW_{-\frac{3}{2}}^0 }
\end{align*}
where the dotted arrow occurs only when $m>3$. Here, an arrow from  $\mfW_i^j$ to $\mfW_{i-1}^k$ for some $i,j,k$ implies that $\breve{\mfW}_i^j \subset \g_1 \cdot \breve{\mfW}_{i-1}^k$ for any choice of irreducible $\g_0$-modules $\breve{\mfW}_i^j$ and $\breve{\mfW}_{i-1}^k$ isomorphic to $\mfW_i^j$ and $\mfW_{i-1}^k$ respectively, or equivalently that $ \ker {}^\mfW _\xi \Pi_i^j \subset \ker {}^\mfW _\xi \Pi_{i-1}^k$.
\end{prop}

\begin{proof}
Since $\mfW$ is not a $\g$-module, one cannot strictly follow the argument of section \ref{sec-gen-fil}. The idea is nevertheless very similar. We first note that the filtration \eqref{eq-null-structure} on $\mfV$ induces the filtration \eqref{eq-intor-incl} of $\prb$-submodules of $\mfW$. Now consider the associated graded $\prb$-module \eqref{eq-assgradW}.
To make the analysis more tractable, we can work with the grading \eqref{eq-V-splitting} so that we have linear isomorphisms
\begin{align*}
\gr_{-\frac{3}{2}} (\mfW)  & \cong \mfV_{-\frac{1}{2}} \otimes \left( \wedge^2 \mfV_{-\frac{1}{2}} \right) \, , & \gr_{-\frac{1}{2}} (\mfW) & \cong \mfV_{\frac{1}{2}} \otimes \left( \wedge^2 \mfV_{-\frac{1}{2}} \right) \, ,
\end{align*}
between $\prb$-modules and $\g_0$-modules. That each of these $\g_0$-modules splits into irreducibles, i.e.\ $\breve{\mfW}_{-\tfrac{3}{2}}^0 \oplus \breve{\mfW}_{-\tfrac{3}{2}}^1 \cong \wedge^3 \mfV_{-\frac{1}{2}} \oplus \left( \mfV_{-\frac{1}{2}} \circledcirc \left( \wedge^2 \mfV_{-\frac{1}{2}} \right) \right)$ and $\breve{\mfW}_{-\tfrac{1}{2}}^0 \oplus \breve{\mfW}_{-\tfrac{1}{2}}^1 \cong \mfV_{-\frac{1}{2}} \oplus \left( \mfV_{\frac{1}{2}} \circledcirc \left( \wedge^2 \mfV_{-\frac{1}{2}} \right) \right)$ respectively, is clear.

Let us be a bit more explicit by viewing an element of $\mfW^{-\frac{3}{2}}$ as an element of $\mfV^{-\frac{1}{2}} \otimes \left( \wedge^2 \mfS^{\frac{m-2}{4}} \right)$, i.e.\ of the form $\Gamma \ind{_{abc}} \xi \ind*{^b^B} \xi \ind*{^c^C}$ or $\Gamma \ind{_{abc}} \xi \ind*{^{bc}^{B'}} \pmod {\alpha_a \xi^{B'}}$ where $\Gamma \ind{_{abc}} = \Gamma \ind{_{a[bc]}}$ lies in the $\g$-module $\mfV \otimes \g$. This means that $\mfW^{-\frac{1}{2}} = \{ \Gamma \ind{_{abc}} \xi \ind*{^b^B} \xi \ind*{^c^C} \in \mfW^{-\frac{3}{2}} : \Gamma \ind{_{abc}} \xi \ind*{^a^A} \xi \ind*{^b^B} \xi \ind*{^c^C} = 0 \}$. To describe elements of the $\g_0$-modules $\breve{\mfW}_i^j$ we write
\begin{align}
\begin{aligned}\label{eq-connection1form}
 \Gamma \ind{_{abc}} \xi \ind*{^b^B} \xi \ind*{^c^C} & = \eta \ind*{_a_A} \Gamma \ind{^{ABC}} + \xi \ind*{_a^A}  \Gamma \ind{_A^{BC}} \, , \\
  \Gamma \ind{_{abc}} \xi \ind*{^{bc}^{D'}}
 & = \left( \eta \ind*{_a_A} \Gamma \ind{^{ABC}} + \xi \ind*{_a^A}  \Gamma \ind{_A^{BC}} \right) \eta \ind*{_c_C} \gamma \ind{^c_B^{D'}} 
+ 2\,  \left( \xi \ind*{_a^A} \Gamma \ind{_{AC}^C} + \eta \ind*{_a_A}  \Gamma \ind{^A_C^C} \right) \xi \ind*{^{D'}}  \, , \\
  \gamma \ind{^a_{D'}^A} \Gamma \ind{_{abc}} \xi \ind*{^{bc}^{D'}}
 & = \Gamma \ind{^{EBC}} \eta \ind*{_b_B} \eta \ind*{_c_C} \gamma \ind{^{bc}_E^A} + 4 \Gamma \ind{_C^{CA}} 
+ 2 \,  \Gamma \ind{^A_C^C} \, ,
\end{aligned}
\end{align}
where $\Gamma \ind{_A^{BC}} = \Gamma \ind{_A^{[BC]}}$, $\Gamma \ind{^{ABC}} = \Gamma \ind{^{A[BC]}}$, $\Gamma \ind{_A_B^C}$ and $\Gamma \ind{^A_B^C}$ are all elements of tensor products of $\mfS_{\frac{m-2}{4}}$ and $\mfS_{-\frac{m-2}{4}}$ --- see appendix \ref{sec-g0-bases}. As before, $\xi^{A'}$ and $\eta_{A'}$ satisfy $\xi^{A'} \eta_{A'} = - \frac{1}{2}$, and we recall that $I \ind*{_A^B} := \xi \ind*{^a^B} \eta \ind*{_a_A}$ is the identity map on $\mfS_{\frac{m-2}{4}}$. In particular, the irreducible $\g_0$-components of an element of $\mfW$ are determined by
\begin{align}
\begin{aligned}\label{eq-Wirred}
\Gamma \ind{^{[ABC]}} & \in \breve{\mfW}_{-\frac{3}{2}}^0 \, , &
\Gamma \ind{^{(AB)C}} & \in \breve{\mfW}_{-\frac{3}{2}}^1 \, , \\
\Gamma \ind{_A^{AC}} & \in \breve{\mfW}_{-\frac{1}{2}}^0 \, , & 
\Gamma \ind{_A^{BC}} - \frac{2}{m-1} I \ind*{_A^{\lb{B}|}} \Gamma \ind{_D^{D|\rb{C}}} & \in \breve{\mfW}_{-\frac{1}{2}}^1 \, .
\end{aligned}
\end{align}
Using \eqref{eq-connection1form} and \eqref{eq-Wirred}, it is then straightforward to check that the $\prb$-modules defined by the kernels of the maps ${}^\mfW _\xi \Pi_i^k$ are related to $\mfW_i^j$ as shown by \eqref{eq-W-ker}.

To obtain the diagram encoding the full $\prb$-invariance, we must also examine the action of the nilpotent part $\g_1$ of $\prb$ on each of these irreducible $\g_0$-modules. This can be checked by a direct computation or by noting that $ \ker {}^\mfW _\xi \Pi_i^j \subset \ker {}^\mfW _\xi \Pi_{i-1}^k$ for suitable $i$, $j$ and $k$.

Finally, extra care must be taken when $m=3$ where $\wedge^3 \mfV_{\pm\frac{1}{2}}$ are one-dimensional. We can  realise $\g_1$ as the pairing of $\mfV^{-\frac{1}{2}}/\mfV^{\frac{1}{2}}$ and $\wedge^3 \mfV_{\frac{1}{2}}$: any element of $\g_1$ can be written in the form $\phi_{ab} = \frac{1}{2} \varepsilon_{abc} \phi^c$ for some vector $\phi^c \in \mfV_{-\frac{1}{2}}$, where $\varepsilon_{abc} \in \wedge^3 \mfV_{\frac{1}{2}}$. It then follows that $\g_1 \cdot \wedge^3 \mfV_{-\frac{1}{2}} \subset \mfV_{-\frac{1}{2}}$. This also explains why we have distinguished the cases $m=3$ and $m>3$ in the definition of the map ${}^\mfW _\xi \Pi_{-\frac{1}{2}}^1$. One may also use the identity
 \begin{align*}
\xi \ind*{^a^{\lb{A}}} \Gamma \ind{_{abc}} \xi \ind*{^b^B} \xi \ind*{^c^{\rb{C}}} & = - \frac{2}{3} \xi \ind*{^a^E} \Gamma \ind{_{abc}} \xi \ind*{^{bc}_E} \varepsilon \ind{^{ABC}} = - \frac{1}{6} \gamma \ind{^a_{EF}} \Gamma \ind{_{abc}} \xi \ind*{^{bc}^E} \xi \ind*{_b^F} \varepsilon \ind{^{ABC}} \, ,
 \end{align*}
where $\varepsilon \ind{^{ABC}} := \frac{1}{2} \gamma \ind{^a^{AB}} \gamma \ind{_a^{CD}} \xi_D$ is completely skew-symmetric.
\end{proof}

\section{Decomposition of the curvature}\label{sec-class-curvature}
As before, we assume $m>2$. Consider the following $\g$-modules:
\begin{displaymath}
{\renewcommand{\arraystretch}{1.5}
\begin{array}{||c|c|c||}
\hline
\text{$\g$-mod} & \text{Dimension} & \text{Description} \\
\hline
\mfF  & (2m-1)(m+1) & \{ \Phi_{ab} \in \otimes^2 \mfV^* : \Phi_{ab} = \Phi_{(ab)} \, , \Phi \ind{_c^c} = 0 \} \\[.5em]
\hline
\mfA &  \frac{8}{3}m(m+1)(m-1) & \{ A_{abc} \in \otimes^3 \mfV : A_{abc} = A_{a[bc]} \, , A_{[abc]} = 0 \, , A \ind{^a_{ac}} = 0 \} \\[.5em]
\hline
\mfC &  \frac{1}{3}m(m+1)(2m+1)(2m-3) & \{ C_{abcd} \in \otimes^4 \mfV : C_{abcd} = C_{[ab][cd]} \, , C_{[abc]d} = 0 \, , C \ind{^a_{bad}} = 0 \} \\[.5em]
\hline
\end{array}}
\end{displaymath}
These modules are to be interpreted as the spaces of irreducible curvature tensors of the Levi-Civita connection at a point, more precisely, of the tracefree Ricci tensors, Cotton-York tensors and Weyl tensors.

We shall give $\prb$-invariant decompositions of these modules, where $\prb$ is the stabiliser of a projective pure spinor $[\xi^{A'}]$ in $\g$. We state the results without proofs, which essentially follow from the discussion of section \ref{sec-gen-fil}, and arguments similar to the proof of Proposition \ref{prop-intrinsic_torsion}. Details can be worked out using the material contained in appendix \ref{sec-spinor-descript}, in particular, the bases for the $\g_0$-modules and the multilinear maps referred to in the following propositions. Notation \ref{nota-prop} applies.

\subsection{Decomposition of the space of the tracefree Ricci tensors}
\begin{prop}\label{prop-main_Ricci}
The space $\mfF$ of tensors of tracefree symmetric $2$-tensors admits a filtration
\begin{align*}
\{ 0 \} =: \mfF^2 \subset \mfF^1 \subset \mfF^0 \subset \mfF^{-1} := \mfF \, ,
\end{align*}
of $\prb$-modules
\begin{align*}
\mfF^i & = \{ \Phi_{ab} \in \mfF : {}^\mfF _\xi \Pi_{i-1}^0 (\Phi) = 0 \} \, , & i = 0, 1,
\end{align*}
where the maps ${}^\mfF _\xi \Pi_i^0$ are defined in appendix \ref{sec-projection}.

 Further, each $\prb$-module $\mfF^i/\mfF^{i+1}$ is an irreducible $\prb$-module  as described below:
 \vspace{-6.5mm}
\begin{center}
\begin{minipage}[b]{0.45\linewidth}\centering
\begin{displaymath}
{\renewcommand{\arraystretch}{1.5}
\begin{array}{||c|c|c||}
\hline
\text{$\prb$-mod} & \text{$\g_0$-mod} & \text{Dimension} \\
\hline
\mfF_0^0 & \mfV_{\frac{1}{2}} \circledcirc \mfV_{-\frac{1}{2}}  & m^2-1 \\
\hline
\end{array}}
\end{displaymath}
\end{minipage}
\begin{minipage}[b]{0.45\linewidth}\centering
\begin{displaymath}
{\renewcommand{\arraystretch}{1.5}
\begin{array}{||c|c|c||}
\hline
\text{$\prb$-mod} & \text{$\g_0$-mod} & \text{Dimension} \\
\hline
\mfF_{\pm1}^0 & \mfV_{\pm\frac{1}{2}} \circledcirc \mfV_{\pm\frac{1}{2}}  & \frac{1}{2}m (m+1) \\
\hline
\end{array}}
\end{displaymath}
\end{minipage}
\end{center}
\end{prop}

\subsection{Decomposition of the space of Cotton-York tensors}
\begin{prop}\label{prop-main_CY}
The space $\mfA$ of tensors with Cotton-York symmetries admits a filtration
\begin{align*}
 \{ 0 \} =: \mfA^{\frac{5}{2}} \subset \mfA^{\frac{3}{2}} \subset \mfA^{\frac{1}{2}} \subset \mfA^{-\frac{1}{2}} \subset \mfA^{-\frac{3}{2}} = \mfA \, ,
\end{align*}
of $\prb$-modules
\begin{align*}
\mfA^i & = \{ A_{abc} \in \mfA: {}^\mfA _\xi \Pi_{i-1}^k (A) = 0 \, , \, \, \mbox{for all $k$} \} \, , & i=-\frac{1}{2},\frac{1}{2},\frac{3}{2} \, ,
\end{align*}
where the maps ${}^\mfA _\xi \Pi_i^k$ are defined in appendix \ref{sec-projection}.

The associated graded $\prb$-module $\gr(\mfA) = \bigoplus_{i=-\frac{3}{2}}^{\frac{3}{2}} \gr_i (\mfA)$, where $\gr_i (\mfA) := \mfA^i/\mfA^{i-1}$, splits into a direct sum
\begin{align*}
 \gr_{\pm\frac{3}{2}} (\mfA) & \cong \mfA_{\pm\frac{3}{2}}^0 \, , &
 \gr_{\pm\frac{1}{2}} (\mfA) & \cong \mfA_{\pm\frac{1}{2}}^0 \oplus \mfA_{\pm\frac{1}{2}}^1 \oplus \mfA_{\pm\frac{1}{2}}^2 \, , 
\end{align*}
of irreducible $\prb$-modules  as described below:
\vspace{-6.5mm}
\begin{center}
\begin{minipage}[b]{0.45\linewidth}\centering
\begin{displaymath}
{\renewcommand{\arraystretch}{1.8}
\begin{array}{||c|c|c||}
\hline
\text{$\prb$-mod} & \text{$\g_0$-mod} & \text{Dimension} \\
\hline
\mfA_{\pm\frac{3}{2}}^0 & \mfV_{\pm\frac{1}{2}} \circledcirc \g_{\pm1} & \frac{1}{3}m (m^2-1) \\
\hline
\mfA_{\pm\frac{1}{2}}^0 & \mfV_{\pm\frac{1}{2}} \circledcirc \mfz_0 & m \\
\mfA_{\pm\frac{1}{2}}^1 & \mfV_{\mp\frac{1}{2}} \circledcirc \g_{\pm1} & \frac{1}{2}m (m-2)(m+1) \\
\mfA_{\pm\frac{1}{2}}^2 & \mfV_{\pm\frac{1}{2}} \circledcirc \slie_0 & \frac{1}{2}m (m+2)(m-1) \\
\hline
\end{array}}
\end{displaymath}
\end{minipage}
\end{center}
Further,
\begin{align*}
 \mfA_i^j & = \{ A_{abc} \in \mfA^i: {}^\mfA _\xi \Pi_i^k (A) = 0 \, , \, \, \mbox{for all $k \neq j$} \} /\mfA^{i+1} \, , & \mbox{for $|i|=\frac{1}{2}$.}
\end{align*}

Finally, the $\prb$-module $\gr(\mfA)$ can be expressed by means of the directed graph
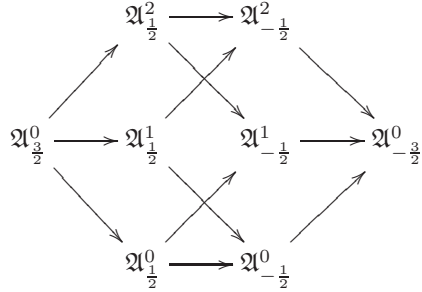
\begin{equation*}
 \xymatrix@R=1em{
	& \mfA_{\frac{1}{2}}^2 \ar[r] \ar[ddr]	& \mfA_{-\frac{1}{2}}^2 \ar[ddr] & \\
       &  &  & &  \\
\mfA_{\frac{3}{2}}^0 \ar[uur] \ar[r] \ar[ddr]	& \mfA_{\frac{1}{2}}^1 \ar[uur]  \ar[ddr]	& \mfA_{-\frac{1}{2}}^1 \ar[r] &  \mfA_{-\frac{3}{2}}^0 \\
	  &  &  & & \\
       & \mfA_{\frac{1}{2}}^0 \ar[r] \ar[uur]	& \mfA_{-\frac{1}{2}}^0 \ar[uur] &  }
\end{equation*}
where an arrow from  $\mfA_i^j$ to $\mfA_{i-1}^k$ for some $i,j,k$ implies that $\breve{\mfA}_i^j \subset \g_1 \cdot \breve{\mfA}_{i-1}^k$ for any choice of irreducible $\g_0$-modules $\breve{\mfA}_i^j$ and $\breve{\mfA}_{i-1}^k$ isomorphic to $\mfA_i^j$ and $\mfA_{i-1}^k$ respectively, or equivalently that $ \ker {}^\mfA _\xi \Pi_i^j \subset \ker {}^\mfA _\xi \Pi_{i-1}^k$.
\end{prop}

\subsection{Decomposition of the space of Weyl tensors}
\begin{prop}\label{prop-main_Weyl}
The space $\mfC$ of tensors with Weyl symmetries admits of a filtration
\begin{align}\label{eq-Weyl_filtration}
 \{ 0 \} =: \mfC^3 \subset \mfC^2 \subset \mfC^1 \subset \mfC^0 \subset \mfC^{-1} \subset \mfC^{-2} := \mfC \, ,
\end{align}
of $\prb$-modules
\begin{align*}
  \mfC^i & = \{ C \in \mfC : {}^\mfC _\xi \Pi_{i-1}^k (C) = 0 \, , \, \, \mbox{for all $k$} \} \, , & i=-1,0,1,2,
\end{align*}
where the maps ${}^\mfC _\xi \Pi_i^k$ are defined in appendix \ref{sec-projection}.

The associated graded $\prb$-module $\gr(\mfC) = \bigoplus_{i=-2}^2 \gr_i (\mfC)$, where $\gr_i (\mfC) := \mfC^i/\mfC^{i-1}$, splits into a direct sum
\begin{align*}
 \gr_{\pm2} (\mfC) & \cong \mfC_{\pm2}^0 \, , & 
 \gr_{\pm1} (\mfC) & \cong \mfC_{\pm1}^0 \oplus \mfC_{\pm1}^1 \, , &
 \gr_0 (\mfC) & \cong \mfC_0^0 \oplus \mfC_0^1 \oplus \mfC_0^2 \oplus \mfC_0^3 \, ,
\end{align*}
of irreducible $\prb$-modules  as described below:
\vspace{-6.5mm}
\begin{center}
\begin{minipage}[b]{0.45\linewidth}\centering
\begin{displaymath}
{\renewcommand{\arraystretch}{1.5}
\begin{array}{||c|c|c||}
\hline
\text{$\prb$-mod} & \text{$\g_0$-mod} & \text{Dimension} \\
\hline
\mfC_{\pm2}^0 & \g_{\pm1} \circledcirc \g_{\pm1} & \frac{1}{12}m^2 (m^2-1) \\
\hline
\mfC_{\pm1}^0 & \g_{\pm1} \circledcirc \mfz_0 & \frac{1}{2}m (m-1) \\
\mfC_{\pm1}^1 & \g_{\pm1} \circledcirc \slie_0 & \frac{1}{3} m^2 (m^2 - 4) \\
\hline
\end{array}}
\end{displaymath}
\end{minipage}
\begin{minipage}[b]{0.45\linewidth}\centering
\begin{displaymath}
{\renewcommand{\arraystretch}{1.5}
\begin{array}{||c|c|c||}
\hline
\text{$\prb$-mod} & \text{$\g_0$-mod} & \text{Dimension} \\
\hline
\mfC_0^0 & \mfz_0 \circledcirc \mfz_0 & 1 \\
\mfC_0^1 & \slie_0 \circledcirc \mfz_0 & m^2 - 1 \\
\mfC_0^2 & \g_1 \circledcirc \g_{-1} & \frac{1}{4}m^2 (m+1) (m-3) \\
\mfC_0^3 & \slie_0 \circledcirc \slie_0 & \frac{1}{4}m^2 (m-1) (m+3) \\
\hline
\end{array}}
\end{displaymath}
\end{minipage}
\end{center}
with the proviso that $\mfC_0^2$ does not occur when $m=3$. Further
\begin{align*}
 \mfC_i^j & = \{ C \in \mfC^i : {}^\mfC _\xi \Pi_i^k (C) = 0 \, , \mbox{for all $k \neq j$} \} / \mfC^{i+1} \, , & \mbox{for $|i| \leq 1$.}
\end{align*}

Finally, the $\prb$-module $\gr(\mfC)$ can be expressed by means of the directed graph
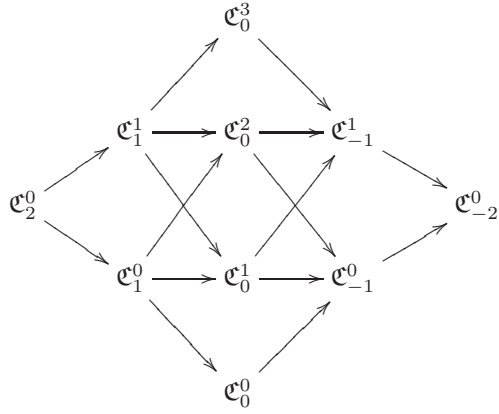
\begin{equation*}
 \xymatrix@R=1em{
	&	& \mfC_0^3 \ar[ddr] & & \\
       & 	&  & &  \\
	& \mfC_1^1 \ar[uur] \ar[r] \ar[ddr]  & \mfC_0^2 \ar[r] \ar[ddr] & \mfC_{-1}^1 \ar[dr] &  \\
\mfC_2^0 \ar[ur] \ar[dr] & 	&  &  & \mfC_{-2}^0 \\
	  & \mfC_1^0 \ar[ddr] \ar[r] \ar[uur] & \mfC_0^1  \ar[uur] \ar[r] & \mfC_{-1}^0 \ar[ur] & \\
       & &  & &  \\
&	& \mfC_0^0  \ar[uur] & & }
\end{equation*}
where an arrow from  $\mfC_i^j$ to $\mfC_{i-1}^k$ for some $i,j,k$ implies that $\breve{\mfC}_i^j \subset \g_1 \cdot \breve{\mfC}_{i-1}^k$ for any choice of irreducible $\g_0$-modules $\breve{\mfC}_i^j$ and $\breve{\mfC}_{i-1}^k$ isomorphic to $\mfC_i^j$ and $\mfC_{i-1}^k$ respectively, or equivalently that $ \ker {}^\mfC _\xi \Pi_i^j \subset \ker {}^\mfC _\xi \Pi_{i-1}^k$.
\end{prop}

\section{Differential geometry of pure spinor fields}\label{sec-geometry}
Throughout this section, $(\mcM , g )$ will denote an $n$-dimensional complex Riemannian manifold, where $n=2m$, i.e.\ a complex manifold $\mcM$ equipped with a global non-degenerate holomorphic section $g_{ab}$ of $\odot^2 \Tgt^* \mcM$, where $\Tgt^* \mcM$ is the holomorphic cotangent bundle of $\mcM$. We also assume that $(\mcM,g)$ is equipped with a global holomorphic volume element and a spin structure. These data are equivalent to a holomorphic reduction of the structure group of the frame bundle $\mathrm{F} \mcM$ of $\mcM$ to $G := \Spin(2m,\C)$, with the Lie algebra $\g$. Holomorphic vector bundles over $\mcM$ can be constructed in terms of finite representations of $G$ or $\g$ in the standard way \cites{Salamon1989,vCap2009}. For instance, if $\mfV$ is the standard representation of $G$, then the holomorphic tangent bundle is simply $\Tgt \mcM := \mathrm{F} \mcM \times_G \mfV$, and holomorphic sections of $\Tgt \mcM$ can be viewed as equivariant holomorphic functions on $\mathrm{F} \mcM$ taking values in $\mfV$. Similarly, the spinor bundle, the chiral positive and negative spinor bundles, $\mcS$, $\mcS^+$ and $\mcS^-$ arise from the spinor representations $\mfS$, $\mfS^+$ and $\mfS^-$ of section \ref{sec-algebra} respectively.

The unique torsion-free metric-compatible holomorphic Levi-Civita connection and its associated covariant derivative on $\mcM$ will both be denoted by $\nabla_a$. Recall that for any other metric-compatible holomorphic connection $\partial_a$, the difference between $\nabla_a$ and $\partial_a$ is given by
\begin{align}\label{eq-LCconnection}
 \nabla_a V^b & = \partial_a V^b + \Gamma \ind{_{ac}^b} V^c \, ,
\end{align}
for any holomorphic vector field $V^a$, for some holomorphic tensor field $\Gamma_{abc}= \Gamma_{a[bc]}$. For instance, if $\partial_a$ preserves an orthonormal frame, then $\Gamma_{abc}$ can be identified with the components of the Levi-Civita connection $1$-form in that frame. The torsion of $\partial_a$ equals $2 \, \Gamma \ind{_{[ab]}^c}$. The Riemann tensor of $\nabla_a$ is given by
\begin{align*}
 2 \, \nabla \ind{_{\lb{a}}} \nabla \ind{_{\rb{b}}} V \ind{^d} & = R \ind{_{abc}^d} V \ind{^c} \, ,
\end{align*}
for any holomorphic vector field $V \ind{^a}$, and satisfies the Bianchi identity
\begin{align}\label{eq-Bianchi_id}
 \nabla \ind{_{\lb{a}}} R \ind{_{b \rb{c}de}} & = 0 \, .
\end{align}
The Riemann tensor splits into $\OO(2m,\C)$-irreducible components as
\begin{align} \label{eq-Riem_decomposition}
 R_{abcd} & = C_{abcd} + \frac{4}{n-2} \Phi_{\lb{c}|\lb{a}} g_{\rb{b}|\rb{d}} + \frac{2}{n(n-1)} R g_{c\lb{a}} g_{\rb{b} d} \, .
\end{align}
where $C \ind{_{abcd}}$ is the Weyl tensor, $\Phi \ind{_{ab}}$ the tracefree part of the Ricci tensor $R \ind{_{ab}}:= R \ind{_{acb}^c}$, and $R := R \ind{_a^a}$ the Ricci scalar. For $m>2$, this decomposition is also $\SO(2m,\C)$- and $G$-irreducible. When $m=2$, the Weyl tensor splits into a self-dual part and an anti-self-dual part, each $\SL(2,\C)$-irreducible. 

Sections of $\mcS^+$ and $\mcS^-$ will be denoted in the obvious way by means of the abstract index notation of section \ref{sec-algebra}, eg by $\xi^{A'}$ and $\zeta^A$ and similarly for their dual. The spin connection on $\mcS$, $\mcS^+$ and $\mcS^-$ can be constructed canonically as a lift of the Levi-Civita connection, and will also be denoted $\nabla_a$. It has the property of preserving the Clifford module structure of $\mcS$ in the sense that
\begin{align*}
\nabla_a ( V^b \gamma \ind{_b _{A'}^B} \xi^{A'} ) = ( \nabla_a V^b ) \gamma \ind{_b _A^{B'}} + V^b \gamma \ind{_b _{A'}^B} \nabla_a \xi \ind*{^{A'}} \, , 
\end{align*}
for any holomorphic vector field $V^a$ and positive spinor field $\xi^{A'}$, and similarly for the other spinor bundles. Lifting any other metric-compatible holomorphic connection $\partial_a$ to $\mcS$, we have, with reference to \eqref{eq-LCconnection},
\begin{align}\label{eq-LC-spin-connection}
 \nabla \ind{_a} \xi \ind*{^{A'}} & = \partial \ind{_a} \xi \ind*{^{A'}} - \frac{1}{4} \Gamma \ind{_{abc}} \gamma \ind{^{bc}_{B'}^{A'}} \xi \ind*{^{B'}} \, ,
\end{align}
for any holomorphic spinor field $\xi^{A'}$. Finally, the curvature of the spin connection is given by
\begin{align*}
 2 \, \nabla \ind{_{\lb{a}}} \nabla \ind{_{\rb{b}}} \xi^{A'} & = - \frac{1}{4} R \ind{_{abcd}} \gamma \ind{^{cd} _{B'}^{A'}} \xi^{B'} \, ,
\end{align*}
for any holomorphic spinor field $\xi^{A'}$, and similarly for spinors of other types.

\begin{nota}
As in sections \ref{sec-algebra}, \ref{sec-class-curvature} and \ref{sec-alg_intrinsic_torsion}, we shall use the short-hand notation
\begin{align*}
 \xi \ind*{_a^A} & := \xi \ind*{^{B'}} \gamma \ind{_a_{B'}^A} \, , &
 \xi \ind*{_{ab}^{A'}} & := \xi \ind*{^{B'}} \gamma \ind{_{ab}_{B'}^{A'}} \, , &
 \eta \ind*{_a_A} & := \gamma \ind{_a_A^{B'}} \eta \ind*{_{B'}} \, , &
 \eta \ind*{_{ab}_{A'}} & := \gamma \ind{_{ab}_{A'}^{B'}} \eta \ind*{_{B'}} \, , \\
 \zeta \ind*{_a^{A'}} & := \zeta \ind*{^B} \gamma \ind{_a_B^{A'}} \, , &
 \zeta \ind*{_{ab}^A} & := \zeta \ind*{^B} \gamma \ind{_{ab}_B^A} \, , & \ldots
\end{align*}
and so on, for spinors $\xi^{A'}$, $\eta_{A'}$, $\zeta^A$.
\end{nota}

\begin{ass}
Throughout this section, we shall assume that our tensor and spinor fields depend holomorphically on $\mcM$, and $\Gamma( \cdot)$ will denote the space of holomorphic sections of a holomorphic fiber bundle. The reader will sometimes be reminded of this assumption for clarity. The application of this work to real pseudo-Riemannian manifolds is given in section \ref{sec-real-geom}.

Further, unless otherwise stated, we shall assume $m>2$ for definiteness, although many of the results here specialise to the case $m=2$ too. Appendix \ref{sec-spin-calculus4} contains a brief review of this case.

Finally, let us emphasize that our results will be essentially local.
\end{ass}

\subsection{Projective pure spinor fields}
We first make the following definition.
\begin{defn}
An \emph{almost null structure} $\mcN$ on $(\mcM,g)$ is a rank-$m$ distribution that is totally null, i.e.\ $g(v,w)=0$ for all sections $v$, $w$ of $\mcN$. An almost null structure is \emph{(anti-)self-dual} if it is annihilated by an (anti-)self-dual $m$-form.
\end{defn}
A self-dual, respectively, anti-self-dual, almost null structure will also be referred to as an $\alpha$-plane, respectively, $\beta$-plane, distribution. We shall denote the bundle of all self-dual, respectively, anti-self-dual, almost null structures on $(\mcM,g)$ by $\Gr_m^+( \Tgt \mcM,g)$, respectively $\Gr_m^-( \Tgt \mcM,g)$. These bundles have fibers isomorphic to the $\frac{1}{2}m(m-1)$-dimensional family $\Gr_m^+( \Tgt_p \mcM,g)$ of $\alpha$-planes, respectively $\Gr_m^-( \Tgt_p \mcM,g)$ of $\beta$-planes, in $\Tgt_p \mcM$ at any point $p$.

The existence of an almost null structure does not require a spin structure. But the latter allows us to identify almost null structures  with \emph{pure spinor fields} up to scale, i.e.\ spinor fields satisfying condition \eqref{eq-purity_cond} or \eqref{eq-purity_cond_Cartan} at every point. For, by Proposition \ref{prop-fundamental}, a totally null  $m$-dimensional vector subspace of the tangent space at a point can be identified with a pure spinor up to scale. Thus, we shall also identify the bundles $\Gr_m^\pm( \Tgt \mcM,g)$ with the bundles of projective pure spinors of either chirality.

Now, let $[\xi^{A'}]$ be a holomorphic projective pure spinor field, i.e.\ a holomorphic section of $\Gr_m^+( \Tgt \mcM,g)$, and denote by $\mcN_\xi$ its associated almost null structure, which can also be assumed to be holomorphic. Then the structure group of $\mathrm{F} \mcM$ is reduced to the stabiliser $P$ of $[\xi^{A'}]$ as described in section \ref{sec-algebra}, and we obtain filtrations of vector bundles, together with their associated graded vector bundles, constructed from finite representations of $P$ or its Lie algebra $\prb$. For instance, the filtration $\prb$-modules $\{ \mfC^i \}$ on the space $\mfC$ of tensors with Weyl symmetries gives rise to a filtration of vector subbundles $\mcC^i$ over $\mcM$, where $\mcC^i := \mathrm{F} \mcM \times_P \mfC^i$, and so does the story go for the associated graded $\prb$-modules $\gr_i(\mfC)$, its irreducible modules $\mfC_i^j$ and the graded $\g_0$-modules $\mfC_i$ in the obvious way and notation \cite{vCap2009}. We shall then recycle the notation of sections  \ref{sec-alg_intrinsic_torsion} and \ref{sec-class-curvature},  and appendix \ref{sec-spinor-descript} in this curved setting as the need arises. We shall characterise the (local) algebraic degeneracy of curvature tensors with respect to $[\xi^{A'}]$ by means of the maps ${}^\mfF _\xi \Pi_i^j$, ${}^\mfA _\xi \Pi_i^j$ and ${}^\mfC _\xi \Pi_i^j$.

\subsubsection{Intrinsic torsion}\label{sec-geom-intrinsic-torsion}
Having singled out a holomorphic projective pure spinor field $[\xi^{A'}]$ on $\mcM$, it remains to characterise the various degrees of `integrability' of the $P$-structure it defines. Following \cite{Salamon1989}, the $P$-structure being integrable to first order, i.e.\ there exists a torsion-free connection compatible with the $P$-structure, is essentially equivalent to $[\xi^{A'}]$ being parallel with respect to the Levi-Civita connection $\nabla_a$, ie
\begin{align}\label{eq-recurrent_spinor_basic}
\nabla_a [ \xi^{A'} ] & = 0 \, , & \mbox{ie} & &  \nabla_a \xi^{B'} & = \alpha \ind{_a} \xi \ind*{^{B'}} \, ,
  \end{align}
for some $1$-form $\alpha \ind{_a}$. Equation \eqref{eq-recurrent_spinor_basic} can be more conveniently expressed as
\begin{align}\label{eq-recurrent_spinor}
 ( \nabla \ind{_a} \xi \ind*{^b^B} ) \xi \ind*{_b^C} & = 0 \, , &  & \mbox{or equivalently,} &  ( \nabla \ind{_a} \xi \ind*{^{[B'}} ) \xi \ind*{^{C']}} & = 0 \,.
\end{align}
which is also equivalent to the Levi-Civita connection being $\prb$-valued.

The obstruction to \eqref{eq-recurrent_spinor} is known as the \emph{intrinsic torsion} or \emph{structure function} of the $P$-structure defined by $[\xi^{A'}]$  \cites{Chern1953,Bernard1960,Salamon1989}. Measuring the extent to which \eqref{eq-recurrent_spinor} fails can be achieved by characterising
\begin{align}\label{eq-intrinsic_torsion}
 ( \nabla \ind{_a} \xi \ind*{^b^B} ) \xi \ind*{_b^C} \quad \in \quad \mfV^{-\frac{1}{2}} \otimes \wedge^2 \mfS^{\frac{m-2}{4}} \, ,
\end{align}
as an element of a $\prb$-submodule of $\mfW = \mfV \otimes \left( \g/\prb \right)$. Here, we have made use of the fact that the connection $1$-form is $\g$-valued, and the pair of skew-symmetric spinor indices of \eqref{eq-intrinsic_torsion} projects out the part of $\g$ not in $\prb$.
This can be made more explicit by choosing a connection $\partial_a$ that preserves $[\xi^{A'}]$ so that \eqref{eq-LC-spin-connection} becomes
\begin{align}\label{eq-nabla-Gamma}
 \nabla \ind{_a} \xi \ind*{^{A'}} & = - \frac{1}{4} \Gamma \ind{_{abc}} \gamma \ind{^{bc}_{B'}^{A'}} \xi \ind*{^{B'}} \pmod{ \alpha_a \xi^{A'} } \, .
\end{align}
In fact, with no loss of generality, we could choose $\partial_a$ to preserve a chosen $\xi^{A'}$. This makes contact with the description of elements of $\mfW$ given in section \ref{sec-alg_intrinsic_torsion}. The expression \eqref{eq-nabla-Gamma} allows us to express the algebraic characterisation of the intrinsic torsion of the $P$-structure of Proposition \ref{prop-intrinsic_torsion} in terms of \eqref{eq-intrinsic_torsion}. This yields the next proposition, and we leave the details of the proof the reader.

\begin{prop} \label{prop-intrinsic_torsion_connection} Let $[\xi^{A'}]$ be a holomorphic projective pure spinor field on $(\mcM,g)$, and let $\Gamma _{abc} \xi^{bB} \xi^{cC} \in \mfW$ be its associated intrinsic torsion. Then, pointwise,
 \begin{itemize}
\item ${}^\mfW _\xi \Pi _{-\frac{3}{2}}^0 (\Gamma) = 0$ if and only if
\begin{align}
( \xi \ind*{^a^{\lb{A}}} \nabla \ind{_a} \xi \ind*{^b^B} ) \xi \ind*{_b^{\rb{C}}} & = 0  \, ; \label{eq-skew-foliating}  
\end{align}
  \item ${}^\mfW _\xi \Pi _{-\frac{3}{2}}^1 (\Gamma) = 0$ if and only if
\begin{align}
( \xi \ind*{^a^{\lp{A}}} \nabla \ind{_a} \xi \ind*{^b^{\rp{B}}} ) \xi \ind*{_b^C} & = 0 \, ; \label{eq-sym-foliating}
\end{align}
\item ${}^\mfW _\xi \Pi _{-\frac{1}{2}}^0 (\Gamma) = 0$ if and only if
\begin{align}
\xi \ind*{^{A'}} \nabla \ind{_b} \xi \ind*{^b^{B}} - \xi \ind*{^b^B} \nabla \ind{_b} \xi \ind*{^{A'}} & = 0 \label{eq-proj-Dirac} \, ;
\end{align}
\item ${}^\mfW _\xi \Pi _{-\frac{1}{2}}^1 (\Gamma) = 0$ if and only if
\begin{align}
 ( \nabla \ind{_a} \xi \ind*{^b^B} ) \xi \ind*{_b^C} + \frac{2}{m-1} \left( \xi \ind*{_a^{\lb{B}}} \nabla \ind{_b} \xi \ind*{^b^{\rb{C}}} + \xi \ind*{^b^{\lb{B}}} \nabla \ind{_b} \xi \ind*{_a^{\rb{C}}} \right) & = 0 \, , & m>3 \, ; \label{eq-proj-twistor} \\
( \nabla \ind{_a} \xi \ind*{^b^B} ) \xi \ind*{_b^C} + \left( \xi \ind*{_a^{\lb{B}}} \nabla \ind{_b} \xi \ind*{^b^{\rb{C}}} + \xi \ind*{^b^{\lb{B}}} \nabla \ind{_b} \xi \ind*{_a^{\rb{C}}} \right) - \frac{2}{3} \xi \ind*{^b^A} ( \nabla \ind{_b} \xi _A ) \gamma \ind{_a^{BC}} & = 0 \, , & m=3 \, ; \label{eq-proj-twistor6a}
\end{align}
where we have made use of the isomorphism $\mfS^+ \cong (\mfS^-)^*$ when $m=3$.
 \end{itemize}

All of these statements are independent of the scaling of $\xi^{A'}$. 
\end{prop}

\begin{rem}
For the case $m=3$, we refer to appendix \ref{sec-spin-calculus6} where conditions \eqref{eq-skew-foliating}, \eqref{eq-sym-foliating}, \eqref{eq-proj-Dirac} and \eqref{eq-proj-twistor6a} are given as conditions \eqref{eq-skew-foliating6}, \eqref{eq-sym-foliating6}, \eqref{eq-proj-Dirac6} and \eqref{eq-proj-twistor6}.
\end{rem}

\subsubsection{Geometric properties}

\begin{defn}
An almost null structure $\mcN$ is said to be \emph{totally geodetic} if $\nabla_X Y \in \Gamma(\mcN)$ for all $X,Y \in \Gamma(\mcN)$.
\end{defn}

Clearly, if $\mcN$ is totally geodetic, it is integrable as a distribution, i.e.\ $[ \Gamma(\mcN) , \Gamma(\mcN) ] \subset \Gamma(\mcN)$. By the Frobenius theorem, $\mcN$ is locally tangent to a foliation by $m$-dimensional complex submanifolds of $\mcM$, each of which is a totally geodetic and totally null. In fact, the converse is also true \cite{Taghavi-Chabert2012}.

\begin{lem}\label{lem-int-geod}
An almost null structure is integrable as a distribution if and only if it is totally geodetic.
\end{lem}

We shall henceforth also refer to an integrable almost null structure as a \emph{totally geodetic null structure}.

\begin{prop}[\cites{Hughston1988}] Let $\mcN_\xi$ be an almost null structure with associated projective pure spinor field $[ \xi^{A'} ]$ on $(\mcM,g)$. Then $\mcN_\xi$ is totally geodetic if and only if $[\xi^{A'}]$ satisfies
\begin{align}\label{eq-foliating_spinor}
 ( \xi \ind*{^a^A} \nabla \ind{_a} \xi \ind*{^b^B} ) \xi \ind*{_b^C} & = 0 \, ,
  &  & \mbox{or equivalently,} &  (  \xi \ind*{^a^A} \nabla \ind{_a} \xi \ind*{^{[B'}} ) \xi \ind*{^{C']}} & = 0 \,.
\end{align}
\end{prop}

\begin{proof}
Using the decomposition \eqref{eq-connection1form} together with \eqref{eq-nabla-Gamma}, we find $ ( \xi \ind*{^a^A} \nabla \ind{_a} \xi \ind*{^b^B} ) \xi \ind*{_b^C} = \Gamma \ind{^{ABC}}$ for some functions $\Gamma \ind{^{ABC}}$ that we can identify with the components of the Levi-Civita connection with respect to a basis for $\mcN_\xi$. But $\mcN_\xi$ being totally geodetic is equivalent to $g \ind{_{ab}} X \ind{^a} Y \ind{^c} \nabla \ind{_c} Z \ind{^b} = 0$, for all $X^a$, $Y^a$, $Z^a$ as shown in \cite{Taghavi-Chabert2012}. Hence the result.
\end{proof}

Equation \eqref{eq-foliating_spinor} also appears (in a slightly different form) in \cite{Lawson1989} in the almost Hermitian setting.

\begin{defn}
We shall refer to a projective pure spinor field $[ \xi^{A'} ]$ satisfying \eqref{eq-foliating_spinor} as \emph{geodetic}.
\end{defn}

\paragraph{Conformal invariance}
With reference to appendix \ref{sec-conformal}, one can prove
\begin{prop}\label{prop-conformal-invariance-spinor}
Conditions \eqref{eq-skew-foliating}, \eqref{eq-sym-foliating} and \eqref{eq-proj-twistor} are conformally invariant.
 
 Suppose further that $[ \xi \ind*{^{A'}} ]$ satisfies \eqref{eq-proj-twistor} and
\begin{align}\label{eq-Lee-spinor}
 \xi \ind*{^{A'}} \nabla \ind{_b} \xi \ind*{^b^{B}} - \xi \ind*{^b^B} \nabla \ind{_b} \xi \ind*{^{A'}} = - (m-1) \xi \ind*{^{A'}} \xi \ind*{^b^B} \nabla \ind{_b} f \, ,
\end{align}
for some holomorphic function $f$. Then there exists a holomorphic conformal rescaling of the metric such that $[ \xi \ind*{^{A'}} ]$ is parallel, i.e.\ it satisfies \eqref{eq-recurrent_spinor}.
\end{prop}

\paragraph{Curvature conditions}
\begin{prop}[\cites{Hughston1988,Taghavi-Chabert2011,Taghavi-Chabert2012}]\label{prop-int_cond_foliating_spinor}
 Let $\xi^{A'}$ be a geodetic pure spinor on $(\mcM,g)$, i.e.\ $\xi^{A'}$ satisfies \eqref{eq-foliating_spinor}, i.e.\ its associated almost null structure $\mcN_\xi$ is totally geodetic (or equivalently, integrable). Then
\begin{align} \label{eq-int_cond_foliating_spinor}
 \xi \ind*{^a ^A} \xi \ind*{^b ^B} \xi \ind*{^c ^C} \xi \ind*{^d ^D} C_{abcd} & = 0 \, , & \mbox{ie} & & {}^\mfC _\xi \Pi_{-2}^0 (C) & = 0 \, .
\end{align}
\end{prop}

\begin{proof}
 We first note that \eqref{eq-foliating_spinor} can be rewritten as
$\xi \ind*{^a ^A} \nabla_a \xi^{B'} = \alpha \ind{^A} \xi \ind*{^{B'}}$ for some $\alpha \ind{^A}$. Differentiating it along $\mcN_\xi$ yields
\begin{align*}
 \alpha \ind{^A} \alpha \ind{^B} \xi \ind*{^{C'}} + \xi \ind*{^a^A} \xi \ind*{^b^B} \nabla \ind{_a} \nabla \ind{_b} \xi \ind*{^{C'}} = (\xi \ind*{^a^A} \nabla \ind{_a} \alpha \ind{^B} ) \xi \ind*{^{C'}} + \alpha \ind{^A} \alpha \ind{^B} \xi \ind*{^{C'}} \, .
\end{align*}
Commuting the derivatives leads to 
\begin{align}\label{eq-Riem-int-cond-foliating}
 - \frac{1}{4} R \ind{_{abcd}} \xi \ind*{^a^A} \xi \ind*{^b^B} \gamma \ind{^{cd}_{D'}^{C'}} \xi \ind*{^{D'}} & = 2 \, ( \xi \ind*{^a^{\lb{A}}}   \nabla \ind{_a} \alpha \ind{^{\rb{B}}} ) \xi \ind*{^{A'}} \, ,
\end{align}
which is equivalent to $\xi \ind*{^a ^A} \xi \ind*{^b ^B} \xi \ind*{^c ^C} \xi \ind*{^d ^D} R_{abcd} = 0$. The decomposition of the Riemann tensor together with the purity condition concludes the proof.
\end{proof}

Closing this section, we give the integrability condition for the existence of a parallel projective pure spinor.
\begin{prop}\label{prop-int_cond_recurrent_spinor}
Let $[\xi^{A'}]$ be a parallel projective pure spinor on $(\mcM,g)$, i.e.\ $\xi^{A'}$ satisfies \eqref{eq-recurrent_spinor}. Then
\begin{align}
\xi \ind*{^a ^A} \xi \ind*{^b ^B} R_{abcd} & = 0 \, , \label{eq-recurrent_spinor_Riemman} \\
\xi \ind*{^a ^A} \xi \ind*{^b ^B} \Phi_{ab} & = 0 \, , & \mbox{ie} & & {}^\mfF _\xi \Pi^0_{-1} (\Phi) & = 0 \label{eq-recurrent_spinor_Ricci} \\
\xi \ind*{^a^A} \xi \ind*{^b^B} \xi \ind*{^c^C}  C \ind{_{abcd}} & = 0 \, , & \mbox{ie} & & {}^\mfC _\xi \Pi_{-1}^0 (C) = {}^\mfC _\xi \Pi_{-1}^1 (C) & = 0 
\label{eq-recurrent_spinor_Weyl3}
\end{align}
and in addition, when $m>3$,
\begin{align}\label{eq-recurrent_spinor_Weyl3>}
 {}^\mfC _\xi \Pi_0^2 (C) = 0 \, .
\end{align}

Further, 
\begin{align*}
R & = 0 & \Longleftrightarrow & & {}^\mfC _\xi \Pi_0^0 (C) = 0 \quad \left( \mbox{ie} \quad \xi \ind*{^{ab}^{A'}} C \ind{_{abcd}} \xi \ind*{^{cd}^{D'}} = 0 \, , \right)
\end{align*}
and in addition, when $m>2$,
\begin{align*}
 {}^\mfF _\xi \Pi^0_0 (\Phi) & = 0 \quad  \left( \mbox{ie} \quad  \xi \ind*{^a ^A} \Phi_{ab} = 0 \right) & \Longleftrightarrow & & {}^\mfC _\xi \Pi_0^1 (C) & = 0 \, .
\end{align*}
\end{prop}

\begin{proof}
Taking a covariant derivative of equation \eqref{eq-recurrent_spinor_basic} and commuting the derivatives yield
\begin{align*}
 - \frac{1}{4} R \ind{_{abcd}} \gamma \ind{^{cd}_{B'}^{A'}} \xi \ind*{^{B'}} & = 2 \, ( \nabla \ind{_{\lb{a}}} \alpha \ind{_{\rb{b}}} ) \xi \ind*{^{A'}} \, ,
\end{align*}
which is equivalent to equation \eqref{eq-recurrent_spinor_Riemman}. Contracting equation \eqref{eq-recurrent_spinor_Riemman} with $\gamma \ind{^{ab}_B^C}$ yields the condition \eqref{eq-recurrent_spinor_Ricci} on the Ricci tensor. Conditions \eqref{eq-recurrent_spinor_Weyl3} and \eqref{eq-recurrent_spinor_Weyl3>} are obtained from the decomposition \eqref{eq-Riem_decomposition}. We find
\begin{align*}
 \xi \ind*{^a^A} C \ind{_{a[bc]d}} \xi \ind*{^d^D} & = \frac{2}{n-2} \, \xi \ind*{_{[b}^{[A}} \Phi \ind{_{c]d}} \xi \ind*{^d^{D]}} + \frac{1}{n(n-1)} \, R \, \xi \ind*{_{[b}^A} \xi \ind*{_{c]}^D}\, , \\
 \xi \ind*{^{ae}^{C'}} C \ind{_{aedb}} \xi \ind*{^d^D} & = 2 \, \frac{n-4}{n-2} \, \xi^{C'} \Phi \ind{_{bd}} \xi \ind*{^d^{D}} + 2 \,  \frac{n-2}{n(n-1)} \, R \, \xi \ind*{_b^D} \xi^{C'} \, , \\
 \xi \ind*{^{ae}^{C'}} C \ind{_{aedf}} \xi \ind*{^{df}^{F'}} & = - 2 \, \frac{n-2}{n-1} \, R \, \xi^{C'} \xi^{F'} \, ,
\end{align*}
and the remaining statements follow immediately from the formulae for ${}^\mfC _\xi \Pi_0^2$, ${}^\mfC _\xi \Pi_0^1$, ${}^\mfC _\xi \Pi_0^0$ of appendix \ref{sec-projection}.
\end{proof}

\begin{rem}
 The purity condition is crucial in deducing conditions \eqref{eq-recurrent_spinor_Weyl3} and \eqref{eq-recurrent_spinor_Weyl3>} on the Weyl tensor. A study of pseudo-Riemannian manifolds admitting more general recurrent spinors was carried out in \cite{Galaev2013}.
\end{rem}

\subsection{Spinorial differential equations}\label{sec-spin-diff-eq}
So far we have only considered spinorial differential equations on \emph{projective} pure spinor fields, i.e.\ differential equations that are invariant under rescalings of $\xi^{A'}$. In this section, we study spinorial differential equations on pure spinors of fixed scales emphasing their relations to the intrinsic torsion of their associated $P$-structures.

\subsubsection{Scale-dependent geodetic spinors}
A scale-dependent variation of the geodetic spinor equation \eqref{eq-foliating_spinor} is given by
\begin{align}\label{eq-strongly_foliating}
 \xi \ind*{^a ^A} \nabla_a \xi^{B'} & = 0 \, ,
\end{align}
on a holomorphic pure spinor field $\xi^{A'}$. Since $\hat{\xi} \ind{^a^B} \hat{\nabla} \ind{_a} \left( \Omega^{-1} \xi \ind*{^{A'}} \right) = \Omega^{-2} \left( \xi \ind*{^a^B} \nabla \ind{_a} \xi \ind*{^{A'}} \right)$, equation \eqref{eq-strongly_foliating} is clearly conformally invariant if and only if the spinor field $\xi \ind*{^{A'}}$ has conformal weight $-1$. Accordingly, the integrability condition for \eqref{eq-strongly_foliating} is expected to be conformally invariant. Indeed, a variation of the proof of Proposition \ref{prop-int_cond_foliating_spinor} with $\alpha_a=0$ leads to
\begin{prop}\label{prop-conf_inv_foliating}
 Let $\xi^{A'}$ be a holomorphic pure spinor field satisfying \eqref{eq-strongly_foliating}. Then ${}^\mfC _\xi \Pi_{-1}^0 (C) = 0$, i.e.\ $C \ind{_{abcd}} \xi \ind*{^a^A} \xi \ind*{^b^B} \xi \ind*{^{cd}^{C'}} = 0$.
\end{prop}

\subsubsection{Parallel pure spinors}
The integrability condition for the existence of a parallel spinor $\xi^{A'}$ is clearly that it annihilates the Riemann tensor, i.e.\ $R_{abcd} \gamma \ind{^{cd}_{B'} ^{A'}} \xi^{B'} = 0$. The assumption that $\xi^{A'}$ is pure allows us to derive more information. To prove the next proposition, set $\alpha_a=0$ in the proof of Proposition \ref{prop-int_cond_recurrent_spinor}.
\begin{prop}\label{prop-parallel}
Let $\xi^{A'}$ be a parallel pure spinor on $(\mcM,g)$, i.e.\ $\xi^{A'}$ satisfies
\begin{align}\label{eq-parallel-spinor}
 \nabla_a \xi^{B'} & = 0 \, .
\end{align}
Then
\begin{align*}
R_{abcd} \xi \ind*{^{cd} ^{D'}} & = 0 \, ,  \\
{}^\mfF _\xi \Pi^0_0 (\Phi) & = 0 \, , & & \mbox{ie} & \Phi_{ab} \xi \ind*{^{b}^B} & = 0 \, , \\
 R & = 0 \, , \\
{}^\mfC _\xi \Pi _1^0 (C) & = 0 \, , & &  \mbox{ie} & C_{abcd} \xi \ind*{^{cd} ^{D'}} & = 0 \, . 
\end{align*}
\end{prop}

\subsubsection{Null zero-rest-mass fields}
Conditions weaker than \eqref{eq-parallel-spinor} can be obtained by decomposing the covariant derivative of a spinor field into two irreducible parts under $\Spin(2m,\C)$. The smaller of these is known as the \emph{(Weyl-)Dirac equation}
\begin{align}\label{eq-Weyl-Dirac}
 \gamma \ind{^a _{A'}^B} \nabla_a \xi^{A'} & = 0 \, ,
\end{align}
on a holomorphic spinor field $\xi^{A'}$. It admits a generalisation to irreducible symmetric spinor fields of higher valence, i.e.\ spinor fields\footnote{From a representational theoretic viewpoint, $\phi^{A_1'A_2' \ldots A_k'}$ lies in the $k$-fold Cartan product of $\mfS^+$.} $\phi^{A_1'A_2' \ldots A_k'} = \phi^{(A_1'A_2' \ldots A_k')}$ satisfying $\gamma \ind{^a _{A_1'}^B} \gamma \ind{_a_{A_2'}^C} \phi^{A_1'A_2' \ldots A_k'} = 0$, which is known as the \emph{zero-rest-mass (zrm) field equation} \cite{Hughston1988},
\begin{align}\label{eq-zrm}
 \gamma \ind{^a _{A_1'}^B} \nabla_a \phi^{A_1'A_2' \ldots A_k'} & = 0 \, .
\end{align}
The case $k=2$ corresponds to a closed, and thus coclosed, self-dual $m$-form. Equation \eqref{eq-zrm} is conformally invariant provided that its solutions $\phi^{A_1'A_2' \ldots A_k'}$ are of conformal weight $-m-k+1$. For $k>2$, there is a strong integrability condition on $\phi^{A_1'A_2' \ldots A_k'}$ given by the following lemma.
\begin{lem}
 For $k>2$, let $\phi \ind{^{A_1'A_2' \ldots A_k'}}$ be a solution of the zrm field equation \eqref{eq-zrm} on $(\mcM,g)$. Then
\begin{align*}
 \gamma \ind{^a _{C_1'}^A} \gamma \ind{^b _{C_2'}^B} C \ind{_{abcd}} \gamma \ind{^{cd} _{D'}^{(C_3'}} \phi \ind{^{C_4' \ldots C_k') C_1' C_2' D'}} & = 0 \, .
\end{align*}
\end{lem}

\begin{proof}
We compute
\begin{align*}
 0 & = 2 \, \gamma \ind{^a _{C_1'}^{\lb{A}}} \gamma \ind{^b _{C_2'}^{\rb{B}}} \nabla_a \nabla_b \phi^{C_1'C_2' \ldots C_k'} \\
   & = - \frac{1}{2} \gamma \ind{^a _{C_1'}^A} \gamma \ind{^b _{C_2'}^B} R \ind{_{abcd}} \gamma \ind{^{cd} _{D'}^{C_1'}} \phi \ind{^{C_2' \ldots C_k'D'}} - \frac{k-2}{4} \gamma \ind{^a _{C_1'}^A} \gamma \ind{^b _{C_2'}^B} R \ind{_{abcd}} \gamma \ind{^{cd} _{D'}^{(C_3'}} \phi \ind{^{C_4' \ldots C_k') C_1' C_2' D'}} \\
   & = \Phi \ind{_{bc}} \gamma \ind{^b _{C_2'}^{\lb{A}}} \gamma \ind{^c _{D'}^{\rb{B}}} \phi \ind{^{C_2' \ldots C_k'D'}} - \frac{k-2}{4} \gamma \ind{^a _{C_1'}^A} \gamma \ind{^b _{C_2'}^B} C \ind{_{abcd}} \gamma \ind{^{cd} _{D'}^{(C_3'}} \phi \ind{^{C_4' \ldots C_k') C_1' C_2' D'}} \, .
\end{align*}
The first term must vanish by symmetry consideration, which concludes the proof.
\end{proof}

A holomorphic irreducible symmetric spinor as above is said to be \emph{null} if it takes the form
\begin{align*}
 \phi^{A_1' A_2' \ldots A_k'} & = \ee^\psi \xi^{A_1'} \xi^{A_2'} \ldots \xi^{A_k'} \, ,
\end{align*}
for some holomorphic function $\psi$ and holomorphic pure spinor field $\xi^{A'}$. In this case, the integrability condition for the existence of a solution of equation \eqref{eq-zrm} is given by the following
\begin{cor}
 For $k>2$, suppose that $\phi \ind{^{A_1'A_2' \ldots A_k'}} := \ee^\psi \xi^{A_1'} \xi^{A_2'} \ldots \xi^{A_k'} $ is a solution of the zrm field equation \eqref{eq-zrm} on $(\mcM,g)$. Then
 \begin{align}\label{eq-null-zrm-int-cond}
 {}^\mfC _\xi \Pi_{-1}^0 (C) & = 0 \, , & &  \mbox{ie} & \xi \ind*{^a ^A} \xi \ind*{^b ^B} C \ind{_{abcd}} \xi \ind*{^{cd} ^{D'}} & = 0 \, .
 \end{align}
\end{cor}

The relation between pure solutions to the zrm field equation and geodetic spinors was first established by Robinson \cite{Robinson1961} in four dimensions in his study of electromagnetism. It was later generalised by Hughston and Mason \cite{Hughston1988} to even dimensions.

\begin{thm}[\cites{Robinson1961,Hughston1988}]\label{thm-Robinson}
 Let $\xi^{A'}$ be a holomorphic pure spinor field on $(\mcM,g)$ with almost null structure $\mcN_\xi$.
 
 Let $\psi$ be a holomorphic function and suppose that $\phi^{A_1' \ldots A_k'} := \ee^\psi \xi^{A_1'} \xi^{A_2'} \ldots \xi^{A_k'}$ satisfies the zrm field equation \eqref{eq-zrm}. Then $\xi^{A'}$ locally satisfies \eqref{eq-foliating_spinor}, i.e.\ $\xi^{A'}$ is geodetic, i.e.\ $\mcN_\xi$ is totally geodetic.

Conversely, suppose that $\xi^{A'}$ is geodetic, i.e.\ $\mcN_\xi$ is totally geodetic. Then locally there exists a holomorphic function $\psi$ such that $\phi^{A'B'} := \ee^\psi \xi^{A'} \xi^{B'}$ satisfies the zrm field equation \eqref{eq-zrm}. Suppose further that $\xi^{A'}$ satisfies \eqref{eq-null-zrm-int-cond}. Then locally, for every $k>2$, there exists a holomorphic function $\psi$ such that $\phi^{A_1' \ldots A_k'} :=\ee^\psi \xi^{A_1'} \xi^{A_2'} \ldots \xi^{A_k'}$ satisfies the zrm field equation \eqref{eq-zrm}. In both cases, there is the freedom of adding to $\psi$ a holomorphic function constant along the leaves of $\mcN_\xi$.
\end{thm}

\subsubsection{Conformal Killing spinors}
The larger irreducible part of the covariant derivative of a spinor field leads to the \emph{twistor equation}
\begin{align}
 \nabla_a \xi^{A'} + \frac{1}{\sqrt{2}} \gamma \ind{_a _B^{A'}} \zeta^B & = 0 \, , \label{eq-twistor-spinor}
\end{align}
on a holomorphic spinor field $\xi^{A'}$ -- here, \eqref{eq-twistor-spinor} determines $\zeta^B = \frac{\sqrt{2}}{n} \gamma \ind{^a _{A'}^B} \nabla_a \xi^{A'}$. We shall refer to a solution of \eqref{eq-twistor-spinor} as a \emph{conformal Killing spinor} or \emph{twistor spinor}. It is well-known that the twistor equation is overdetermined, and for this reason, it is often more convenient to consider its prolongation \cites{Penrose1986,Baum2010}
\begin{align}
 \nabla \ind{_a} \zeta \ind*{^B} + \frac{1}{\sqrt{2}} \Rho \ind{_{ab}} \gamma \ind{^b_{A'}^B} \xi \ind*{^{A'}} & = 0 \, , \label{eq-twistor-spinor2}
\end{align}
where $\Rho_{ab} := \frac{1}{2-n} \Phi_{ab} - R \frac{1}{2n(n-1)} g_{ab}$ is the \emph{Rho} or \emph{Schouten} tensor (see appendix \ref{sec-conformal}).
We immediately deduce
\begin{align}\label{eq-Dirac-twistor}
 \nabla_b \zeta^{bA'} & = - \frac{1}{2\sqrt{2}(n-1)} R \xi ^{A'} \, .
\end{align}
Equations \eqref{eq-twistor-spinor} and \eqref{eq-twistor-spinor2} are conformally invariant, provided that under a conformal change of metric $\hat{g}_{ab} = \Omega^2 g_{ab}$ for some non-vanishing holomorphic function $\Omega$, $\xi^{A'}$ and $\zeta^A$ transform as
\begin{align}\label{eq-twistor-conf_transf}
 \xi^{A'} &\mapsto \hat{\xi}^{A'} = \xi^{A'} \, , & \zeta^A &\mapsto \hat{\zeta}^A = \Omega^{-1} \left( \zeta^A + \frac{1}{\sqrt{2}} \Upsilon_a \xi^{aA} \right) \, .
\end{align}
The equivalence class of pairs of spinors $(\xi^{A'} , \zeta^A)\sim(\hat{\xi}^{A'} , \hat{\zeta}^A)$ is a section of what is known as the \emph{local twistor bundle} \cite{Penrose1986} or \emph{spin tractor bundle} \cite{Hammerl2011}. Such a pair of spinors will be referred to as a \emph{tractor spinor}. This bundle arises from a chiral spinor representation for $\Spin(2m+2,\C)$.

We shall mostly be concerned with the case where the conformal Killing spinor $\xi^{A'}$ is pure. We note that the purity of $\xi^{A'}$ does not in general entail the purity of $\zeta^A$ in dimensions greater than six. For future reference, we record the integrability condition for a pure conformal Killing spinor.
\begin{prop}\label{prop-int_cond_twistor-spinor}
 Let $\xi^{A'}$ be a pure conformal Killing spinor on $(\mcM,g)$ with $\zeta^B = \frac{\sqrt{2}}{n} \gamma \ind{^a _{A'}^B} \nabla_a \xi^{A'}$. Then
\begin{align}
 C_{abcd} \xi \ind*{^{cd}^{B'}} & = 0 \, , & & \mbox{ie} & {}^\mfC _\xi \Pi_1^0 (C) & = 0 \, , \label{eq-int-CKsp1} \\
 C \ind{_{a b c d}} \zeta \ind*{^{c d} ^C} - 2 \sqrt{2} A \ind{_{cab}} \xi \ind*{^c ^C} & = 0 \, , \label{eq-int-CKsp2} \\
  A \ind{_{abc}} \xi \ind*{^a ^A} \xi \ind*{^{bc}^{B'}} & = 0 \, , & & \mbox{ie} & {}^\mfA _\xi \Pi_{-\frac{1}{2}}^0 (A) & = 0 \, . \label{eq-int-CKsp3}
\end{align}
Here $A_{abc} := 2 \, \nabla_{[b} \Rho_{c]a}$ is the \emph{Cotton-York tensor} (see appendix \ref{sec-conformal}).
\end{prop}

\begin{proof}
The LHS of \eqref{eq-int-CKsp1} and \eqref{eq-int-CKsp2} are the usual integrability conditions for a (not necessarily pure) conformal Killing spinor (see eg \cite{Baum2010}), and the RHS of \eqref{eq-int-CKsp1} is simply the interpretation in the case when $\xi^{A'}$ is pure. Condition \eqref{eq-int-CKsp3} follows from \eqref{eq-int-CKsp1} and \eqref{eq-int-CKsp2}.
\end{proof}

In four dimensions, a conformal Killing spinor is always geodetic \cites{Penrose1967,Penrose1986}, but it is not so in general in higher dimensions. We can nevertheless give necessary and sufficient conditions for this to happen.
\begin{prop}\label{prop-foliating_twistor_spinor}
Let $\xi^{A'}$ be a pure conformal Killing spinor on $(\mcM,g)$ with almost null structure $\mcN_\xi$. Set $\zeta^B = \frac{\sqrt{2}}{n} \gamma \ind{^a _{A'}^B} \nabla_a \xi^{A'}$. Then $\xi^{A'}$ satisfies
\begin{align}\tag{\ref{eq-sym-foliating}}
 ( \xi \ind*{^a^{\lp{A}}} \nabla \ind{_a} \xi \ind*{^b^{\rp{B}}} ) \xi \ind*{_b^C} & = 0 \, .
\end{align}

 Further, $\xi^{A'}$ is geodetic, i.e.\ $\mcN_\xi$ is totally geodetic, if and only if
\begin{align}\label{eq-pure_pair}
 \zeta \ind*{^a ^{A'}} \zeta \ind*{_a ^{B'}} & = 0 \, , & \xi \ind*{^a ^B} \zeta \ind*{_a ^{A'}} & = - 2 \, \zeta \ind*{^B} \xi \ind*{^{A'}} \, ,
\end{align}
i.e.\ $\zeta^A$, if non-zero, is pure, and its almost null structure $\mcN_\zeta$ intersects $\mcN_\xi$ in a totally null $(m-1)$-plane at any point.

Suppose that $\xi^{A'}$ is geodetic so that $\zeta^A$ satisfies conditions \eqref{eq-pure_pair}. Then $\zeta^A$ satisfies
\begin{align}\label{eq-twistor2-almost-foliating}
 \left( \zeta \ind*{^{a[A'}} \nabla_a \zeta \ind*{^{bB'}} \right) \zeta \ind*{_b^{C']}} & = 0 \, .
\end{align}
\end{prop}

\begin{proof}
From the twistor equation \eqref{eq-twistor-spinor} and assuming $\xi^{A'}$ to be pure, we compute
\begin{align*}
 \left( \xi \ind*{^{aA}} \nabla_a \xi \ind*{^{bB}} \right) \xi \ind*{_b^C} & = - \frac{1}{\sqrt{2}} \xi \ind*{^a^A} \zeta \ind*{_a^{D'}} \gamma \ind{^b_{D'}^B} \xi \ind*{_b^C} = - \frac{1}{\sqrt{2}} \xi \ind*{^a^A} \zeta \ind*{_{ab}^B} \xi \ind*{^b^C} \, .
\end{align*}
This expression is skew in $BC$ and $AB$, which proves the first claim.

To prove the second statement, we consider the contraction of equation \eqref{eq-twistor-spinor} with $\xi \ind*{^a ^B}$, ie
\begin{align}\label{eq-contracted-twistor-spinor-eq}
 \xi \ind*{^a ^B} \nabla_a \xi^{A'} = - \frac{1}{\sqrt{2}} \xi \ind*{^a ^B} \zeta \ind*{_a ^{A'}} \, .
\end{align}
Let us work at a point, and in line with the notation of section \ref{sec-algebra}, set $\mfS^{\frac{m}{4}} := \langle \xi ^{A'} \rangle$, $\mfS^{\frac{m-2}{4}} := \im \xi \ind*{_a^A}$, $\mfS^{\frac{m-4}{4}} := \im \xi \ind*{_{ab}^{A'}}$ and $\mfS^{\frac{m-6}{4}} := \im \xi \ind*{_{abc}^{A'}}$. Generically, both sides of \eqref{eq-contracted-twistor-spinor-eq} lie in $\mfS^{\frac{m-2}{4}} \otimes \mfS^{\frac{m-4}{4}}$, which contains $\mfS^{\frac{m-2}{4}} \otimes \mfS^{\frac{m}{4}}$. In fact, both sides of \eqref{eq-contracted-twistor-spinor-eq} will lie in $\mfS^{\frac{m-2}{4}} \otimes \mfS^{\frac{m}{4}}$ if and only if $\xi \ind*{^a ^A} \nabla_a \xi^{B'} = \beta \ind{^A} \xi \ind*{^{B'}}$ for some $\beta \ind{^A}$ in $\mfS^{\frac{m-2}{4}}$ if and only if $\xi^{A'}$ is geodetic. In sum, the conformal Killing spinor $\xi^{A'}$ is geodetic if and only if
\begin{align}\label{eq-geod-CKsp}
\beta \ind{^A} \xi \ind*{^{B'}} & = - \frac{1}{\sqrt{2}} \xi \ind*{^a ^B} \zeta \ind*{_a ^{A'}} \, .
\end{align}
The `if' part of the statement is immediate already from \eqref{eq-contracted-twistor-spinor-eq}, so we focus on the `only if' part.

We know that generically, $\zeta^A$ lies in $\mfS^{\frac{m-6}{4}}$, which contains $\mfS^{\frac{m-2}{4}}$. One way to see this is to use \eqref{eq-LC-spin-connection} where $\partial_a$ is so chosen as to preserve $\xi^{A'}$. Then
\begin{align*}
\zeta^A = -\frac{1}{2\sqrt{2}n} \left( \Gamma_{abc} \xi^{abc A} + 2 \, \Gamma \ind{^a_{ab}} \xi^{bA} \right) \in \mfS^{\frac{m-6}{4}} \, .
\end{align*}
Since the LHS of \eqref{eq-geod-CKsp} is in $\mfS^{\frac{m-2}{4}} \otimes \mfS^{\frac{m}{4}}$, so must be the RHS. This tells us that $\zeta^A$ must be proportional to $\beta^A$ and lie in $\mfS^{\frac{m-2}{4}}$ too.\footnote{This can also be computed by applying Cartan's criterion \eqref{eq-pure_pair2} of Proposition \ref{prop-pure-spinors-Cartan} to \eqref{eq-geod-CKsp} with $k=m-1$ and $\alpha^{A'} = \xi^{A'}$.}  By Corollary \ref{prop-pure-spinors3}, we conclude that $\zeta^A$ satisfies \eqref{eq-pure_pair} as claimed. The geometric interpretation is given by Proposition \ref{prop-pure-spinors}.

Finally, if $\xi^{A'}$ is geodetic, then $\zeta^A$ satisfies \eqref{eq-pure_pair} by the above argument, and a computation leads to $\left( \nabla_a \zeta \ind*{^b^{B'}} \right) \zeta \ind*{_b^{C'}} = - 2 \cdot \frac{1}{\sqrt{2}} \,  \Rho _{ab} \zeta \ind*{^b^{[B'}} \xi ^{C']}$ from which \eqref{eq-twistor2-almost-foliating} can be deduced.
\end{proof}

\begin{rem}
The statements of Proposition \ref{prop-foliating_twistor_spinor} are conformally invariant. In fact, it is straightforward to check that conditions \eqref{eq-pure_pair} are invariant under the transformation \eqref{eq-twistor-conf_transf}.

Further, the conditions that $\xi^{A'}$ be pure and $\zeta^A$ satisfy \eqref{eq-pure_pair} is equivalent to the corresponding tractor spinor $(\xi^{A'},\zeta^A)$ being a pure section of the local twistor bundle, i.e.\ it is a pure spinor for $\Spin(2m+2,\C)$ as stated in \cite{Hughston1988} -- for a proof, see \cite{Taghavi-Chabert2015}.
\end{rem}

\begin{rem}
Clearly, if $\zeta=0$ in Proposition \ref{prop-foliating_twistor_spinor}, then $\xi$ is parallel and $\mcN_\xi$ is thus integrable. One can show \cite{Hammerl2016} that if $\xi$ is a geodetic pure conformal Killing spinor then there exists a conformal rescaling such that $\xi$ is parallel.
\end{rem}

The next result is a refinement of Proposition \ref{prop-foliating_twistor_spinor} in the case $m=3$. A proof is given in appendix \ref{sec-spin-calculus6}.
\begin{prop}\label{prop-foliating_twistor_spinor6}
 Let $\xi_A$ be a conformal Killing spinor on $(\mcM,g)$ where $\mcM$ has dimension six. Then $\xi_A$ satisfies condition \eqref{eq-proj-twistor6a} (or \eqref{eq-proj-twistor6}), and thus condition \eqref{eq-sym-foliating} (or \eqref{eq-sym-foliating6}).
\end{prop}

\begin{exa}\label{exa-3distribution}
It is shown in \cite{Bryant2006} how one can associate to a generic $3$-plane distribution $\mcN$ on a six-dimensional manifold $\mcM$ a conformal structure $[g]$. By generic, we mean that $\mcN$ is maximally non-integrable, i.e.\ $\Gamma(\mcN) + [\Gamma(\mcN),\Gamma(\mcN)] = \Gamma (\Tgt \mcM)$. The authors of \cite{Hammerl2011} later characterised $(\mcM,[g])$ in terms of a tractor spinor $(\xi_A, \zeta^A)$ which  is \emph{generic} in the sense that $(\xi_A, \zeta^A)$ satisfy the non-degeneracy condition $\xi_A \zeta^A \neq 0$, i.e.\ $(\xi_A, \zeta^A)$ is an `impure' tractor spinor for the group $\Spin(4,4)$. In this case, the conformal holonomy of $(\mcM,[g])$ being reduced to (a subgroup of) $\Spin(3,4)$. This example clearly works in the category of complex Riemannian manifold --- see section \ref{sec-real-geom}. Note that in this case, the intrinsic torsion of the $P$-structure defined by $[\xi^{A'}]$ is \emph{not} generic since \eqref{eq-sym-foliating} is satisfied.
\end{exa}

\subsubsection{Relation to the Goldberg-Sachs theorem in higher dimensions}
The \emph{Goldberg-Sachs theorem} \cites{Goldberg2009} is a classical theorem of general relativity, which, in the context of the present paper, can be interpreted in the following terms \cites{Gover2010,Taghavi-Chabert2012}.
\begin{thm}\label{thm-GS4}
Let $(\mcM,g)$ be a four-dimensional non-conformally flat spin complex Riemannian manifold satisfying the Einstein equations $R_{ab} = \lambda g_{ab}$ for some constant $\lambda$. Let $[\xi^{A'}]$ is a holomorphic projective pure spinor, then locally
\begin{align}\label{eq-GS4}
C_{abcd} \xi \ind*{^a^A} \xi \ind*{^b^B} \xi \ind*{^c^C} & = 0 & \Longleftrightarrow & &  \mbox{$[\xi^{A'}]$ is geodetic.}
\end{align}
\end{thm}
One can show that the condition on the Weyl tensor really restricts to its self-dual part -- see appendix \ref{sec-spin-calculus4}. There are other versions of the theorem, all of which are reviewed in \cite{Gover2010}. A conformally invariant version \cites{Kundt1962,Robinson1963,Penrose1986} motivated the author's partial generalisation, which we present in truncated form in the language of pure spinors.
\begin{thm}[\cite{Taghavi-Chabert2012}]\label{thm-GS}
Assume $m\geq2$. Let $[\xi^{A'}]$ be a holomorphic projective pure spinor field on a $2m$-dimensional spin complex Riemannian manifold $(\mcM,g)$ with associated almost null structure $\mcN_\xi$. Suppose that the Weyl tensor and the Cotton-York tensor satisfy the algebraic degeneracy conditions
\begin{align}\label{eq-Weyl_CY_GS}
\begin{aligned}
{}^\mfC _\xi \Pi_{-1}^0 (C) = {}^\mfC _\xi \Pi_{-1}^1 (C) & = 0 \, , \qquad \mbox{ie}  \qquad C_{abcd} \xi \ind*{^a^A} \xi \ind*{^b^B} \xi \ind*{^c^C} = 0 \, , \\
 {}^\mfA _\xi \Pi_{-{\frac{3}{2}}}^0 (A) & = 0 \, , \qquad \mbox{ie} \qquad A_{abc} \xi \ind*{^a^A} \xi \ind*{^b^B} \xi \ind*{^c^C} = 0 \, .
 \end{aligned}
\end{align}
Suppose further that the Weyl tensor is otherwise generic. Then locally, $[\xi^{A'}]$ is geodetic, i.e.\ $\mcN_\xi$ is totally geodetic (or equivalently, integrable).
\end{thm}

When $m=2$, Theorem \ref{thm-GS} agrees with parts of the generalisation \cites{Kundt1962,Robinson1963,Penrose1986}. If $(\mcM,g)$ is assumed to be Einstein, one can dispense of the genericity assumption and recover the $(\Rightarrow)$ part of \eqref{eq-GS4}.

However, when $m>2$, even if we assume that $(\mcM,g)$ is Einstein, the proof of Theorem \ref{thm-GS} does not account for all the possible degeneracies of $C_{abcd}$ and must depend on the genericity of $C_{abcd}$. In fact, Proposition \ref{prop-foliating_twistor_spinor} invalidates the $(\Rightarrow)$ part of \eqref{eq-GS4}: generically, a pure conformal Killing spinor $\xi^{A'}$ is not geodetic, but satisfies \eqref{eq-sym-foliating}. On the other hand, by Proposition \ref{prop-int_cond_twistor-spinor}, conditions \eqref{eq-Weyl_CY_GS} are satisfied \emph{except} for the genericity assumption, since ${}^\mfC _\xi \Pi_1^0 (C) = 0$. This is independent of whether $(\mcM,g)$ is Einstein or not, and in fact, with reference to Example \ref{exa-3distribution}, Einstein solutions where $\xi^{A'}$ is non-geodetic exist in dimension six \cite{Hammerl2011}. This suggests the following conjecture improving Theorem \eqref{thm-GS}.
\begin{conjec}\label{conjec-GS}
 Suppose that $[ \xi^{A'} ]$ is a projective pure spinor field on a $2m$-dimensional non-conformally flat Einstein spin complex Riemannian manifold such that the Weyl tensor satisfies $C_{abcd} \xi^{aA} \xi^{bB} \xi^{cC} =0$. Then $\xi^{A'}$ satisfies
\begin{align}\tag{\ref{eq-sym-foliating}}
 \left( \xi \ind*{^{a(A}} \nabla_a \xi \ind*{^{bB)}} \right) \xi \ind*{_b^C} & = 0 \, .
\end{align}
\end{conjec}
When $m=2$, this conjecture does agree with the $(\Rightarrow)$ part of \eqref{eq-GS4}. Variants involving the Cotton-York tensor and weaker conditions on the Weyl tensor such as $C_{abcd} \xi^{aA} \xi^{bB} \xi^{cdC'} =0$ may also be possible.

\subsection{Application to real pseudo-Riemannian manifolds}\label{sec-real-geom}
Let $(\mcM',g')$ be a spin oriented $2m$-dimensional real pseudo-Riemannian manifold where $g'$ has signature $(p,q)$ with $p+q=2m$ -- we assume that $(\mcM',g')$ is also time-oriented when $p q \neq 0$. We then have a reduction of the structure group of the frame bundle $\mathrm{F} \mcM'$ of $\mcM'$ to $\Spin_0(p,q)$, the two-fold covering of the identity component $\SO_0(p,q)$ of $\SO(p,q)$.

With reference to section \ref{sec-real-spin}, we define an almost null structure on $(\mcM',g')$ to be a totally null complex $m$-plane distribution $\mcN \subset \Tgt^\C \mcM'$. Here $\Tgt^\C_p \mcM' := \C \otimes \Tgt_p \mcM'$, and $\mcN_p$ is totally null with respect to the complexification of $g'_p$ at any point $p$. A pure spinor field $\xi$ up to scale determines an almost null structure $\mcN_\xi$, and any almost null structure arises in this way. The complex conjugation on $\Tgt^\C \mcM'$ that fixes $g'$ sends $\mcN_\xi$ to its complex conjugate $\overline{\mcN}_{\bar{\xi}}$, which we can associate to the conjugate spinor $\bar{\xi}$ of $\xi$. We shall assume that the real index $r:\mcM'\rightarrow \Z_{\geq0}:p \mapsto r_p:=\dim (\mcN_\xi)_p \cap (\overline{\mcN}_{\bar{\xi}})_p$ of $\xi$ is constant on $\mcM'$. The structure group of $\mathrm{F} \mcM'$ is reduced to the stabiliser of \emph{both} $[\xi]$ and $[\bar{\xi}]$ in $\Spin_0(p,q)$ at any point. One can then study the geometric properties of the resulting $\G$-structure on the basis of the algebraic properties of its intrinsinc torsion. A notable example is when $r=0$ and $g$ is positive definite signature, so that $(\mcN_\xi, \overline{\mcN}_{\bar{\xi}})$ define an almost Hermitian structure \cites{Gray1980,Tricerri1981,Falcitelli1994}.

Alternatively, if $(\mcM',g')$ is \emph{real-analytic}, then we can complexify $\mcM'$ to a complex manifold $\mcM$, and extend $g'$ analytically to a holomorphic metric $g$ on $\mcM$. Similar extensions apply to any real-analytic structure on $(\mcM',g')$ such as spin structures and almost null structures \cites{Whitney1959,Woodhouse1977,Eastwood1984}. We then have a complex Riemannian manifold $(\mcM,g)$ just as before except that it is endowed with an additional complex conjugation that fixes the real slice $\mcM'$. In particular,  $\left. \Tgt \mcM \right|_{\mcM'} = \Tgt^\C \mcM'$, and any real-analytic conjugate pair of complex almost null structures $(\mcN_\xi,\overline{\mcN}_{\bar{\xi}})$ extends to two \emph{independent} holomorphic almost null structures $(\mcN_\xi,\widetilde{\mcN}_{\tilde{\xi}})$, say,  on $(\mcM,g)$. In practice, it is enough to consider only one of  these, and apply the machinery of the present paper to it. The real geometry can then be recovered by applying reality conditions on the tensor and spinor fields on restriction to the real slice $\mcM'$.

Finally, if $(\mcM',g')$ is \emph{smooth}, then we cannot in general complexify $(\mcM',g')$ to a complex Riemannian manifold $(\mcM,g)$. This is particularly problematic for almost null structures with $r \neq0,m$, notably in relation with the existence of local `complex' foliations on $(\mcM',g')$, the difficulty being in applying the Frobenius theorem to a \emph{formally} integrable, i.e.\ involutive, smooth \emph{complex} distribution. For instance, Theorem \ref{thm-Robinson} will not work in general -- see \cite{Tafel1985} when $m=2$, $r=1$. Some of these issues are explained in  \cites{Nurowski2002,Trautman2002} for the case $r=1$. The reader should be warned of any pitfalls regarding the use of the results of the present paper in the smooth category.

We can however distinguish the two special cases:
\begin{enumerate}
\item When $r=0$, $g'$ must have signature $(2k,2\ell)$, and $\mcN_\xi$ and $\overline{\mcN}_{\bar{\xi}}$ can be identify with the $\pm\ii$-eigenbundles of an \emph{almost complex structure} $J$ compatible with $g'$, i.e.\ an endomorphism $J$ of $\Tgt \mcM'$ such that $J^2=-\Id$ and $g \circ J = - J \circ g$. The Newlander-Niremberg theorem \cite{Newlander1957} tells us that  even if $(\mcM',g')$ is smooth, the formal integrability of $\mcN_\xi$ (ie the vanishing of the \emph{Nijenhuis tensor} of $J$) is equivalent to its integrability.
\item When $r=m$, $g'$ must have signature $(m,m)$, and there is a \emph{real} span of $\mcN_\xi$ where $\xi$ can be taken to be a real pure spinor. The stabiliser $P'$ of $[\xi]$ in $\Spin_0(m,m)$ is parabolic, and $\mcN_\xi$ thus defines a smooth $P'$-structure on $(\mcM',g')$. The representation theory of $P'$ works in the same way as its complex counterpart. In fact, the spinor representations are the real spans of pure spinors, and all the vector bundles considered are real and smooth.  In this case, we can reformulate all the results of section \ref{sec-geometry} in the smooth real category: $\mcM$ is a spin oriented and time-oriented $2m$-dimensional smooth manifold equipped with a smooth metric $g$ of signature $(m,m)$ with Levi-Civita connection $\nabla$, and we can safely substitute the word `smooth' for `holomorphic'. In particular, the Frobenius theorem applies to prove Theorem \ref{thm-Robinson}.
\end{enumerate}

\paragraph{Acknowlegments}
The author wishes to thank Jan Slov\'ak (Masaryk), Lionel Mason (Oxford), Thomas Leistner (Adelaide) and Dennis The (Australian National University) for their hospitality and helpful discussions, and the organisers of the workshop `The interaction of geometry and representation theory. Exploring new frontiers' at the Erwin Schr\"{o}dinger Institute, Vienna in September 10-14, 2012, where parts of this work were carried out. Finally, he is particularly grateful to Michael Eastwood (Adelaide, ANU) and the anonymous referee for making useful suggestions to improve this article.

This work was funded by a SoMoPro (South Moravian Programme) Fellowship: it has received a financial contribution from the European Union within the Seventh Framework Programme (FP/2007-2013) under Grant Agreement No. 229603, and is also co-financed by the South Moravian Region.

Revision of this paper was carried out while the author was on an Eduard \v{C}ech Institute postdoctoral fellowship GPB201/12/G028, and a GA\v{C}R (Czech Science Foundation) post-doctoral grant GP14-27885P.

\appendix
\section{Spinorial description of curvature tensors}\label{sec-spinor-descript}
Throughout the appendix, $\mfV$ will denote a $2m$-dimensional complex vector space equipped with a non-degenerate symmetric bilinear form $g_{ab}$ and a pure spinor $\xi^{A'}$ as in section \ref{sec-algebra}, to which the reader should refer for the notation.

\subsection{Elements of the $\g_0$-submodules of $\mfF$, $\mfA$ and $\mfC$}\label{sec-g0-bases}
Let us fix a pure spinor $\eta_{A'}$ such that $\xi^{A'}\eta_{A'} = -\frac{1}{2}$. Then we have a splitting \eqref{eq-V-splitting} of $\mfV$ where $\mfV_{\frac{1}{2}} = \ker \xi \ind*{^a^A}$ and $\mfV_{-\frac{1}{2}} = \ker \eta \ind*{^a_A}$ and $\mfS_{\frac{m-2}{4}} = \im \xi \ind*{^a^A}$ and $\mfS_{-\frac{m-2}{4}} = \im \eta \ind*{^a_A}$ are $\g_0$-modules. This splitting induces a splitting of any $\g$-submodule $\mfM$ of $\otimes^k \mfV$ into $\g_0$-submodules. We can then use $\xi \ind*{^a^A}$ and $\eta \ind*{^a_A}$ to project from $\mfM$ to any of its $\g_0$-submodules, and dually, to inject any of its $\g_0$-submodules into $\mfM$. In effect, we convert spinorial quantities to tensorial ones, and vice versa.  In fact, we can think of $\{ \xi^{aA} , \eta^a_A\}$ as a basis for $\mfV_{\frac{1}{2}} \oplus \mfV_{-\frac{1}{2}}$, and these induce bases of $\g_0$-modules. The components of an element of $\mfM$ in this basis can then be interpreted as a spinor. For instance, a spinor
\begin{align*}
 \sigma \ind{_{A \ldots C}^{D \ldots F}}  \in \mfS_{-\frac{m-2}{4}} \otimes \ldots \otimes \mfS_{-\frac{m-2}{4}} \otimes \mfS_{\frac{m-2}{4}} \otimes \ldots \otimes \mfS_{\frac{m-2}{4}} \, ,
\end{align*}
will be sent to the tensor
\begin{align*}
\sigma^{a \ldots c d \ldots f} :=  \xi \ind*{^a^A} \ldots \xi \ind*{^c^C} \, \eta \ind*{^d_D}  \ldots \eta \ind*{^f_F} \, \sigma \ind{_{A \ldots C}^{D \ldots F}}  \in  \mfV \otimes \ldots \otimes \mfV \otimes \mfV \otimes \ldots \otimes \mfV \, .
\end{align*}
If the $\g_0$-module is irreducible, then its elements (ie their indices) will be saturated with symmetries. This clearly applies to $\g$-modules too. In the following any spinor will be referred as \emph{(totally) tracefree} if the contraction of any pair of indices with the identity element $I \ind*{_A^B} := \eta \ind*{^a_A} \xi \ind*{_a^B} $ (see definition \eqref{eq-id_m}) vanishes, eg $\sigma \ind{_{A}^{B}}  I \ind*{_B^A} = 0$. On the other hand, the image of $I \ind*{_A^B}$ in $\wedge^2 \mfV$ will be denoted by the $2$-form $\omega_{ab} := 2 \, \xi \ind*{_{[a}^A} \eta \ind*{_{b]}_A}$.

We apply this procedure to the irreducible $\g_0$-modules $\mfF_i$ $\breve{\mfA}_i^j$, $\breve{\mfC}_i^j$ with reference to Propositions \ref{prop-main_Ricci}, \ref{prop-main_CY} \ref{prop-main_Weyl} of section \ref{sec-class-curvature}.

\paragraph{The tracefree Ricci tensor}
 Let $\Phi_{ab} \in \mfF$. Then
\begin{itemize}
 \item $\Phi_{ab} \in \mfF_0$ if and only if $\Phi_{ab} = 2 \, \xi \ind*{_{\lp{a}} ^A} \eta \ind*{_{\rp{b}} _B} \Phi \ind{_A^B}$
for some tracefree $\Phi \ind{_A^B}$;
\item $\Phi_{ab} \in \mfF_1$ if and only if
$\Phi_{ab} = \xi \ind*{_a ^A} \xi \ind*{_b ^B} \Phi_{AB}$
for some $\Phi_{AB}=\Phi_{(AB)}$, and similarly for $\mfF_{-1} \cong (\mfF_1)^*$ by substituting $\eta_{A'}$ for $\xi^{A'}$, and changing the index structure appropriately.
\end{itemize}

\paragraph{The Cotton-York tensor}
 Let $A_{abc} \in \mfA$. Then
\begin{itemize}
 \item $A_{abc} \in \breve{\mfA}_{\frac{1}{2}}^0$ if and only if $A \ind{_{abc}} = A \ind{_a} \omega \ind{_{bc}} - A \ind{_{\lb{b}}} \omega \ind{_{\rb{c}a}} + \frac{3}{n-1} \, g \ind{_{a \lb{b}}} \omega \ind{_{\rb{c} d}} A \ind{^d}$ for some $A_c = \xi \ind*{_c^C} A_C$;

 \item $A_{abc} \in \breve{\mfA}_{\frac{1}{2}}^1$ if and only if $A_{abc} =  \xi \ind*{_b ^A} \xi \ind*{_c ^B} \eta \ind*{_a _C} A \ind{_{AB}^C}  - \xi \ind*{_a ^A} \xi \ind*{_{\lb{b}} ^B} \eta \ind*{_{\rb{c}} _C} A \ind{_{AB}^C}$ for some tracefree $A \ind{_{AB}^C} = A \ind{_{[AB]}^C}$;

 \item $A_{abc} \in \breve{\mfA}_{\frac{1}{2}}^2$ if and only if $A_{abc} = 2 \, \xi \ind*{_a ^A} \xi \ind*{_{\lb{b}} ^B} \eta \ind*{_{\rb{c}} _C} A \ind{_{AB}^C}$ for some tracefree $A \ind{_{AB}^C} = A \ind{_{(AB)}^C}$;
 
 \item $A_{abc} \in \breve{\mfA}_{\frac{3}{2}}^0$ if and only if $A_{abc} = \xi \ind*{_a ^A} \xi \ind*{_b ^B} \xi \ind*{_c ^C} A \ind{_{ABC}}$ for some $A_{ABC} = A_{A[BC]}$ satisfying $A_{[ABC]} = 0$.
\end{itemize}
Since $(\breve{\mfA}_i^j)^* \cong \breve{\mfA}_{-i}^j$, spinorial formulae for elements of $\breve{\mfA}_{-i}^j$ for $i >0$ can be obtained from those of $\breve{\mfA}_i^j$ by simply interchanging $\xi^{A'}$ and $\eta_{A'}$ and making appropriate changes of index structures.

\paragraph{The Weyl tensor}
Let $C_{abcd} \in \mfC$. Then
\begin{itemize}
 \item $C_{abcd} \in \breve{\mfC}_0^0$ if and only if
$C \ind{_{abcd}} = c \left( 2 \, \omega \ind{_{ab}} \omega \ind{_{cd}} - 2 \, \omega \ind{_{a\lb{c}}} \omega \ind{_{\rb{d}b}} 
+ \frac{6}{n-1} g \ind{_{a \lb{c}}} \, g \ind{_{\rb{d}b}} \right)$
for some complex $c$;
 \item $C_{abcd} \in \breve{\mfC}_0^1$ if and only if
\begin{align*}
 C \ind{_{abcd}} & = \omega \ind{_{ab}} C \ind{_{cd}} + C \ind{_{ab}} \omega \ind{_{cd}} - 2 \, \omega \ind{_{\lb{a}|\lb{c}}} C \ind{_{\rb{d}|\rb{b}}}
- \frac{6}{n-2} \, \left( g \ind{_{\lb{a} | \lb{c}}} \omega \ind{_{\rb{d}}^e} C \ind{_{|\rb{b}e}} + g \ind{_{\lb{c} | \lb{a}}} \omega \ind{_{\rb{b}}^e} C \ind{_{|\rb{d}e}} \right) \, .
\end{align*}
where $C \ind{_{cd}} := 2 \, \xi \ind*{_{\lb{c}} ^C} \eta \ind*{_{\rb{d}} _D} C \ind{_C^D}$ for some tracefree $C \ind{_C^D}$;
\item when $m>3$, $C_{abcd} \in \breve{\mfC}_0^2$ if and only if
\begin{align*}
  C_{abcd} = \xi \ind*{_a ^A} \xi \ind*{_b ^B} \eta \ind*{_c _C} \eta \ind*{_d _D} C \ind{_{A B}^{C D}} + \xi \ind*{_c ^A} \xi \ind*{_d ^B} \eta \ind*{_a _C} \eta \ind*{_b _D} C \ind{_{A B}^{C D}} - 2 \, \xi \ind*{_{\lb{a}|} ^A} \xi \ind*{_{\lb{c}} ^C} \eta \ind*{_{\rb{d} |} _D} \eta \ind*{_{\rb{b}} _B} C \ind{_{A C}^{D B}} \, ,
\end{align*}
for some tracefree $C \ind{_{A C}^{D B}} = C \ind{_{[A C]}^{[D B]}}$;
 \item $C_{abcd} \in \breve{\mfC}_0^3$ if and only if $C_{abcd} = 4 \, \xi \ind*{_{\lb{a}|} ^A} \xi \ind*{_{\lb{c}} ^C} \eta \ind*{_{\rb{d}|} _D} \eta \ind*{_{\rb{b}} _B} C \ind{_{A C}^{D B}}$ for some tracefree $C \ind{_{A C}^{D B}} = C \ind{_{(A C)}^{(D B)}} $;
\item $C_{abcd} \in \breve{\mfC}_1^0$ if and only if
$C \ind{_{abcd}} = \omega \ind{_{ab}} C \ind{_{cd}} + C \ind{_{ab}} \omega \ind{_{cd}} - 2 \, \omega \ind{_{\lb{a}|\lb{c}}} C \ind{_{\rb{d}|\rb{b}}}$ where $C_{ab} := \xi \ind*{_a ^A} \xi \ind*{_b ^B} C_{A B}$ for some $C _{CD} = C _{[CD]}$;
\item $C_{abcd} \in \breve{\mfC}_1^1$ if and only if
$C_{abcd} = 2 \, \xi \ind*{_a ^A} \xi \ind*{_b ^B} \xi \ind*{_{\lb{c}} ^C} \eta \ind*{_{\rb{d}} _D} C \ind{_{A B C}^D} + 2 \, \xi \ind*{_c ^A} \xi \ind*{_d ^B} \xi \ind*{_{\lb{a}} ^C} \eta \ind*{_{\rb{b}} _D} C \ind{_{A B C}^D}$ for some tracefree $C \ind{_{ABC}^D} = C \ind{_{[AB]C}^D}$ satisfying $C \ind{_{[ABC]}^D} = 0$;
\item $C_{abcd} \in \breve{\mfC}_2^0$ if and only if $C_{abcd} = \xi \ind*{_a ^A} \xi \ind*{_b ^B} \xi \ind*{_c ^C} \xi \ind*{_d ^D} C_{A B C D}$ for some $C _{ABCD} = C _{[AB][CD]}$ satisfying $C _{[ABC]D} = 0$.
\end{itemize}
Since $(\breve{\mfC}_i^j)^* \cong \breve{\mfC}_{-i}^j$, spinorial formulae for elements of $\breve{\mfC}_{-i}^j$ for $i >0$ can be obtained from those of $\breve{\mfC}_i^j$ by simply interchanging $\xi^{A'}$ and $\eta_{A'}$ and making appropriate changes of index structures.

\subsection{Maps describing elements of $\prb$-submodules of $\mfF$, $\mfA$ and $\mfC$}\label{sec-projection}
The kernels of the following maps ${}^\mfF _\xi \Pi_i^j$, ${}^\mfA _\xi \Pi_i^j$ and ${}^\mfC _\xi \Pi_i^j$ are $\prb$-submodules of $\mfF$, $\mfA$ and $\mfC$. Their relations to the irreducible $\prb$-modules $\mfF_i^j$, $\mfA_i^j$ and $\mfC_i^j$ as stated in Propositions \ref{prop-main_Ricci}, \ref{prop-main_CY} \ref{prop-main_Weyl} of section \ref{sec-class-curvature} can be verified using arbitrary elements $\breve{\mfF}_i^j$, $\breve{\mfA}_i^j$ and $\breve{\mfC}_i^j$ as given in section \ref{sec-g0-bases}. This can also be seen from the fact they are saturated with symmetries.

\paragraph{The tracefree Ricci tensor}
For $\Phi \ind{_{ab}} \in \mfF$, we define
\begin{align*}
{}^\mfF _\xi \Pi^0_{-1} ( \Phi ) & := \xi \ind*{^a ^A} \xi \ind*{^b ^B} \Phi \ind{_{a b}} \, , &
{}^\mfF _\xi \Pi^0_0 ( \Phi ) & := \xi \ind*{^a ^A} \Phi \ind{_{a b}} \, .
\end{align*}

\paragraph{The Cotton-York tensor}
For $A \ind{_{abc}} \in \mfA$, define
\begin{align*}
{}^\mfA _\xi \Pi^0_{-\frac{3}{2}} ( A ) & := \xi \ind*{^a ^A} \xi \ind*{^b ^B} \xi \ind*{^c ^C} A \ind{_{a b c}} \, , \\
{}^\mfA _\xi \Pi^0_{-\frac{1}{2}} ( A  ) & := \xi \ind*{^{a} ^{A}} A \ind{_{a b c}} \xi \ind*{^{bc} ^{C'}} \, , \\
{}^\mfA _\xi \Pi^1_{-\frac{1}{2}} ( A ) & :=  \xi \ind*{^b ^B} \xi \ind*{^c ^C} A \ind{_{a b c}} + \frac{1}{n-2} \gamma \ind{_a _{D'}^{\lb{B}}} \xi \ind*{^d ^{\rb{C}}} \xi \ind*{^{bc} ^{D'}} A \ind{_{d b c}}   \, , \\
{}^\mfA _\xi \Pi^2_{-\frac{1}{2}} ( A ) & := \xi \ind*{^{a} ^{\lp{A}}} \xi \ind*{^{b} ^{\rp{B}}} A \ind{_{a b c}} + \frac{3}{2(n+2)} \gamma \ind{_c _{D'}^{\lp{A}}} \xi \ind*{^d ^{\rp{B}}} \xi \ind*{^{ba} ^{D'}} A \ind{_{d b a}}  \, , \\
{}^\mfA _\xi \Pi^0_{\frac{1}{2}} ( A )  & := A \ind{_{a b c}} \xi \ind*{^{bc} ^{A'}} \, , \\
{}^\mfA _\xi \Pi^1_{\frac{1}{2}} ( A ) & := \xi \ind*{^{c} ^{A}} A \ind{_{c a b}} - \frac{1}{n-2} \gamma \ind{_{\lb{a}} _{D'} ^A} A \ind{_{\rb{b} c d}} \xi \ind*{^{cd} ^{D'}} \, , \\
{}^\mfA _\xi \Pi^2_{\frac{1}{2}} ( A ) & := A \ind{_{(a b) c}} \xi \ind*{^{c} ^{A}} - \frac{3}{2(n+2)} \gamma \ind{_{\lp{a}} _{D'}^A} A \ind{_{\rp{b} c d}} \xi \ind*{^{cd} ^{D'}} \, .
\end{align*}

\paragraph{The Weyl tensor}
For $C \ind{_{abcd}} \in \mfC$, define
\begin{align*}
{}^\mfC _\xi \Pi_{-2}^0 (C) & := \xi \ind*{^a ^{A}} \xi \ind*{^b ^{B}} \xi \ind*{^c ^{C}} \xi \ind*{^d ^{D}} C \ind{_{a b c d}} \, , \\
{}^\mfC _\xi \Pi_{-1}^0(C) & := \xi \ind*{^a ^{A}} \xi \ind*{^b ^{B}} \xi \ind*{^{c d} ^{C'}} C \ind{_{a b c d}} \, , \\
{}^\mfC _\xi \Pi_{-1}^1(C) & := \xi \ind*{^a ^{A}} \xi \ind*{^b ^{B}} \xi \ind*{^c ^{C}} C \ind{_{a b c e}} + \frac{1}{n+2} \left( \xi \ind*{^a ^{A}} \xi \ind*{^b ^{B}} \xi \ind*{^{c d} ^{D'}} C \ind{_{a b c d}} \gamma \ind{_e _{D'} ^C} - \xi \ind*{^a ^{C}} \xi \ind*{^b ^{\lb{A}}} \xi \ind*{^{c d} ^{D'}} C \ind{_{a b c d}} \gamma \ind{_e _{D'} ^{\rb{B}}} \right) \, , \\
{}^\mfC _\xi \Pi_0^0(C) & := \xi \ind*{^{a b} ^{B'}} \xi \ind*{^{c d} ^{D'}} C \ind{_{a b c d}} \, , \\
{}^\mfC _\xi \Pi_0^1(C) & := \xi \ind*{^{a b} ^{B'}} \xi \ind*{^c ^D} C \ind{_{a b c d}} + \frac{1}{n} \xi \ind*{^{a b} ^{B'}} \xi \ind*{^{c e} ^{D'}} C \ind{_{a b c e}} \gamma \ind{_d _{D'} ^D} \, , \\
{}^\mfC _\xi \Pi_0^2(C) & := 
\xi \ind*{^a ^A} C \ind{_{a[bc]d}} \xi \ind*{^d ^D} + \frac{1}{n-4} \xi \ind*{^{ae} ^{C'}} C \ind{_{aed\lb{b}}} \gamma \ind{_{\rb{c}}_{C'} ^{\lb{A}}} \xi \ind*{^d ^{\rb{D}}} - \frac{1}{2(n-2)(n-4)} \xi \ind*{^{ae} ^{C'}} C \ind{_{aedf}} \xi \ind*{^{df} ^{F'}} \gamma \ind{_{\lb{b}} _{C'} ^{A}} \gamma \ind{_{\rb{c}} _{F'} ^{D}} \, , \\
{}^\mfC _\xi \Pi_0^3(C)  & := 
\xi \ind*{^a ^A} C \ind{_{a(bc)d}} \xi \ind*{^d ^D} - \frac{3}{n+4} \xi \ind*{^{ae} ^{C'}} C \ind{_{aed\lp{b}}} \gamma \ind{_{\rp{c}} _{C'} ^{\lp{A}}} \xi \ind*{^d ^{\rp{D}}} \\
& \qquad \qquad \qquad \qquad \qquad \qquad \qquad - \frac{3}{2(n+2)(n+4)} \xi \ind*{^{ae} ^{C'}} C \ind{_{aedf}} \xi \ind*{^{df} ^{F'}} \gamma \ind{_{\lp{b}} _{C'} ^{A}} \gamma \ind{_{\rp{c}} _{F'} ^{D}} \, , \\
{}^\mfC _\xi \Pi_1^0(C) & : = \xi \ind*{^{a b} ^{B'}} C \ind{_{a b c d}} \, , \\
{}^\mfC _\xi \Pi_1^1(C) & := \xi \ind*{^a ^B} C \ind{_{a b c d}} + \frac{1}{n+2} \left ( \xi \ind*{^{a e} ^{C'}} \gamma \ind{_b _{C'} ^B} C \ind{_{a e c d}} - \xi \ind*{^{a e} ^{C'}} C \ind{_{a e b \lb{c}}} \gamma \ind{_{\rb{d}} _{C'} ^B} \right) \, ,
\end{align*}
with the proviso that ${}^\mfC _\xi \Pi_0^2$ does not occur when $m=3$.

\section{Spinor calculus in four and six dimensions}\label{sec-spin-calculus4-6}
We briefly sketch the spinor calculus in dimensions four and six. Details for the former can be found in \cites{Penrose1984,Penrose1986} and references therein, and for the latter in \cites{Hughston1995,Mason1995}. Our notation will be consistent with the one introduced in section \ref{sec-algebra}. For definiteness, we work over $\C$, but the real case is completely analogous.

\subsection{Four dimensions}\label{sec-spin-calculus4}
Let $(\mcM,g)$ be a four-dimensional complex Riemannian manifold equipped with a holomorphic volume form and a holomorphic spin structure. We first work at a point. The spin group $G:=\Spin(4,\C)$ is isomorphic to ${}^{+} G \times {}^{-} G$, where ${}^{\pm} G:=\SL(2,\C)^{\pm}$ are two distinct copies of $\SL(2,\C)$, acting on the two-dimensional chiral spinor representations $\mfS^\pm$. All spinors in $\mfS^\pm$ are pure. The spaces $\mfS^\pm$ are equipped with volume forms $\varepsilon_{A'B'}$ and $\varepsilon_{AB}$, which will be assumed to satisfy $\varepsilon_{AC} \varepsilon^{BC} = \delta_A^B$. These identify $\mfS^\pm$ with their dual $(\mfS^\pm)^*$. Any irreducible representation of ${}^{\pm} G$ is isomorphic to a $k$-symmetric power $\odot^k \mfS^\pm$ of $\mfS^\pm$ for some $k \geq 0$ -- here $\odot^0 \mfS^\pm \cong \C$. By extension, any irreducible representation of $G$ is isomorphic to  $\left( \odot^k \mfS^+ \right) \otimes \left( \odot^\ell \mfS^- \right)$ for some $k,\ell \geq 0$. In particular, if $\mfV$ is the standard representation of $\SO(4,\C)$, then $\mfV \cong \mfS^+ \otimes \mfS^-$, and we shall convert tensorial indices into spinorial ones by means of the  $\gamma$-matrices $\gamma \ind{^a _{AA'}}$, which satisfy
\begin{align}\label{eq-g-epsilon4}
 g_{ab} \gamma \ind{^a _{AA'}} \gamma \ind{^b _{BB'}} & = 2 \, \varepsilon \ind{_{A'B'}} \varepsilon \ind{_{AB}} \, .
\end{align}
Thus, for any vector $V^a$, we have $V^{AB'} = \frac{1}{\sqrt{2}} \gamma \ind{_a^{AB'}} V^a$.

The Lie algebra $\g$ of $G$ splits into two irreducible parts ${}^{\pm} \g \cong \slie(2,\C)^{\pm}$, the Lie algebras of ${}^{\pm} G$. These are isomorphic to the spaces of self-dual and anti-self-dual $2$-forms $\wedge^2_\pm \mfV$. Correspondingly, a $2$-form $F_{ab}$ splits into self-dual and anti-self-dual parts represented by symmetric spinors $\phi_{A'B'}$ and $\phi_{AB}$ respectively, ie
\begin{align*}
 F_{ab} & = \phi_{A'B'} \varepsilon_{AB} + \phi_{AB} \varepsilon_{A'B'} \quad \in \quad \odot^2 \mfS^+ \oplus \odot^2 \mfS^- \cong \wedge^2_+ \mfV \oplus \wedge^2_- \mfV \cong {}^{+} \g  \oplus {}^{-} \g \, .
\end{align*}
Accordingly, the space $\mfA$ of tensors with Cotton-York symmetries and the space $\mfC$ of tensors with Weyl symmetries split into self-dual and anti-self-dual parts $\mfA^\pm \cong \mfS^\mp \otimes \left( \odot^3 \mfS^\pm \right)$ and $\mfC^\pm \cong \odot^4 \mfS^\pm$ respectively. Further, we have $\mfF \cong \left( \odot^2 \mfS^+ \right) \otimes \left( \odot^2 \mfS^- \right)$. Thus, the tracefree Ricci tensor, the Cotton-York tensor and the Weyl tensor can be expressed as
\begin{align*}
 \Phi_{ab} & =  \Phi \ind{_{A B A'B' }} \, , &
 A_{abc} & = A \ind{_{A A'B'C' }} \varepsilon_{B C} + A \ind{_{A' A B C }} \varepsilon_{B' C'} \, , &
  C_{abcd} & = \Psi_{A'B'C'D'} \varepsilon_{AB} \varepsilon_{CD} + \Psi_{ABCD} \varepsilon_{A'B'} \varepsilon_{C'D'} \, ,
\end{align*}
respectively, where $\Psi_{A'B'C'D'}$ and $\Psi_{ABCD}$ are the self-dual and anti-self-dual parts of $C_{abcd}$ respectively, and $A \ind{_{A A'B'C' }}$ and $A \ind{_{A' A B C }}$ are the self-dual and anti-self-dual parts of $A_{abc}$ respectively.

\subsubsection{Projective spinor fields}
Let $[ \xi^{A'} ]$ be a holomorphic projective spinor field on $(\mcM,g)$, so that the structure group is the frame bundle is reduced to $P$, the stabiliser of $[\xi^{A'}]$ in $G$ at a point, with Lie algebra $\prb$. As in the higher-dimensional case, $P$ is also a parabolic Lie subgroup of $G$, and $\prb$ induces a $|1|$-grading $\g = \g_1 \oplus \g_0 \oplus \g_{-1}$. The only difference now comes from the semi-simplicity of $\g$: if $\mfz_0$ and $\slie_0$ denote the center and simple part of $\g_0$, we have that ${}^{+} \prb := \prb \cap {}^{+} \g  \cong \mfz_0 \oplus \g_1$ and  $\prb \cap {}^{-} \g  = {}^{-} \g \cong \slie_0$. Then ${}^{-} G \cap  P = {}^{-} G$ and ${}^{+} P := {}^{+} G \cap P$ is a parabolic Lie subgroup of ${}^{+} G$. In effect, we can write $P = {}^{+} P \times {}^{-} G$. This means that for any irreducible $G$-module  $\mfM:= \odot^k \mfS^+ \otimes \odot^\ell \mfS^-$ for $k,\ell \geq0$, we have a filtration of $P$-submodules of $\mfM$ induced from a filtration of ${}^{+} P$-submodules of the ${}^{+} G$-module $\odot^k \mfS^+$. Following section \ref{sec-algebra}, set
\begin{align*}
 \mfS^{\frac{1}{2}} & := \langle \xi^{A'} \rangle \, , & \mfS^{-\frac{1}{2}} & := \mfS^+ \, , & \mfS^0 & := \mfS^- \, .
\end{align*}
As a consequence of the two-dimensionality of $\mfS^+$, we can characterise $\mfS^{\frac{1}{2}}$ as
\begin{align}\label{eq-principal_spinor}
\mfS^{\frac{1}{2}} & = \{ \alpha^{A'} \in \mfS^+ : \xi^{A'} \alpha_{A'} = 0 \}  \, .
\end{align}
More generally, any irreducible ${}^{+} G$-module  $\mfS^{-\frac{k}{2}} := \odot^k \mfS^{-\frac{1}{2}}$ admits a filtration
\begin{align}\label{eq-spin-filt-4}
 \mfS^{\frac{k}{2}} \subset \mfS^{\frac{k}{2}-1} \subset \mfS^{\frac{k}{2}-2} \subset \ldots \subset \mfS^{-\frac{k}{2}+1} \subset \mfS^{-\frac{k}{2}} \, .
\end{align}
of ${}^{+} P$-modules
\begin{align*}
 \mfS^{\frac{k-2\ell+2}{2}} := \{ \phi \ind{_{A_1' A_2' \ldots A_k'}} \in \mfS^{-\frac{k}{2}}  :  \phi \ind{_{A_1' A_2' \ldots  A_\ell' A_{\ell+1}' \ldots A_k'}} \xi^{A_1'} \xi^{A_2'} \ldots \xi^{A_\ell'} & = 0  \} \, , & \mbox{for $\ell = 1, \ldots, k$.}
\end{align*}
For any integer $k$, each summand $\mfS^{\frac{k-2\ell+2}{2}} / \mfS^{\frac{k-2\ell+4}{2}}$ in the associated graded module of \eqref{eq-spin-filt-4} is a one-dimensional ${}^+ P$-module isomorphic to a $\C$-module $\mfS_{\frac{k-2\ell+2}{2}}$, on which the grading element $\xi_{(A'} \eta_{B')}$ has eigenvalue $\frac{k-2\ell+2}{2}$ -- here $\xi^{A'} \eta_{A'} =1$.

\paragraph{Intrinsic torsion}
The intrinsic torsion of the $P$-structure can be identified as an element of the $P$-module $\mfW := \mfV \otimes \g/\prb$ at a point. In fact, since ${}^- \g$ acts trivially on $[\xi^{A'}]$, it is enough to consider the $P$-module ${}^+\mfW := \mfV \otimes {}^+ \g/{}^+\prb$. This is in fact consistent with the fact that the spin connection on $\mcS^+$ takes value in ${}^+\g$. We obtain a filtration of $P$-modules ${}^+\mfW^{-\frac{1}{2}} \subset {}^+\mfW^{-\frac{3}{2}} = {}^+\mfW$, where both ${}^+\mfW^{-\frac{1}{2}}$ and ${}^+\mfW^{-\frac{3}{2}} /{}^+\mfW^{-\frac{1}{2}}$ are one-dimensional. Details are left to the reader. The intrinsic torsion generically lies in ${}^+\mfW^{-\frac{3}{2}}$ and whether it degenerates to an element of ${}^+\mfW^{-\frac{1}{2}}$ or vanishes can be expressed by
\begin{align*}
 \xi^{A'} \xi^{B'} \nabla_{AA'} \xi_{B'} & = 0 \, , & \xi^{B'} \nabla_{AA'} \xi_{B'} & = 0 \, ,
\end{align*}
respectively. Here $\nabla_{AB'}$ stands for the Levi-Civita connection $\nabla_a$. These are precisely the geodetic spinor equation \eqref{eq-foliating_spinor} and the recurrent spinor equation \eqref{eq-recurrent_spinor} respectively.

Many of the results of section \ref{sec-geometry} can easily be adapted to the four-dimensional setting. For instance, equation \eqref{prop-conformal-invariance-spinor} can be rewritten as
 \begin{align*}
 \xi^{B'} \nabla_{AA'} \xi_{B'} & = \xi_{A'} \xi^{B'} \nabla_{AB'} f \, .
\end{align*}
We refer to the literature, notably \cites{Penrose1984,Penrose1986} for a detailed study of these spinorial equations and others.

\paragraph{Curvature tensors}
To describe irreducible $P$-submodules of $\left( \odot^k \mfS^+ \right) \otimes \left( \odot^\ell \mfS^- \right)$ for non-negative $k$ and $\ell$, it suffices to tensor the ${}^{+} P$-invariant filtration on $\odot^k \mfS^+$ with $\odot^\ell \mfS^-$. Thus, for the space $\mfF$, we have a filtration $ \{ 0 \} =: \mfF^2 \subset \mfF^1 \subset \mfF^0 \subset \mfF^{-1} := \mfF$ of indecomposable $P$-modules.
For the spaces $\mfA^\pm$, we obtain two distinct filtrations
\begin{align*}
  {}^+\mfA^{\frac{3}{2}} \subset {}^+\mfA^{\frac{1}{2}} \subset {}^+\mfA^{-\frac{1}{2}} \subset {}^+\mfA^{-\frac{3}{2}} = {}^+\mfA \, , &				& {}^-\mfA^{\frac{1}{2}} \subset {}^-\mfA^{-\frac{1}{2}} = {}^-\mfA \, .
\end{align*}
Finally, since $P$ induces no non-trivial filtration on ${}^- \mfC$, we are left with a filtration
\begin{align}\label{eq-filtration-SD-C-4}
 {}^+\mfC^2 \subset {}^+\mfC^1 \subset {}^+\mfC^0 \subset {}^+\mfC^{-1} \subset {}^+\mfC^{-2} = {}^+\mfC \, , 
\end{align}
of ${}^+ P$-submodules of the space ${}^+\mfC$ of self-dual Weyl tensors. Defining 
\begin{align}
\begin{aligned}\label{eq-maps-SD-C-4}
 {}^{{}^+\mfC} _{\; \; \, \xi} \Pi_{-2}^0 ( \Psi' ) & = \xi \ind{^{A'}} \xi \ind{^{B'}} \xi \ind{^{C'}} \xi \ind{^{D'}} \Psi \ind{_{A'B'C'D'}} \, , &
  {}^{{}^+\mfC} _{\; \; \, \xi} \Pi_{-1}^0 ( \Psi' ) & = \xi \ind{^{A'}} \xi \ind{^{B'}} \xi \ind{^{C'}} \Psi \ind{_{A'B'C'D'}} \, , \\
 {}^{{}^+\mfC} _{\; \; \, \xi} \Pi_0^0 ( \Psi' ) & = \xi \ind{^{A'}} \xi \ind{^{B'}} \Psi \ind{_{A'B'C'D'}} \, , &
 {}^{{}^+\mfC} _{\; \; \, \xi} \Pi_1^0 ( \Psi' ) & = \xi \ind{^{A'}} \Psi \ind{_{A'B'C'D'}} \, ,
 \end{aligned}
\end{align}
we see that ${}^+\mfC^i := \ker  {}^{{}^+\mfC} _{\; \; \, \xi} \Pi_{i-1}^0$ for all $i=-1,0,1,2$.

It is instructive to compare these maps with the maps ${}^\mfC _\xi \Pi_i^j$ defined in appendix \ref{sec-projection}, which can also be used in dimension four. It is relatively straightforward to show that, in four dimensions,
\begin{align*}
{}^\mfC _\xi \Pi_{\pm1}^1 (C) = {}^\mfC _\xi \Pi_0^1 (C) & = 0 \, , \\
{}^{\mfC} _\xi \Pi_i^0 ( C ) & = {}^{{}^+\mfC} _{\; \; \, \xi} \Pi_i^0 ( \Psi' ) \, , & \mbox{for $i=-1,0,1,2$,} \\
{}^{\mfC} _\xi \Pi_0^3 ( C ) & = \xi_{A'} \xi_{B'}  \Psi_{ABCD} \, .
\end{align*}
while ${}^\mfC _\xi \Pi_0^2$ is not defined. Since $\xi^{A'}$ is always assumed to be non-zero, one can interpret ${}^{\mfC} _\xi \Pi_0^3$ as the projection from $\mfC$ to ${}^-\mfC$, and expect it to replace the self-duality condition in higher dimensions.

\subsubsection{Principal spinors and the Petrov-Penrose classification}
For comparison, we recall some of the related notions given in \cites{Penrose1984,Penrose1986}. We say that $\xi^{A'}$ is a \emph{$(k-\ell+1)$-fold principal spinor} of an irreducible spinor $\phi \ind{_{A_1' \ldots A_k'}}$ if
\begin{align}\label{eq-principal-spinor}
\phi \ind{_{A_1' A_2' \ldots  A_\ell' A_{\ell+1}' \ldots A_k'}} \xi^{A_1'} \xi^{A_2'} \ldots \xi^{A_\ell'} & = 0 \, .
\end{align}
In the case $k=4$, we have a notion of $(5-\ell)$-fold principal spinor of the self-dual Weyl tensor, which is itself intimately connected to the \emph{Petrov-Penrose classification} of the self-dual Weyl tensor \cites{Petrov2000,Witten1959,Penrose1960}: at a point $p$, $\Psi_{A'B'C'D'}$ defines a homogeneous quartic polynomial $\Psi' (\pi) := \Psi_{A'B'C'D'} \pi^{A'} \pi^{B'} \pi^{C'} \pi^{D'} = 0$, where $[\pi^{A'}]$ are homogeneous coordinates on $\CP^1$, the fiber of the projective spinor bundle over $p$. This polynomial has four roots, and the multiplicities of these roots define the various \emph{Petrov types} $\{ 1 1 1 1 \}$, $\{2 1 1 \}$, $\{ 3 1 \}$, $\{ 2 2 \}$, $\{ 4 \}$ and $\{ - \}$ of $\Psi_{A'B'C'D'}$. In the generic case $\{ 1 1 1 1 \}$, $\Psi' (\pi)$ has four distinct roots, and thus four distinct principal spinors at $p$. Type $\{2 1 1 \}$ consists of a double root and two distinct simple roots, and thus  a $2$-fold principal spinors and two distinct $1$-fold principal spinors, and so on. Type $\{ - \}$ simply means $\Psi_{A'B'C'D'} =0$.

It must be emphasised that \emph{in sharp contract with the main ideas of the present paper}, the Petrov-Penrose classification makes \emph{no assumption} on the existence of a preferred (projective) spinor field on $(\mcM,g)$. In fact, one could single out \emph{any} spinor field $\xi^{A'}$ on $(\mcM,g)$. One would have a filtration \eqref{eq-filtration-SD-C-4} of ${}^{+} P$-modules on ${}^+ \mfC$. Then $\xi^{A'}$ would be a principal spinor for $\Psi_{A'B'C'D'}$ if and only if ${}^{{}^+\mfC} _{\; \; \, \xi} \Pi_{-2}^0 ( \Psi' ) =0$, which is already a non-trivial condition from the viewpoint of the $P$-structure. More generally, comparison of the maps \eqref{eq-maps-SD-C-4} and the definition \eqref{eq-principal-spinor} of principal spinors, one has
\begin{itemize}
\item $\Psi_{A'B'C'D'}$ is of type $\{1 1 1 1\}$ \emph{with} $1$-fold principal spinor $\xi^{A'}$ if and only if ${}^{{}^+\mfC} _{\; \; \, \xi} \Pi_{-2}^0 ( \Psi' ) =0$,
\item $\Psi_{A'B'C'D'}$ is of types $\{2 1 1\}$ or $\{2 2\}$ \emph{with} $2$-fold principal spinor $\xi^{A'}$ if and only if ${}^{{}^+\mfC} _{\; \; \, \xi} \Pi_{-1}^0 ( \Psi' ) =0$,
\item $\Psi_{A'B'C'D'}$ is of type $\{3 1\}$ \emph{with} $3$-fold principal spinor $\xi^{A'}$ if and only if ${}^{{}^+\mfC} _{\; \; \, \xi} \Pi_0^0 ( \Psi' ) =0$,
\item $\Psi_{A'B'C'D'}$ is of type $\{ 4 \}$ \emph{with} $4$-fold principal spinor $\xi^{A'}$ if and only if ${}^{{}^+\mfC} _{\; \; \, \xi} \Pi_1^0 ( \Psi' ) =0$.
\end{itemize}

On the other hand, any principal spinor \emph{field} $\xi^{A'}$ of $\Psi_{A'B'C'D'}$ on $(\mcM,g)$ defines a holomorphic reduction to the structure group $P$, the stabiliser of $[\xi^{A'}]$ at a point in $G$, and one can relate the Petrov types with \eqref{eq-filtration-SD-C-4} as we have just done.

\subsection{Six dimensions}\label{sec-spin-calculus6}
Let $(\mcM,g)$ be a six-dimensional complex Riemannian manifold equipped with a holomorphic volume form and a holomorphic spin structure. We first work at a point. The chiral spinor spaces are dual to each other, i.e.\ $(\mfS^\pm)^* \cong \mfS^\mp$, and can be identified with the four-dimensional standard and dual representations of the spin group $G = \Spin(6,\C) \cong \SL(4,\C)$. All spinors in $\mfS^\pm$ are pure. One can then eliminate the use of primed indices in favour of the unprimed ones, so we shall write $\mfS$ for $\mfS^-$ and $\mfS^*$ for $\mfS^+$. We can also convert tensor indices into a skew-symmetrised pair of indices by means of the skew-symmetric  $\gamma$-matrices $\frac{1}{2} \gamma \ind{^a _{AB}}$ and $\frac{1}{2} \gamma \ind{^a ^{AB}}$, which satisfy the identity
\begin{align}\label{eq-g-epsilon6}
 g_{ab} \gamma \ind{^a _{AB}} \gamma \ind{^b _{CD}} & = 2 \, \varepsilon \ind{_{ABCD}} \, , & g_{ab} \gamma \ind{^a ^{AB}} \gamma \ind{^b ^{CD}} & = 2 \, \varepsilon \ind{^{ABCD}} \, , & g_{ab} \gamma \ind{^a _{AB}} \gamma \ind{^b ^{CD}} & = 4 \, \delta \ind*{_{\lb{A}}^C} \delta \ind*{_{\rb{B}}^D}  \, ,
\end{align}
where $\varepsilon_{ABCD}=\varepsilon_{[ABCD]}$ and $\varepsilon^{ABCD}=\varepsilon^{[ABCD]}$ are volume forms on $\mfS$ and $\mfS^*$ respectively satisfying the normalisation
\begin{align*}
 \varepsilon \ind{_{ABCD}} \varepsilon \ind{^{EFGH}} & = 24 \, \delta \ind*{_{\lb{A}}^E} \delta \ind*{_B^F} \delta \ind*{_C^G} \delta \ind*{_{\rb{D}}^H} \, .
\end{align*}
Skew-symmetrised pairs of spinor indices can be raised and lowered by means of $\frac{1}{2}\varepsilon \ind{_{ABCD}}$ and $\frac{1}{2}\varepsilon \ind{^{ABCD}}$, eg $V_{AB} = \frac{1}{2} \varepsilon \ind{_{ABCD}} V^{CD}$. The isomorphism $\wedge^2 \mfS \cong \wedge^2 \mfS^*$ is the spinorial counterpart of the metric isomorphism between the standard representation $\mfV$ of $\SO(4,\C)$ and its dual $\mfV^*$. More generally, we identify
\begin{align*}
 \mfV & \cong \wedge^2 \mfS \, ,
& \wedge^2 \mfV & \cong \mfS \otimes_\circ \mfS^* \, ,
& \wedge^3_+ \mfV & \cong \odot^2 \mfS^* \, ,
& \wedge^3_- \mfV  & \cong \odot^2 \mfS \, .
\end{align*}
In addition, the tracefree Ricci tensor, the Weyl tensor, and the Cotton-York take the spinorial forms
\begin{align*}
 \Phi \ind{_{ab}} & = \Phi \ind{_{ABCD}} \, , &
 C \ind{_a^b_c^d} & = 8 \, \delta \ind*{_{\lb{A}}^{\lb{C}}} C \ind*{_{\rb{B}\lb{E}}^{\rb{D}\lb{G}}} \delta \ind*{_{\rb{F}}^{\rb{H}}} \, , &
 A \ind{_{ab}^c} & = 4 \, A \ind{_{A B \lb{C}} ^{\lb{E}}} \delta \ind*{_{\rb{D}}^{\rb{F}}} \, ,
\end{align*}
where $\Phi_{ABCD}=\Phi_{[AB][CD]}$ satisfies $\Phi_{[ABC]D} = 0$, $C \ind*{_{AB}^{CD}}= C \ind*{_{(AB)}^{(CD)}}$ is tracefree, and $A \ind{_{A B C} ^D}=A \ind{_{[A B] C} ^D}$ satisfies $A \ind{_{[A B C]} ^D}=0$ and $A \ind{_{A B C} ^A}=0$.

\subsubsection{Projective spinor fields} Let $[ \xi_A ]$ be a holomorphic projective spinor field on $(\mcM,g)$ so that the structure group of the frame bundle is reduced to $P$, the stabiliser of $[\xi_A]$ at a point in $G$. As in section \ref{sec-algebra}, $P$ induces filtrations
\begin{align*}
 \langle \xi_A \rangle  = \mfS^{\frac{3}{4}} & \subset  \mfS^{-\frac{1}{4}} = \mfS^* \, , &
\{ \beta^A \in \mfS : \beta^A \xi_A = 0 \} = \mfS^{\frac{1}{4}} & \subset \mfS^{-\frac{3}{4}} = \mfS \, ,
\end{align*}
of $P$-submodules. We can also re-express $\mfS^{\frac{3}{4}}  = \{ \alpha_A \in \mfS^* : \xi_{\lb{A}} \alpha_{\rb{B}} = 0 \}$. We can extend this argument to spinors of any valence, and play the same game with $\mfV$ and $\g$. In particular, the maps \eqref{eq-map_grading_g} defining the irreducible $\prb$-modules $\g_i^j$ of $\gr(\g)$ can then simply be expressed as
\begin{align*}
 {}^\g _\xi \Pi^0_{-1} (\phi )  & := \xi_{\lb{A}} \phi \ind{_{\rb{B}} ^C} \xi_C \, , &
 {}^\g _\xi \Pi^0_0 (\phi ) & := \phi \ind{_A ^B} \xi_B \, , &
 {}^\g _\xi \Pi^1_0 (\phi ) & := \xi_{\lb{A}} \phi \ind{_{\rb{B}} ^C} - \frac{1}{3} \delta \ind*{_{\lb{A}}^C} \phi \ind{_{\rb{B}} ^D} \xi_D \, .
\end{align*}

\paragraph{Intrinsic torsion}
The intrinsic torsion of the $P$-structure at point viewed as an element of the module $\mfW := \mfV \otimes \g/\prb$ was already described in section \ref{sec-alg_intrinsic_torsion}, and its description in terms of the Levi-Civita connection in section \ref{sec-geom-intrinsic-torsion}. We have already noted the slight differences between six dimensions and higher dimensions. These can be more clearly expressed in the present calculus. Thus, denoting by $\nabla_{AB}$ the Levi-Civita connection, the geodetic spinor equation \eqref{eq-foliating_spinor} and the recurrent spinor equation \eqref{eq-recurrent_spinor} read as
\begin{align}
 \left( \xi_D \nabla \ind{^{DA}} \xi_{\lb{B}} \right) \xi _{\rb{C}} & = 0 \, , \label{eq-foliating_spinor6} \\
 \left( \nabla \ind{^{AB}} \xi_{\lb{C}} \right) \xi _{\rb{D}} & = 0 \, , \label{eq-recurrent_spinor6}
\end{align}
Taking the irreducible parts of these equations yield
\begin{align}
 \xi_A \nabla \ind{^{AB}} \xi_B  & = 0 \, ,  \label{eq-skew-foliating6} \\
  \left( \xi_D \nabla \ind{^{DA}} \xi_{\lb{B}} \right) \xi _{\rb{C}} - \frac{1}{3} \left( \xi_D \nabla \ind{^{DE}} \xi_E  \right) \delta \ind*{^A_{\lb{B}}}\xi _{\rb{C}} & = 0 \, ,  \label{eq-sym-foliating6} \\
 \xi_A \nabla \ind{^{BC}} \xi_C + \xi_C \nabla \ind{^{CB}} \xi_A  & = 0 \label{eq-proj-Dirac6} \, , \\
 \left( \nabla \ind{^{AB}} \xi_{\lb{C}} \right) \xi _{\rb{D}} - \left( \nabla \ind{^{\lb{A}|E}} \xi_E  \right) \delta \ind*{^{|\rb{B}}_{\lb{C}}}\xi \ind{_{\rb{D}}} - \left( \xi_E  \nabla \ind{^{\lb{A}|E}} \xi _{\lb{C}} \right) \delta \ind*{^{|\rb{B}}_{\rb{D}}} - \frac{1}{3} \left(\xi_E \nabla \ind{^{EF}} \xi_F  \right) \delta \ind*{_{\lb{A}}^C} \delta \ind*{_{\rb{B}}^D} & = 0 \, , \label{eq-proj-twistor6} 
\end{align}
which are equivalent to \eqref{eq-skew-foliating}, \eqref{eq-sym-foliating} \eqref{eq-proj-Dirac} and \eqref{eq-proj-twistor6a} respectively.

\paragraph{Proof of Proposition \ref{prop-foliating_twistor_spinor6}}
Consider a conformal Killing spinor $\xi_A$ on $(\mcM,g)$, i.e.\ a solution of
\begin{align*}
 \nabla \ind{^{AB}} \xi \ind{_C} + \frac{2}{3} \delta \ind*{^{[A}_{C}} \nabla \ind{^{B]E}} \xi \ind{_E} & = 0 \, .
\end{align*}
A little algebra yields
\begin{align*}
 \left( \nabla \ind{^{AB}} \xi \ind{_{[C}} \right) \xi \ind*{_{D]}} - \frac{2}{3} \xi \ind{_{[C}} \delta \ind*{_{D]}^{[A}} \nabla \ind{^{B]E}} \xi \ind{_E} & = 0 \, , \\
\xi \ind{_E} \left( \nabla \ind{^{E[A}} \xi \ind{_{[C}} \right) \delta \ind*{_{D]}^{B]}} - \frac{1}{3} \xi \ind{_{[C}} \delta \ind*{_{D]}^{[A}} \left( \nabla \ind{^{B]E}} \xi \ind{_E} \right) - \frac{1}{3} \delta \ind*{_{[C}^A} \delta \ind*{_{D]}^B} \left( \xi \ind{_E} \nabla \ind{^{EF}} \xi \ind{_F} \right) & = 0 \, ,
\end{align*}
from which we deduce that $\xi \ind{_A}$ satisfies equations \eqref{eq-proj-twistor6} and \eqref{eq-sym-foliating6}. This proves Proposition \ref{prop-foliating_twistor_spinor6}.

\paragraph{Curvature tensors}
Finally, we record the maps given in appendix \ref{sec-projection} characterising the $P$-submodules of the spaces of curvature tensors in this spinor calculus:
\begin{itemize}
\item for $\mfF \cong (\wedge^2 \mfS ) \odot (\wedge^2 \mfS )$,
\begin{align*}
{}^\mfF _\xi \Pi^0_{-1} ( \Phi ) & := \xi_{\lb{A}} \Phi  \ind{_{B \rb{C} \lb{D} E}} \xi_{\rb{F}} \, , &
 {}^\mfF _\xi \Pi^0_0 ( \Phi ) & := \xi_{\lb{A}} \Phi  \ind{_{B \rb{C} D E}} \, ,
\end{align*}
\item for $\mfA \cong (\wedge^2 \mfS ) \odot \mfS \otimes_\circ \mfS^*$,
\begin{align*}
 {}^\mfA _\xi \Pi^0_{-\frac{3}{2}} ( A ) & := \xi_{\lb{A}} A \ind{_{B \rb{C} \lb{D}} ^F} \xi_{\rb{E}} \xi_F \\
 {}^\mfA _\xi \Pi^0_{-\frac{1}{2}} ( A ) & := \xi_{\lb{A}} A \ind{_{B \rb{C} D} ^E} \xi_E \, , \\
 {}^\mfA _\xi \Pi^1_{-\frac{1}{2}} (A) & := A \ind{_{AB \lb{C}}^E} \xi \ind{_{\rb{D}}} \xi_E + A \ind{_{CD \lb{A}}^E} \xi \ind{_{\rb{B}}} \xi_E  \, , \\
 {}^\mfA _\xi \Pi^2_{-\frac{1}{2}} ( A ) & := \xi_{\lb{A}} A \ind{_{B \rb{C} \lb{D}} ^F} \xi_{\rb{E}} - \frac{1}{4} \delta \ind*{_{\lb{A}}^F} A \ind{_{B \rb{C} \lb{D}} ^G} \xi_{\rb{E}} \xi_G - \frac{1}{4} \xi \ind{_{\lb{A}}} A \ind{_{B \rb{C} \lb{D}} ^G} \delta \ind*{_{\rb{E}}^F} \xi_G \, , \\
 {}^\mfA _\xi \Pi^0_{\frac{1}{2}} ( A ) & := A \ind{_{A B C} ^D} \xi_D \, , \\
 {}^\mfA _\xi \Pi^1_{\frac{1}{2}} ( A ) & := \xi_{\lb{A}} A \ind{_{B \rb{C} D} ^E} - \frac{1}{2} \delta \ind*{_{\lb{A}}^E} A \ind{_{B \rb{C} D} ^F} \xi_F \, , \\
 {}^\mfA _\xi \Pi^2_{\frac{1}{2}} (A)  & := A \ind{_{AB \lb{C}}^E} \xi \ind{_{\rb{D}}} + A \ind{_{CD \lb{A}}^E} \xi \ind{_{\rb{B}}} - \frac{2}{5} \left( A \ind{_{AB \lb{C}}^F} \delta \ind*{_{\rb{D}}^E} \xi_F + A \ind{_{CD \lb{A}}^F} \delta \ind*{_{\rb{B}}^E} \xi_F \right) \, ,
\end{align*}
\item for $\mfC \cong (\odot^2 \mfS ) \otimes_\circ (\odot^2 \mfS^*)$,
\begin{align*} 
{}^\mfC _\xi \Pi_{-2}^0 ( C ) & := \xi \ind{_{\lb{A}}}C \ind*{_{\rb{B}\lb{C}}^{EF}}  \xi \ind{_{\rb{D}}} \xi \ind{_{E}} \xi \ind{_{F}} \, ,
\\
{}^\mfC _\xi \Pi_{-1}^0 ( C ) & := \xi \ind{_{\lb{A}}} C \ind*{_{\rb{B} C}^{DE}} \xi \ind{_{D}} \xi \ind{_{E}}  \, ,
\\
{}^\mfC _\xi \Pi_{-1}^1 ( C ) & := \xi \ind{_{\lb{A}}}C \ind*{_{\rb{B}\lb{C}}^{EF}}  \xi \ind{_{\rb{D}}} \xi \ind{_{F}} - \frac{1}{4} \delta \ind*{_{\lb{A}} ^E} C \ind*{_{\rb{B}\lb{C}}^{FG}} \xi \ind{_{\rb{D}}} \xi \ind{_{F}} \xi \ind{_{G}} - \frac{1}{4} \delta \ind*{_{\lb{C}} ^E} C \ind*{_{\rb{D}\lb{A}}^{FG}} \xi \ind{_{\rb{B}}} \xi \ind{_{F}} \xi \ind{_{G}}  \, ,
\\
{}^\mfC _\xi \Pi_{0}^0 ( C ) & := C \ind*{_{AB}^{CD}} \xi \ind{_{C}} \xi \ind{_{D}} \, ,
\\
{}^\mfC _\xi \Pi_{0}^1 ( C ) & := \xi \ind{_{\lb{A}}} C \ind*{_{\rb{B}C}^{DE}}  \xi \ind{_{E}} - \frac{1}{3} \delta \ind*{_{\lb{A}}^D} C \ind*{_{\rb{B} C}^{EF}}  \xi \ind{_{E}} \xi \ind{_{F}} \, ,
\\
{}^\mfC _\xi \Pi_{0}^3 ( C ) & := \xi \ind{_{\lb{A}}}C \ind*{_{\rb{B}\lb{C}}^{EF}}  \xi \ind{_{\rb{D}}} - \frac{2}{5} \delta \ind*{_{\lb{A}}^{\lp{E}}} C \ind*{_{\rb{B} \lb{C}}^{\rp{F} G}} \xi \ind{_{\rb{D}}} \xi \ind{_{G}} - \frac{2}{5} \delta \ind*{_{\lb{C}}^{\lp{E}}} C \ind*{_{\rb{D} \lb{A}}^{\rp{F} G}} \xi \ind{_{\rb{B}}} \xi \ind{_{G}} + \frac{1}{10}  \delta \ind*{_{\lb{A}} ^{\lp{E}|}} C \ind*{_{\rb{B}\lb{C}}^{GH}} \delta \ind*{_{\rb{D}}^{|\rp{F}}} \xi \ind{_{G}} \xi \ind{_{H}}  \, ,
\\
{}^\mfC _\xi \Pi_{1}^0 ( C ) & := C \ind*{_{AB}^{CD}} \xi \ind{_{D}} \, ,
\\
{}^\mfC _\xi \Pi_{1}^1 ( C ) & := \xi \ind{_{\lb{A}}} C \ind*{_{\rb{B}C}^{DE}} - \frac{1}{2} \delta \ind*{_{\lb{A}}^{\lp{D}}} C \ind*{_{\rb{B}C}^{\rp{E}F}} \xi \ind{_{F}}  \, .
\end{align*}
\end{itemize}
\begin{rem}
 `Coarser' versions of some of the maps ${}^\mfC _\xi \Pi_i^j$ were already given by Jeffryes \cites{Jeffryes1995,Mason1995} in his investigation of the `principal spinors' of the Weyl tensor in dimension six. The maps ${}^\mfC _\xi \Pi_i^j$ are saturated with symmetries, and are thus more tightly connected to the representation theory of $P$ on $\mfC$.
\end{rem}

\section{Conformal structures}\label{sec-conformal}
We collect a few facts and conventions pertaining to conformal geometry. We roughly follow \cite{Bailey1994}, although our staggering of indices differs from theirs. For specificity, we work in the holomorphic category.

A \emph{holomorphic conformal structure} on a complex manifold $\mcM$ is an equivalence class of holomorphic metrics $[g_{ab}]$ on $\mcM$, whereby two metrics $\hat{g}_{ab}$ and $g_{ab}$ belong to the same class if and only if
\begin{align}\label{eq-conformal_change_metric}
\hat{g} \ind{_{ab}} & = \Omega^2 g \ind{_{ab}}  \, ,
\end{align}
for some non-vanishing holomorphic function $\Omega$ on $\mcM$. The respective Levi-Civita connections $\nabla_a$ and $\hat{\nabla}_a$ of $g_{ab}$ and $\hat{g}_{ab}$ are then related by
\begin{align*}
 \hat{\nabla} \ind{_a} V \ind{^b} & = \nabla \ind{_a} V \ind{^b} + Q \ind{_{ac}^b} V \ind{^c} \, , &
Q \ind{_{abc}} := Q \ind{_{ab}^d} g \ind{_{dc}} = 2 \Upsilon \ind{_{\lp{a}}} g \ind{_{\rp{b}c}} - \Upsilon \ind{_c} g \ind{_{ab}} \, ,
\end{align*}
for any holomorphic vector field $V^a$, where $\Upsilon _a := \Omega^{-1} \nabla _a \Omega$.

\paragraph{Spinor bundles}
We first note that under a rescaling \eqref{eq-conformal_change_metric}, the $\gamma$-matrices can be chosen to transform as
\begin{align*}
 \gamma \ind{_a_A^{B'}} & \mapsto \hat{\gamma} \ind{_a_A^{B'}} = \Omega \gamma \ind{_a_A^{B'}} \, , &
 \gamma \ind{_a_{B'}^A} & \mapsto \hat{\gamma} \ind{_a_{B'}^A} = \Omega \gamma \ind{_a_{B'}^A} \, ,
\end{align*}
where $\hat{\gamma} \ind{_a_A^{B'}}$ and $\hat{\gamma} \ind{_a_{B'}^A}$ denote the $\gamma$-matrices for the metric $\hat{g}_{ab}$. In addition, we can choose the $\Spin(2m,\C)$-invariant bilinear forms on $\mcS$ to rescale with a conformal weight of $1$, and their dual with a conformal weight of $-1$. For instance, $\gamma _{A'B'} \mapsto \hat{\gamma}_{A'B'} = \Omega \gamma_{A'B'}$ when $m$ is even, $\gamma^{A'B} \mapsto \hat{\gamma}^{A'B} = \Omega^{-1} \gamma^{A'B}$ when $m$ is odd, and so on. This means in particular that the quantities $\gamma \ind{_a^{AB'}}$ and $\gamma \ind{^a_{AB'}}$ when $m$ is even, and $\gamma \ind{_a^{A'B'}}$ and $\gamma \ind{^a_{A'B'}}$, and their unprimed counterparts, when $m$ is odd, have conformal weight $0$. Then the spin connection $\hat{\nabla}_a$ is related to $\nabla_a$ by
\begin{align} \label{eq-conformal_spin_connection2}
 \hat{\nabla} \ind{_a} \xi \ind*{^{B'}} & = \nabla \ind{_a} \xi \ind*{^{B'}} - \frac{1}{2} \Upsilon \ind{_b} \gamma \ind{^b_{C'}^D} \gamma \ind{_a_D^{B'}} \xi \ind*{^{C'}} \, ,  
\end{align}
for any holomorphic spinor field $\xi^{A'}$, and similarly for unprimed and dual spinors. This connection can be seen to preserve the hatted $\gamma$-matrices \emph{and} the hatted bilinear forms on $\mcS$. This agrees with the convention of \cite{Penrose1984} but differs from the more standard convention, used in \cite{Lawson1989} for instance.

Now assuming that $\xi^{A'}$ is pure, and setting $\hat{\xi} \ind*{_a^A} := \xi \ind*{^{B'}} \hat{\gamma} \ind{_a _{B'}^A}$, we derive further
\begin{align*}
 \left( \hat{\nabla} \ind{_a} \hat{\xi} \ind*{^b^B} \right) \hat{\xi} \ind*{_b^C} & = \left( \nabla \ind{_a} \xi \ind*{^b^B} \right) \xi \ind*{_b^C} - 2 \Upsilon \ind{_b} \xi \ind*{^b^{\lb{B}}} \xi \ind*{_a^{\rb{C}}} \, , \\
\xi \ind*{^{A'}} \hat{\nabla} \ind{_b} \hat{\xi} \ind*{^b^{B}} - \hat{\xi} \ind*{^b^B} \hat{\nabla} \ind{_b} \xi \ind*{^{A'}} & = \Omega^{-1} \left( \xi \ind*{^{A'}} \nabla \ind{_b} \xi \ind*{^b^{B}} - \xi \ind*{^b^B} \nabla \ind{_b} \xi \ind*{^{A'}} + (m-1) \Upsilon \ind{_a} \xi \ind*{^a^B} \xi \ind*{^{A'}} \right) \, , \\
 \left( \hat{\xi} \ind*{^b^A} \hat{\nabla} \ind{_a} \hat{\xi} \ind*{^b^B} \right) \hat{\xi} \ind*{_b^C} & = \Omega^{-1} \left( \xi \ind*{^a^A} \nabla \ind{_a} \xi \ind*{^b^B} \right) \xi \ind*{_b^C} \, .
\end{align*}
The first two equations can be combined to yield
\begin{align*}
 ( \hat{\nabla} \ind{_a} \hat{\xi} \ind*{^b^B} ) \hat{\xi} \ind*{_b^C} + \frac{2}{m-1} \left( \hat{\xi} \ind*{_a^{\lb{B}}} \hat{\nabla} \ind{_b} \hat{\xi} \ind*{^b^{\rb{C}}} + \hat{\xi} \ind*{^b^{\lb{B}}} \hat{\nabla} \ind{_b} \hat{\xi} \ind*{_a^{\rb{C}}} \right) & = ( \nabla \ind{_a} \xi \ind*{^b^B} ) \xi \ind*{_b^C} + \frac{2}{m-1} \left( \xi \ind*{_a^{\lb{B}}} \nabla \ind{_b} \xi \ind*{^b^{\rb{C}}} + \xi \ind*{^b^{\lb{B}}} \nabla \ind{_b} \xi \ind*{_a^{\rb{C}}} \right) \, .
\end{align*}

\paragraph{Curvature}
In conformal geometry, it is more convenient to use the alternative decomposition to \eqref{eq-Riem_decomposition}
\begin{align} \label{eq-Riem_decomposition-conformal}
 R \ind{_{abcd}} & = C \ind{_{abcd}} - 4 g_{\lb{c}|\lb{a}} \Rho_{\rb{b}|\rb{d}} \, , &
 \Rho_{ab} & := \frac{1}{2-n} \Phi_{ab} - R \frac{1}{2n(n-1)} g_{ab} \, .
\end{align}
where the Weyl tensor $C \ind{_{abc}^d}$ is conformally invariant, and the \emph{Schouten} or \emph{Rho tensor} $\Rho_{ab}$ transforms as
\begin{align}\label{eq-conf-transf-Rho}
 \hat{\Rho}_{ab} & = \Rho_{ab} - \nabla_a \Upsilon_b + \Upsilon_a \Upsilon_b - \frac{1}{2} \Upsilon_c \Upsilon^c g_{ab} \, , &
 \hat{\Rho} & = \Omega^{-2} \left( \Rho - \nabla^c \Upsilon_c - \frac{n-2}{2} \Upsilon^c \Upsilon_c \right) \, ,
\end{align}
where $\Rho := \Rho \ind{_a^a}$. Finally, the Cotton-York tensor $A_{abc} := 2 \nabla \ind{_{\lb{b}}} \Rho \ind{_{\rb{c}a}} = -(n-3) \nabla^d C_{dabc}$, where the expression on the RHS follows from the contracted Bianchi identity, transforms as $\hat{A}_{abc} = A \ind{_{abc}} - \Upsilon^d C_{dabc}$.

\bibliography{biblio}

\end{document}